\documentclass[12pt]{amsart}
\usepackage{amssymb,amsfonts,amsthm,amscd,stmaryrd,dsfont,esint,upgreek,constants,todonotes,mathtools}
\usepackage[mathcal]{euscript}

\usepackage[top=3cm,bottom=3cm, left=3cm,right=3cm]{geometry}

\theoremstyle{plain}
\newtheorem{thm}{Theorem}
\newtheorem{cor}{Corollary}
\newtheorem{lem}[cor]{Lemma}
\newtheorem{prop}[cor]{Proposition}

\theoremstyle{definition}

\numberwithin{cor}{section}
\numberwithin{equation}{section}

\newconstantfamily{C}{symbol=C}
\newconstantfamily{c}{symbol=c}
\newconstantfamily{O}{symbol=\Omega}

\usepackage{color} 

\definecolor{orange}{RGB}{255,127,0}

\newcommand{\R}{\mathbb{R}}

\newcommand{\Z}{\mathbb{Z}}
\newcommand{\N}{\mathbb{N}}

\renewcommand{\d}{d}
\newcommand{\Rd}{\mathbb{R}^\d}
\newcommand{\Sd}{\mathbb{S}^\d}

\newcommand{\ep}{\varepsilon}
\newcommand{\barH}{\overline{H}}
\newcommand{\E}{\mathbb{E}}
\renewcommand{\P}{\mathbb{P}}
\newcommand{\Prob}{\P}

\newcommand{\indc}{\mathds{1}}

\newcommand{\SL}{\mathcal{S}}

\renewcommand{\bar}{\overline}
\renewcommand{\tilde}{\widetilde}
\renewcommand{\hat}{\widehat}

\renewcommand{\Re}{\mathcal{R}}
\newcommand{\Se}{\mathcal{S}}

\newcommand{\F}{\mathcal{F}}
\newcommand{\G}{\mathcal{G}}

\newcommand{\be}{\begin{equation}}
\newcommand{\ee}{\end{equation}}

\DeclareMathOperator{\tr}{tr}
\DeclareMathOperator{\intr}{int}

\DeclareMathOperator{\USC}{USC}

\DeclareMathOperator{\diam}{diam}
\DeclareMathOperator{\dist}{dist}
\DeclareMathOperator{\conv}{conv}
\DeclareMathOperator*{\esssup}{ess\,sup}
\DeclareMathOperator*{\essinf}{ess\,inf}

\begin{document}

\title[Quantitative stochastic homogenization of viscous HJ equations]{Quantitative stochastic homogenization of viscous Hamilton-Jacobi equations}

\author[S. N. Armstrong]{Scott N. Armstrong}
\address{Ceremade (UMR CNRS 7534) \\ Universit\'e Paris-Dauphine \\ Place du Mar\'echal De Lattre De Tassigny \\ 75775 Paris CEDEX 16, France}
\email{armstrong@ceremade.dauphine.fr}

\author[P. Cardaliaguet]{Pierre Cardaliaguet}
\address{Ceremade (UMR CNRS 7534) \\ Universit\'e Paris-Dauphine \\ Place du Mar\'echal De Lattre De Tassigny \\ 75775 Paris CEDEX 16, France}
\email{cardaliaguet@ceremade.dauphine.fr}


\date{\today}

\begin{abstract} 
We prove explicit estimates for the error in random homogenization of degenerate, second-order Hamilton-Jacobi equations, assuming the coefficients satisfy a finite range of dependence. In particular, we obtain an algebraic rate of convergence with overwhelming probability under certain structural conditions on the Hamiltonian.
\end{abstract}

\maketitle

\section{Introduction} \label{I}

\subsection{Motivation and informal statement of results}
The paper is concerned with second-order (i.e., ``viscous") Hamilton-Jacobi equations of the form 
\begin{equation}
\label{e.VHJ}
u^\ep_t -\ep \tr\left(A\left(\frac x\ep \right)D^2u^\ep\right)+ H\left(Du^\ep, \frac x\ep\right) = 0 \quad \mbox{in} \ \Rd \times (0,\infty).
\end{equation}
Here $\ep > 0$ is a small parameter which we will send to zero, the Hamiltonian $H=H(p,y)$ is convex and coercive in $p$ and $A(y)$ is a diffusion matrix which is possibly degenerate. In addition, the coefficients $H$ and $A$ are \emph{random fields}, that is, sampled by a probability measure on the space of all such coefficients. 

\smallskip

The basic \emph{qualitative} result concerning the stochastic homogenization of~\eqref{e.VHJ} states roughly that, under appropriate assumptions on the probability measure~$\P$ (typically that $\P$ is \emph{stationary} and \emph{ergodic} with respect to translations on $\Rd$), there exists a deterministic, convex \emph{effective Hamiltonian}~$\overline H$ such that the solutions $u^\ep$ of~\eqref{e.VHJ} converge uniformly as $\ep \to 0$, with probability one, to the solution of the effective equation
\begin{equation}\label{HJh}
u_t + \overline H(Du) = 0 \quad \mbox{in} \ \Rd \times (0,\infty).
\end{equation}
A formulation of this result was first proved independently by Souganidis~\cite{S} and Rezakhanlou and Tarver~\cite{RT} for first order equations Hamilton-Jacobi equations and later extended to the viscous setting by Lions and Souganidis \cite{LS2} and by Kosygina, Rezakhanlou and Varadhan \cite{KRV}. Generalizations as well as new proofs of these results, using methods that are closer to our perspective in this paper, appeared in~\cite{ASo1,ASo3,AT2}. 

\smallskip

The first complete quantitative homogenization results for Hamilton-Jacobi equations in the stochastic setting were presented in~\cite{ACS}. Previous results in this direction were obtained by Rezakhanlou~\cite{R}, who gave structural conditions on Hamilton-Jacobi equations in dimension $\d=1$ in which a central limit theorem holds, and by Matic and Nolen \cite{MN}, who proved a estimate for the random part of the error for a particular class of equations. In each of these papers, the main assumption quantifying the ergodicity of~$\P$ is a finite range dependence hypothesis in the spatial dependence of the coefficients: this roughly means that there is a length scale $\ell>0$ such that the random elements $(A(x ),H(p,x ))$ and $(A(y ),H(q,y ))$ are independent whenever $|x-y|\geq \ell$. 

Each of the quantitative homogenization results mentioned in the previous paragraph concerned first-order equations, that is, the case that $A\equiv 0$. Working with a general degenerate diffusion matrix $A$ is much more difficult from the point of view of homogenization. On the one hand, even if $A$ is uniformly elliptic there is no useful regularizing effect, since we send $\ep\to 0$. On the other hand, the presence of $A$ makes the dependence of the solutions on the coefficients much more complicated, as it destroys the finite speed of propagation property possessed by first-order Hamilton-Jacobi equations. The latter property is essentially what ``localizes" the problem and allows one to apply the independence property, and it is this idea that played a crucial role in~\cite{ACS}. 

\smallskip

In this paper, we develop a quantitative theory of stochastic homogenization for~\eqref{e.VHJ} with a general degenerate diffusion matrix $A$. In particular, we prove an algebraic convergence rate, with overwhelming probability, as $\ep\to 0$. Roughly, we show that there exists an exponent $\alpha\in (0,1)$, which is given explicitly, such that
$$\Prob\left[ \sup_{(x,t)\in B_T\times [0,T]} |u^\ep(x,t )-u(x,t)|\geq \ep^\alpha\right] \lesssim \exp\left( -\ep^{-\alpha} \right). $$
As in~\cite{ACS}, our arguments are inspired by probability techniques developed in first passage percolation by Kesten~\cite{K2},  Alexander~\cite{A} and Zhang \cite{Z} and in the study of the fluctuations of the Lyapunov exponents for Brownian motion in Poissonian potentials, developed by Sznitman~\cite{SzEE} and by W{\"u}thrich~\cite{W}. In fact, perhaps the closest previous work to ours is that of Sznitman~\cite{SzEE}, who proved a special case of our main result, Theorem~\ref{mpEE}. See the discussion following the statement of Theorem~\ref{mpEE}, below. 

\smallskip

Viscous Hamilton-Jacobi equations like~\eqref{e.VHJ} with random coefficients arise in particular in the study of large deviations of diffusions in random environments. Although we do not explore this topic in detail here, the main results in this paper imply quantitative bounds on the rate of the large deviations principle.

\subsection{Hypotheses}
\label{ss.hypos}
We proceed with the precise formulation of our results, beginning with the definition of the probability space $\Omega$, which is taken to be ``the set of all coefficients" for the equation~\eqref{e.VHJ} subject to some particular structural conditions. We then define a random environment by putting a probability measure on this set. Throughout the paper, we fix the parameters $q > 1$, $n\in \N^*$ and $\Lambda \geq1$. 

\smallskip

We consider diffusion matrices $A:\Rd\to \Sd$ (here $\Sd$ denotes the set of $d$-by-$d$ real symmetric matrices) which have Lipschitz continuous square roots, that is, we assume
\begin{equation} \label{e.Adef}
A = \frac12\Sigma^t \Sigma, \qquad \mbox{($\Sigma^t:=$ \ the transpose of $\Sigma$)}
\end{equation}
for a matrix-valued function $\Sigma :\Rd \to \R^{n\times d}$ which satisfies, for every $y,z\in \Rd$,
\begin{equation}\label{e.sigbnd}
\frac12\left| \Sigma(y)\right|^2 \leq \Lambda
\end{equation}
and
\begin{equation}\label{e.siglip}
\left| \Sigma(y) - \Sigma(z) \right| \leq \Lambda |y-z|.
\end{equation}
The Hamiltonian $H:\Rd \times \Rd \to \R$ is required to satisfy the following: for every $y\in \Rd$,
\begin{equation}\label{e.Hconvex}
p \mapsto H(p,y) \quad \mbox{is convex.}
\end{equation}
For every  $p, \tilde p\in\Rd$ and $y ,z \in \Rd$,
\begin{equation}\label{e.Hsubq}
\frac{1}{\Lambda} |p|^{q} -\Lambda \leq H(p,y) \leq \Lambda |p|^q + \Lambda,
\end{equation}
\begin{equation}\label{e.HsubqLip}
\left| H(p,y) - H(p,z) \right| \leq \Lambda \big( |p|^q +1  \big) |y-z|, 
\end{equation}
and
\begin{equation}\label{e.HsubqDp}
\left| H(p,y) - H(\tilde p,y) \right| \leq \Lambda \big( |p| + |\tilde p| + 1 \big)^{q-1} |p-\tilde p|. 
\end{equation}
Finally, we require the following additional structural condition: for every $p,y\in \Rd$, 
\begin{equation} \label{e.imposition}
H(p,y) \geq H(0,y) \quad \mbox{and} \quad \sup_{z\in \Rd} H(0,z) = 0.
\end{equation}
We denote by $\Omega$ the set of ordered pairs $(\Sigma,H)$ of functions as described above:
\begin{equation*} \label{}
\Omega:= \big\{ (\Sigma,H) \,:\, \mbox{$\Sigma$ and $H$ satisfy~\eqref{e.sigbnd},~\eqref{e.siglip}~\eqref{e.Hconvex},~\eqref{e.Hsubq},~\eqref{e.HsubqLip},~\eqref{e.HsubqDp} and~\eqref{e.imposition}}\big\}.
\end{equation*}
Throughout, we take $A$ to be the $\Sd$--valued random field on $\Omega$ defined by~\eqref{e.Adef}. We endow~$\Omega$ with the following $\sigma$--algebras. For each Borel set $U \subseteq \Rd$, we define  $\F(U)$ by:
\begin{multline*} \label{}
\F(U):= \mbox{$\sigma$--algebra generated by the maps $(\Sigma,H)\mapsto \Sigma(y)$ and $(\Sigma,H) \mapsto H(p,y)$,} \\ \mbox{with $p\in \Rd$ and $y\in U$.}
\end{multline*}
The largest of these $\sigma$--algebras is denoted by $\F:= \F(\Rd)$.

\smallskip

We model the random environment by a given probability measure $\P$ on $(\Omega,\F)$, which is assumed to be \emph{stationary} with respect to translations and have a \emph{finite range of dependence}. The expectation with respect to $\P$ is denoted by $\E$. To state the stationary assumption, we let $\{ \tau_z \}_{z\in\Rd}$ be the group action of translation on $\Omega$, that is, for each $z\in \Rd$, we define the map $\tau_z:\Omega\to \Omega$ by
\begin{equation*} \label{}
\tau_z (\Sigma,H) := \left(\tau_z\Sigma, \tau_z H\right), \quad \mbox{with} \ \  (\tau_z\Sigma)(y) := \Sigma(y+z) \ \  \mbox{and} \ \  (\tau_zH)(p,y) := H(p,y+z).
\end{equation*}
We extend the action of translation to $\F$ by defining $\tau_z F$, for each $F \in \F$, by 
\begin{equation*} \label{}
\tau_zF:= \big\{ \tau_z\omega\,:\ \omega \in F\big\}.
\end{equation*}
With the above notation, the hypotheses on $\P$ are that 
\begin{equation} \label{e.stat}
\mbox{for every $y\in \Rd$ and $F \in \F$,} \qquad \P\left[ \tau_yF\right] = \P\left[ F \right]
\end{equation}
and 
\begin{multline} \label{e.frd}
\mbox{for all Borel sets $U,V \subseteq \Rd$ \ such that \ $\dist(U,V) \geq 1$,} \\ \mbox{$\F(U)$ and $\F(V)$ are $\P$--independent.}
\end{multline}
Here $\dist(U,V) := \inf_{x\in U, \, y\in V} |x-y|$. The first assumption~\eqref{e.stat} states that the statistics of the coefficients are the same everywhere in $\Rd$. The second assumption~\eqref{e.frd} is a quantitative ergodicity assumption, which requires independence, in well-separated regions of space, of coefficients sampled by~$\P$.

\smallskip
 
If $X$ is a random variable, then we may write $X(\omega)$ if we wish to display its dependence on $\omega \in\Omega$ explicitly. However, we avoid this unless it is required for clarity. For instance, we usually write $u^\ep(x,t)$ for the solution of~\eqref{e.VHJ} (e.g., as in the statement of Theorem~\ref{p.cvuep}), but we sometimes use the notation $u^\ep(x,t,\omega)$ for the same quantity. 

\smallskip

With the important exception of~\eqref{e.imposition}, the structural hypotheses on the coefficients comprise a fairly general model of a viscous Hamilton-Jacobi equation. On the other hand,~\eqref{e.imposition} is rather restrictive: e.g., it is not satisfied by a Hamiltonians such as $H(p,y) = |p|^2 - b(y) \cdot p$ where $b:\Rd \to \Rd$ is a random vector field. It would be very interesting if quantitative homogenization results like the ones presented here could be proved without~\eqref{e.imposition}. We expect this to be quite difficult and that the statement of any such results would likely be weaker. Indeed, and as we will see, the hypothesis plays an essential role in the analysis of the metric problem by forcing some monotonicity in the level sets of the solutions. Without such monotonicity, the problem becomes much more complicated. Note also that~\eqref{e.imposition} is natural from the control theoretic viewpoint: see the discussion in~\cite{ACS} for the first-order case, where the same assumption is made.
 
\subsection{Quantitative homogenization of the metric problem}
\label{ss.mpEE}

As shown in~\cite{ASo3,AT2}, the heart of the modern theory of stochastic homogenization of Hamilton-Jacobi equations is the so-called \emph{metric problem}. For $\mu > 0$ and $x\in \Rd$, we consider the unique solutions $m_\mu(\cdot,x) \in C^{0,1}_{\mathrm{loc}} (\Rd)$ of 
\begin{equation}
\label{e.met}
\left\{ \begin{aligned}
& -\tr\left( A(y) D^2m_\mu(\cdot,x) \right) +H(Dm_\mu(\cdot,x),y) = \mu & \mbox{in} & \ \Rd\setminus \overline B_1(x), \\
& m_\mu(\cdot,x) = 0 & \mbox{on} & \ \overline B_1(x). 
\end{aligned} \right.
\end{equation}
The problem~\eqref{e.met} inherits its name from the subadditivity of its solutions (which also explains its usefulness) and the fact that $m_\mu(y,z)$ has a probabilistic interpretation roughly as a cost of transporting a particle from $y$ to $z$, under a certain cost functional corresponding to the coefficients. It can be shown in more generality than we consider here (see~\cite{AT2}) that the solutions~$m_\mu$ of~\eqref{e.met} have deterministic limits: for every $\mu > 0$, there exists a positively homogeneous, convex function $\overline m_\mu\in C(\Rd)$ such that 
\begin{equation} \label{e.methomog}
\P \left[ \sup_{R>0} \, \sup_{y,z\in B_R} \, \lim_{t\to \infty} \, \frac{1}t \left| m_\mu(ty,tz) - \overline m_\mu(ty-tz) \right| = 0 \right] = 1. 
\end{equation}
As it turns out~(c.f.~\cite{AT2}), the limit~\eqref{e.methomog} is strong enough to imply a full homogenization result for viscous Hamilton-Jacobi equations. The effective Hamiltonian $\overline H$ can be expressed in terms of $\overline m_\mu$ by
\begin{equation} \label{e.Hbarmet}
\overline H(p) = \inf \left\{ \mu >0  \, : \, \mbox{for every} \ y\in \Rd, \ \overline m_\mu(y) \geq p\cdot y   \right\}.
\end{equation}
That is, $\overline m_\mu$ describes the $\mu$--sublevel set of $\overline H$ because, as expected, it is the solution of the homogenized metric problem. See the next section for the definition and basic properties of~$m_\mu$.

As~\eqref{e.methomog} is of fundamental importance to the qualitative homogenization theory, quantifying it is a central task  for a quantitative theory. This is our first main result, which is actually the principal result of the paper and contains most of the difficulty in building a quantitative theory of~\eqref{e.VHJ}. The rest of the main results, presented in Theorems~\ref{p.deltaGlobal} and~\ref{p.cvuep} below, quantifies how the metric problem controls the solutions of the approximate cell problem and of~\eqref{e.VHJ}. These are more straightforward to obtain, since they follow from deterministic comparison (i.e., pure PDE) arguments.  

\begin{thm} \label{mpEE} Fix $\mu_0\geq 1$. There exists a constant $C> 0$, depending only on $(d,q,\Lambda,\mu_0)$ such that, for every $0 < \mu \leq \mu_0$, $\lambda > C\mu^{-2}$ and $y\in \R^d$,
\begin{equation}\label{easy}
\Prob\Big[\, m_\mu(y,0) - \overline m_\mu(y) \leq - \lambda \Big] \leq \exp\left(-\frac{\mu^4 \lambda^2}{C |y|} \right),
\end{equation}
and, if 
\begin{equation}\label{lamb-cond}
\lambda \geq C\left(\frac{|y|^{\frac23}}{\mu^{2}} + \frac{|y|^\frac13}{\mu^{4}} \right)  \log\left(2+\frac{|y|}{\mu}\right),
\end{equation}
then
\begin{equation}\label{hard}
\Prob\Big[\, m_\mu(y,0) - \overline m_\mu(y) \geq \lambda \Big] \leq \exp\left(-\frac{\mu^4 \lambda^2}{ C |y|} \right).
\end{equation}
\end{thm}

To see how Theorem~\ref{mpEE} quantifies~\eqref{e.methomog}, 
note that (at least for $\mu>0$) the right sides of~\eqref{easy} and~\eqref{hard} are very small for $\lambda \gg |y|^{1/2}$. The first inequality therefore implies roughly that, for $t \gg 1$, 
\begin{equation*} \label{}
\frac{1}t \left( m_\mu(ty,0) - \overline m_\mu(ty) \right) \ll - t^{-1/2} \quad \mbox{has overwhelmingly small probability.}
\end{equation*}
To interpret the second statement, note that for fixed $\mu>0$ the first term in the parentheses dominates for large $|y|$. We obtain roughly that, for every $\alpha >0$ and $t\gg 1$,
\begin{equation*} \label{}
\frac{1}t \left( m_\mu(ty,0) - \overline m_\mu(ty) \right) \gg  t^{-1/3+\alpha} \quad \mbox{has overwhelmingly small probability.}
\end{equation*}
Note that~\eqref{easy} and~\eqref{hard} both degenerate as $\mu \to 0$. This is necessarily so: see the discussion in~\cite{ACS}, which encounters the same phenomenon in the first-order case. 

\smallskip

The only previous result like Theorem~\ref{mpEE} in the case $A\not\equiv 0$ is due to Sznitman~\cite{SzEE}. The main result of his paper is an error estimate for ``metrics" associated to the long-time behavior of Brownian motion in the presence of truncated Poissonian obstacles. These distance functions, translated into our language (see~\cite{AT2} for the precise translation and more details on this connection), are nothing other than the functions $m_\mu$ in the special case that $A(x) = \frac12 I_d$ (here $I_d$ is the identity matrix) and the Hamiltonian has the form $H(p,y) = \frac12 |p|^2 - V(y)$, with $V$ the truncated Poissonian potential. Compared to~\cite{SzEE}, we are able to consider more general potentials: for example, the potential $V$ in~\cite{SzEE} is generated by radially symmetric bump function, here we do not need such symmetry (but we pay a price with a slightly weakened estimate). The argument in~\cite{SzEE} achieves the localization property due to special properties of Brownian motion and a probabilistic argument which analyzes the sample paths of the diffusion. 

\smallskip

In view of~\eqref{e.Hbarmet}, Theorem~\ref{mpEE} gives explicit error bounds for a numerical method for computing the effective nonlinearity $\overline H$. The task is reduced to the numerical computation of the maximal subsolutions $m_\mu$, and this can be further reduced to a problem on a finite domain by the results in Section~\ref{s.fluctuation}, see in particular Lemma~\ref{l.seige}. 

\smallskip

The proof of Theorem~\ref{mpEE} is broken into two main steps: first, an estimate on the likelihood that the random variable $m_\mu(y,0)$ is far from its mean $\E \left[ m_\mu(y,0) \right]$, which is proved in Section~\ref{s.fluctuation}, and second, an estimate of the difference between $\E \left[ m_\mu(y,0) \right]$ and $\overline m_\mu(y,0)$, which is proved in Section~\ref{s.bias}. The proof of Theorem~\ref{mpEE} itself is given at the beginning of Section~\ref{s.bias}.

\subsection{Quantitative results on the approximate cell problem} 

After proving Theorem~\ref{mpEE}, the rest of the paper is concerned primarily with transferring the error estimates for the metric problem to error estimates for the time-dependent initial-value problem. As an intermediate step, we consider the time-independent \emph{approximate cell problem}: for each $\delta > 0$ and $p\in \Rd$, we consider the unique solution $v^\delta(\cdot,p) \in C^{0,1}(\Rd)$ of
\begin{equation}\label{e.cellp}
\delta v^\delta(y, p) -{\rm tr}\left(A(y )D^2v^\delta(y,p)\right)+H(Dv^\delta(y , p)+p, y )=0 \qquad {\rm in }\; \R^d.
\end{equation}
Results concerning the well-posedness of~\eqref{e.cellp} can be found in~\cite{AT}. This problem arises naturally in the qualitative theory of homogenization of~\eqref{e.VHJ}, and in fact the general homogenization theorem is equivalent to the statement that, for every $p\in\Rd$, 
\begin{equation} \label{e.cellhomog}
\P \left[ \forall p\in \Rd, \ \limsup_{\delta\to 0} \left| \delta v^\delta(0, p)+ \overline H(p) \right| = 0 \right] = 1.
\end{equation}
Our next result, Theorem~\ref{p.deltaGlobal} below, is a quantitative version of~\eqref{e.cellhomog}.

\smallskip

The rate of convergence of the limit in~\eqref{e.cellhomog} cannot be quantified without further assumptions on the law of the Hamiltonian. Indeed, it turns out that the mixing properties of the environment control the rate at which $\delta v^\delta(0,p) + \overline H(p)$ converges to zero \emph{from above} for all $p\in\Rd$, as well as the rate at which this quantity converges to zero from below for $p$'s satisfying $\overline H(p) > 0$. However, the rate at which we can expect $\delta v^\delta(0,p) + \overline H(p)$ to converge to zero for $p\in \intr \{ \overline H=0\}$ is determined by the behavior of $H$ near its maximum: in view of~\eqref{e.imposition}, what is important is how large the probability $\P\left[ H(0,0) > -\delta \right]$ is for $0 < \delta \ll 1$. This phenomenon was already encountered in~\cite{ACS}, where an explicit example was given showing that the rate of the limit may be arbitrarily slow for $p$ belonging to the ``flat spot," defined as the interior of the level set $\{ \overline H=0 \}$.

\smallskip

For simplicity, we quantify~\eqref{e.cellhomog} under an assumption that rules out the existence of such a flat spot, although an inspection of the arguments in Section~\ref{s.cellpb} yield more precise results (e.g., estimates for $p$'s not belonging to the flat spot without this assumption). 

\begin{thm}
\label{p.deltaGlobal}
Assume that the lower bound in~\eqref{e.Hsubq} is replaced by
\begin{equation}\label{e.ConstH}
 H(p,x)\geq \frac{1}{\Lambda}|p|^q.
\end{equation}
Fix $\xi \geq 1$. Then there exists $C>0$, depending on $(d,q,\Lambda,\xi)$, such that, for every $p\in \Rd$ with $|p|\leq \xi$ and $\delta,\lambda>0$ satisfying
$$
\lambda \geq  C\delta^{\frac{1}{7+6q}} \left( 1+  \left|\log \delta \right|\right),
$$
we have
\begin{equation}\label{e.thmdeltaGlobal}
\Prob\Big[\, \left|\delta v^\delta(0 ,p) + \overline H(p) \right|\geq  \lambda \Big] 
\leq  C \lambda^{-\d}\left(\lambda^{-2\d}+\delta^{-2\d}\right)\exp\left(-\frac{\lambda^{4(1+q)}}{C \delta} \right).
\end{equation}
\end{thm}

Notice that Theorem~\ref{p.deltaGlobal} quantifies the limit~\eqref{e.cellhomog} since it implies roughly that, for every $\alpha > 0$ and sufficiently small $\delta >0$,
\begin{equation*} \label{}
\left|\delta v^\delta(0,p) + \overline H(p) \right| \gg  \delta^{\frac1{7+6q}+\alpha} \quad \mbox{has overwhelmingly small probability.}
\end{equation*}
While the exponent $(7+6q)^{-1}$ is far from optimal, this quantifies the convergence with an algebraic rate.

The proof of Theorem~\ref{e.thmdeltaGlobal} (given at the end of~Section~\ref{s.cellpb}) differs substantially from the analogue in~\cite{ACS}. The idea, as in~\cite{ACS}, is to link the metric problem and the approximate cell problem and then apply Theorem~\ref{mpEE}. Howeover, the connection between the two problems relies on comparison and convex geometry arguments recently introduced in~\cite{AT2}, which account for the first proof of qualitative homogenization for~\eqref{e.VHJ} based only on the metric problem. Previous arguments were either for first--order equations~\cite{ASo3}, relied on representation formulae~\cite{KRV,LS2}, or were based on weak convergence techniques~\cite{LS3,ASo1}.

\subsection{Quantitative results on the time dependent problem} 
The third and final result we present in the introduction relates to the time-dependent initial-value problem: we study the convergence rate for the solution $u^\ep=u^\ep(x,t)$ of 
\begin{equation} \label{HJq}
\left\{ \begin{aligned}
& u^\ep_t -\ep \tr \left(A\left(\frac{x}{\ep} \right)D^2u^\ep\right)+ H\left(Du^\ep,\frac x\ep \right) = 0 & \mbox{in} & \ \Rd \times (0,\infty), \\
& u^\ep(x,0)= g(x) \quad \mbox{in} \ \Rd
\end{aligned}\right.
\end{equation}
to the solution $ u$ of the homogenized problem
\begin{equation} \label{HJqhom}
\left\{ \begin{array}{l}
 u_t + \overline H(D u) = 0 \quad \mbox{in} \ \Rd \times (0,\infty),\\
 u(x,0)= g(x) \quad \mbox{in} \ \Rd 
\end{array}\right.
\end{equation}
Here $g\in C^{0,1}(\Rd)$ is given. Results implying the well-posedness of~\eqref{HJq} and~\eqref{HJqhom} can be found in~\cite{AT}.

\begin{thm}[Error estimate for the time dependent problem]\label{p.cvuep}
Assume the hypotheses of Theorem~\ref{e.ConstH}. Let $g\in C^{1,1}(\R^d)$. Then there exists a constant $C>0$, depending only on $(d,q,\Lambda,\| g \|_{C^{1,1}(\Rd)})$ such that, for any $\lambda,\ep>0$ satisfying
$$
\lambda \geq C\ep^{\frac{1}{17+12q}}\left(1+\left| \log \ep \right|\right)
$$
we have 
\begin{equation}\label{e.cvuep}
\Prob\Big[  \, \sup_{x\in B_T, \ t\in [0,T]} \left|u^\ep(x,t )- u(x,t)\right| \geq  \lambda \Big ] \leq  
C(\lambda\ep)^{-3\d}\exp\left(- \frac{\lambda^{\frac{11}{2}+4q}}{C\ep^{\frac12}}\right).
\end{equation}
\end{thm}

Notice that the previous theorem implies roughly that, for every $\alpha > 0$ and sufficiently small $\ep >0$,
\begin{equation*} \label{}
 \left|u^\ep(x,t )- u(x,t)\right| \gg  \ep^{\frac1{17+12q}+\alpha} \quad \mbox{has overwhelmingly small probability.}
\end{equation*}

To prove Theorem~\ref{p.cvuep}, we compare the solutions $v^\delta(\cdot,p)$ of~\eqref{e.cellp} to the solutions of~\eqref{HJq} and then apply the result of Theorem~\ref{p.deltaGlobal}. That is, the proof is a deterministic, pure PDE comparison argument which, while technical, is relatively straightforward. Such arguments appeared first in the context of periodic homogenization of first-order equations in Capuzzo--Dolcetta and Ishii~\cite{ICD}, and a similar technique is used for the proof of \cite[Theorem 4]{ACS}, also for first order HJB equations. Here the techniques required are a bit more complicated because the comparison machinery for viscosity solutions of second-order equations is more involved. The comparison argument is summarized in Lemma~\ref{l.jzfcojjh} and the proof of Theorem~\ref{p.cvuep} is given at the end of Section~\ref{s.timedep}.

\subsection{Outline of the paper} Section~\ref{Pre} is a summary of the basic (deterministic) properties of the metric problem. The proof of Theorem~\ref{mpEE} is split between an estimate of the random error (Section~\ref{s.fluctuation}) and an estimate of the non-random error (Section~\ref{s.bias}). Theorem~\ref{p.deltaGlobal} is proved in Section~\ref{s.cellpb} and Theorem~\ref{p.cvuep} in Section~\ref{s.timedep}.

\section{Basic properties of the metric problem} \label{Pre}
We review some basic facts, needed throughout the paper, concerning the equation
\begin{equation} \label{e.metbas}
-\tr\left( A(y ) D^2w \right) + H(Dw,y ) = \mu.
\end{equation}
Proofs of most of the results collected here can be found in~\cite{AT}. As the statements are deterministic, throughout this section we fix $(\Sigma,H)\in \Omega$.

The basic regularity result for coercive Hamilton-Jacobi equations is the interior Lipschitz continuity of solutions. For a proof of the following proposition with an explicit Lipschitz constant depending on the given parameters, see~\cite[Theorem 3.1]{AT}. 

\begin{prop}
\label{p.lipschitz}
Suppose that $\mu \geq 0$ and $u \in C(B_1)$ is a solution of~\eqref{e.metbas} in $B_1$. Then there exists $L_\mu>0$, depending on $(q,\Lambda)$ and an upper bound for $\mu$, such that 
\begin{equation*} \label{}
\sup_{x,y\in B_{1/2}, \, x\neq y}\, \frac{|u(x)-u(y)|}{|x-y|} \leq L_\mu. 
\end{equation*}
\end{prop}

We next review some properties of the metric problem. The solutions of~\eqref{e.met} can be characterized as the \emph{maximal subsolutions} of~\eqref{e.metbas}, subject to the constraint of being nonpositive in $\overline B_1(x)$. In other words, we define $m_\mu(y,x)$ for each $x,y\in \Rd$ by
\begin{multline}\label{e.mmu}
m_\mu(y,x):= \sup\big\{ w(y) \,:\, w\in \USC( \Rd ) \ \mbox{is a subsolution of~\eqref{e.metbas} in} \ \Rd \\  \mbox{and} \ w \leq 0 \ \mbox{on} \ \overline B_1(x) \big\}.
\end{multline}
Here $\USC(X)$ denotes the set of upper semicontinuous functions on $X$, which is the appropriate space for viscosity subsolutions (c.f.~\cite{AT}). It is convenient to extend the definition of $m_\mu$ from $\{ x \}$ to arbitrary compact $K \subseteq \Rd$: we define, for every $y\in \Rd$,
\begin{multline}\label{e.mmuK}
m_\mu(y,K):= \sup\big\{ w(y) \,:\, w\in \USC( \Rd ) \ \mbox{is a subsolution of~\eqref{e.metbas} in} \ \Rd \\  \mbox{and} \ w \leq 0 \ \mbox{on} \ K+\overline B_1 \big\}.
\end{multline}
Note that $m_\mu(y,x )=m_\mu(y, \{x\} )$. The basic properties of $m_\mu$ are summarized in the following proposition.

\begin{prop}\label{existMP}
Let $\mu > 0$ and $K$ be compact subset of $\Rd$.
\begin{enumerate}
\item[(i)] The function $m_\mu(\cdot,K )$ is a solution of 
\begin{equation}\label{mpagan}
\left\{ \begin{aligned}
& -\tr\left( A(y ) D^2m_\mu(\cdot,K ) \right) +H(Dm_\mu(\cdot,K ),y ) = \mu & \mbox{in} & \ \Rd\setminus (K+\overline B_1), \\
& m_\mu(\cdot,K ) = 0 & \mbox{in} & \ K+\overline B_1. 
\end{aligned} \right.
\end{equation}

\item[(ii)] For every $y,z\in \Rd$ and $K,K'$ compact subsets of $\Rd$,
\begin{equation}\label{lips}
|m_\mu(y,K )  - m_\mu(z,K' )| \leq L_\mu\left(|y-z|+\dist_H(K,K')\right) .
\end{equation}

\item[(iii)] For $x,y,z\in \Rd$,
\begin{equation}\label{subadd}
m_\mu(y,x ) \leq m_\mu(y,z ) + m_\mu(z,x )+L_\mu.
\end{equation}

\item[(iv)]  There exist $0< l_\mu\leq L_\mu$ such that, for some $C,c>0$ depending only on an upper bound for~$\mu$, 
\begin{equation}\label{lmuLmu}
 c\mu \leq l_\mu \leq L_\mu \leq C
\end{equation}
and
\begin{equation}\label{control2}
l_\mu \dist(y,B(x,1)) \leq m_\mu(y,x ) \leq L_\mu \dist(y,B(x,1)).
\end{equation}

\item[(v)]  Let $K$ be a compact subset of $\R^d$. Then 
\begin{equation}\label{esti.K} 
m_\mu(y,K ) \geq l_\mu (\dist(y,K)-2).
\end{equation}
\end{enumerate}
\end{prop}

\begin{proof} (i) and (ii) for $K=K'$ are proved in \cite{AT}. To show \eqref{lips}, it just remains to check that 
\be\label{kqcbskudfbn:}
|m_\mu(y,K )  - m_\mu(y,K' )| \leq L_\mu\dist_H(K,K').
\ee
We first note that  $\dist_H(K+\overline B_1,K'+\overline B_1)\leq \dist_H(K,K')$. Let $w(y)= m_\mu(y, K )-L_\mu \dist_H(K,K')$. Then $w$ is a subsolution to   \eqref{e.metbas}. Moreover, using the Lipschitz estimate \eqref{lips} for $K=K'$ and the fact that $m_\mu(y,K )\leq 0$ in $K+\overline B_1$, $$
w(y)= m_\mu(y, K )-L_\mu \dist_H(K,K')\leq 0 \qquad  \forall y\in K'+\overline B_1.
$$ 
By definition of $m_\mu(\cdot, K' )$, this implies that $m_\mu(\cdot, K' )\geq w$. Then  \eqref{kqcbskudfbn:} follows.

(iii) Let $w(y)= m_\mu(y,x ) - m_\mu(z,x )-L_\mu$. Then $w$ is a subsolution  to \eqref{e.metbas} in $\R^d$ which satisfies thanks to \eqref{lips},  
$$
w(y) \leq m_\mu(z,x )+L_\mu|z-y|-m_\mu(z,x )-L_\mu \leq 0\qquad \forall y\in B_1(z).
$$
This implies \eqref{subadd} by definition of $m_\mu(\cdot, z )$.

For proving (iv), first note that $y\to l_\mu(|y-x|-1)_+$ is a subsolution of \eqref{e.metbas} for $l_\mu$ small enough: hence the left-hand side of \eqref{lmuLmu}. As $y\to L_\mu(|y-x|-1)_+$ is a supersolution \eqref{mpagan} for $L_\mu$ large enough, we get the right-hand side of \eqref{lmuLmu} by comparison (see \cite{AT}). 

(v) Let $\xi:\R^d\to \R$ be a standard mollification kernel and denote, for $\ep>0$, $\xi_\ep(y) := \ep^{-\d} \xi(y/\ep)$. Set $w := \dist(\cdot, K)* \xi_\ep $. Since  $\dist(\cdot, K)$ is Lipschitz continuous, we have $\|Dw\|_\infty\leq 1$, $\|D^2w\|_\infty\leq 1/\ep$ and $\|w-\dist(\cdot, K)\|_\infty\leq C\ep$. For $\lambda= c\ep\mu$ with $c>0$ small enough, the function $\lambda (w-C\ep)$ is a subsolution of \eqref{e.metbas} which is nonpositive on $K+\overline B_1$. By definition of $m_\mu(\cdot, K )$ we have therefore 
$$
m_\mu(y, K ) \geq \lambda (w(y)-C\ep) \geq \lambda ( \dist(y, K) - 2C\ep) \geq l_\mu( \dist(y, K) - 2),
$$
for $\ep$ small enough (and changing the definition of $l_\mu$ if necessary).
\end{proof}

We continue by introducing some notation and basic observations regarding the sublevel sets of the maximal subsolutions, which play a key role in our analysis. For each $t\geq 0$, we let $\Re_{\mu,t}$ denote the $t$-sublevel set of $m_\mu(\cdot,0 )$, that is,
\begin{equation*}\label{}
 \Re_{\mu,t}:= \left\{ z \in \Rd \,:\, m_\mu(z,0 ) \leq t \right\}.
\end{equation*}
Note that, by \eqref{control2}, $\Re_{\mu,0}= \overline B_1$. For each $\mu > 0$, we think of $\Re_{\mu,t}$ as a ``moving front" with the variable~$t$ representing time.

\begin{prop}\label{p.pptRe} For each $\mu>0$ and $t\geq 0$:
\begin{itemize}
\item[(i)] $\Re_{\mu,t}$ is a compact connected subset of $\R^d$.

\item[(ii)] For every $0\leq s< t$, 
\begin{equation}\label{e.movefron}
\dist_H\left( \Re_{\mu,s} , \Re_{\mu,t} \right) \leq \frac{1}{l_\mu} |s-t|+2.
\end{equation}
\end{itemize}
\end{prop}

\begin{proof} (i) Boundedness of $\Re_{\mu,t}$ comes from \eqref{control2}. To prove that $\Re_{\mu,t}$ is connected, let ${\mathcal W}$ be a connected component of $\Re_{\mu,t}$. We claim that ${\mathcal W}$ contains the ball $B_1$: this will show that $\Re_{\mu,t}$ consists in a unique connected component, i.e., that $\Re_{\mu,t}$ is connected. For this, let us assume that, contrary to our claim, $B_1\not\subseteq {\mathcal W}$. Since, by \eqref{control2} again, $B_1$ lies in the interior of $\Re_{\mu,t}$, the sets ${\mathcal W}$ and $B_1$ must have an empty intersection. Therefore $m_\mu(\cdot, 0 )$ is  a solution to the problem
$$
 -\tr\left( A(y ) D^2m_\mu(\cdot,K ) \right) +H(Dm_\mu(\cdot,K ),y ) = \mu\qquad {\rm in}\; {\mathcal W}, 
 $$
 with boundary conditions $m_\mu(\cdot,0 )=t$ on $\partial {\mathcal W}$. As the constant map $w(y)=t$ is a strict subsolution of this equation, we have $m_\mu(\cdot, 0 )>w=t$ in the interior of ${\mathcal W}$ by comparison (see \cite{AT}). This contradicts the definition of ${\mathcal W}$. 
 
(ii) As $\Re_{\mu,s}\subseteq \Re_{\mu,t}$, we just have to prove that 
\begin{equation}\label{ljhqbdfsljhhj}
\Re_{\mu,t}\subseteq \Re_{\mu,s}+(2+\frac{1}{l_\mu} |s-t|)\overline B_1.
\end{equation} 
 Set $K:= \Re_{\mu,s} $ and notice that
\begin{equation}\label{lkjhbqscvezrj}
m_\mu(y,0 )\geq m_\mu(y, K ) + s \quad {\rm in }\ \R^d\setminus K.
\end{equation}
Indeed, let $\ep >0$ and denote 
\begin{equation*}\label{}
w(y):= \begin{cases} m_\mu(y,0 ) & y\in K, \\
\max\left\{ m_\mu(y,0 ),m_\mu(y, K )+s-\ep\right\} & y\in \R^d \setminus K. \end{cases}
\end{equation*}
Observe that $w$ is a subsolution of \eqref{e.subsol} in $\Rd$ which vanishes on $\overline B_1$. Hence $w\leq m_\mu(\cdot,0 )$ by the definition~\eqref{e.mmu}. Letting $\ep \to 0$ we get \eqref{lkjhbqscvezrj}. Combining estimate \eqref{esti.K} with \eqref{lkjhbqscvezrj}, we obtain
\begin{equation*}
 m_\mu(\cdot, 0 )\geq s+ l_\mu \left( \dist(\cdot,K) -2 \right)\quad \mbox{in} \ \Rd\setminus K.
\end{equation*}
This yields~\eqref{ljhqbdfsljhhj}. 
\end{proof}

\section{Estimating the fluctuations}\label{s.fluctuation}

The statement of Theorem~\ref{mpEE} can be divided into two parts: (i) an estimate of the random part of the error, that is, of $m_\mu(y,0) - \E \left[ m_\mu(y,0) \right]$; and (ii) an estimate of the nonrandom part of the error, that is, of $\E \left[ m_\mu(y,0) \right] - \overline m_\mu(y)$. In this section we focus on the first part. Throughout the rest of the paper, we denote the mean of $m_\mu(y,0)$ by
\begin{equation} \label{e.meandef}
M_\mu(y):= \E \left[ m_\mu(y,0) \right].
\end{equation}

\begin{prop}
\label{p.fluc}
Fix $\mu_0\geq 1$. There exists $C>0$, depending only on $(q,\Lambda,\mu_0)$ such that, for every $0< \mu \leq \mu_0$, $\lambda\geq C$ and $|y|\geq C\mu^{-2}$, we have
\begin{equation}\label{e.fluc}
\P\big[ \left|m_\mu(y,0 ) - M_\mu(y) \right| > \lambda  \big] \leq \exp\left( - \frac{\mu^4\lambda^2}{C |y|}\right).
\end{equation}
\end{prop}

\subsection{{An heuristic proof of Proposition~\ref{p.fluc}}}
\label{ss.heur}
The overall strategy underlying the proof of~Proposition~\ref{p.fluc} is similar to the argument of~\cite[Proposition~4.1]{ACS} in the first-order case (which we recommend reading first): we localize the maximal subsolutions in their sublevel sets and apply Azuma's inequality. However, the problem is more difficult here  because the presence of a nonzero diffusion renders the localization phenomenon much more subtle.

The idea to use Azuma's inequality to estimate the fluctuations of a random variable like $m_\mu(y,0 )$ is due to Kesten~\cite{K2}, who applied it to the passage time function in first-passage percolation. Later, Zhang~\cite{Z} modified the approach of~\cite{K2} by introducing the idea of conditioning on the environment in the sublevel sets of the passage time function-- a more geometrically natural construction which was extended to first-order equations in the proof of~\cite[Proposition 4.1]{ACS}.

In this subsection we explain the heuristic ideas underlying the argument for Proposition~\ref{p.fluc}. Most of what is presented here is \emph{not} rigorous, although it is made precise in the rest of the section.

We begin by fixing $y\in \Rd$ with $|y| \gg 1$. To estimate the distribution of the random variable $m_\mu(y,0 )$, we recall that $\Re_{\mu,t}$ denotes the $t$-sublevel set of $m_\mu(\cdot,0 )$, that is,
\begin{equation*}\label{}
\Re_{\mu,t}:= \left\{ z\in \R^d  \,:\, m_\mu(z,0 ) \leq t \right\}.
\end{equation*}
We think of $\Re_{\mu,t}$ as a ``moving front" which ``reveals the environment" as the ``time" $t$ increases. We let $\F'_{\mu,t}$ denote ``the $\sigma$-algebra generated by the environment in $\Re_{\mu,t}$". 

The ``best guess" for the difference between the random variable $m_\mu(y,0 )$ and its mean $M_\mu(y):= \E\left[m_\mu(y,0 )\right]$, given the information about the medium in $\F'_{\mu,t}$, is
\begin{equation*}\label{}
X_t := \E \left[ m_\mu(y,0 ) \,\big\vert\, \F'_{\mu,t} \right] - M_\mu(y).
\end{equation*}
Note that $\{ X_t\}_{t\geq 0}$ is a martingale with $\E X_t=0$. If we could show, for some constant $C_\mu >0$, that
\begin{equation}\label{e.keyinc}
|X_t - X_s| \leq C_\mu (1+|s-t|)
\end{equation}
and, for $T = C_\mu|y|$,
\begin{equation}\label{e.keyinc2}
X_T = m_\mu(y,0 ) - M_\mu(y),
\end{equation}
then Azuma's inequality yields
\begin{equation*}\label{}
\Prob \left[\left|m_\mu(y,0 ) - \E[m_\mu(y,0 )] \right|  > \lambda \right] \leq \exp\left(-\frac{2\lambda^2}{C_\mu|y|} \right),
\end{equation*}
which is the estimate in Proposition~\ref{p.fluc}, albeit with a less explicit constant $C_\mu$. 

Thus the key step in the proof of~Proposition~\ref{p.fluc} is proving~\eqref{e.keyinc} and~\eqref{e.keyinc2}. The quantity $|X_t-X_s|$ for $0 < s < t$ represents ``the magnitude of the improvement in the estimate for $m_\mu(y,0 )$ after gaining information about the environment in $\Re_{\mu,t}\setminus \Re_{\mu,s}$." The reason that we can expect analogues of~\eqref{e.keyinc} and~\eqref{e.keyinc2} to be true has to do with the Lipschitz estimates available for the viscous Hamilton-Jacobi equation and the fact that the maximal subsolutions ``localize in their sublevel sets." The rest of this subsection is devoted to giving  informal heuristic arguments for~\eqref{e.keyinc} and~\eqref{e.keyinc2}. 

Before proceeding to the heuristic proofs, we review the  ingredients in the arguments. The crucial localization property can be stated more (but still not completely) precisely as ``$m_\mu(y,0 )$ is (almost) measurable with respect to the $\sigma$-algebra $\F_{\mu,t}'$ generated by the environment inside $\Re_{\mu,t}$, provided that $t \geq m_\mu(y,0 )$." In other words, if $0 < s \leq t$, then
\begin{equation}\label{e.crutid1}
m_\mu(y,0 )\indc_{\{ y\in \Re_{\mu,s}\}}(\omega) \approx
\E\big[ m_\mu(y,0 )\ |\ \F'_{\mu,t}\big]\indc_{\{y\in \Re_{\mu,s}\}}.
\end{equation}
Put yet another way: ``the sublevel set $\Re_{\mu,t}$ is (almost) measurable with respect to the $\sigma$-algebra $\F_{\mu,t}'$ generated by the environment inside itself."  If this were true, as it is in the first-order case, then since $m_\mu(y,\Re_{\mu,t} )$ depends only on the environment in the complement of $\Re_{\mu,t}$ we would have, by independence,
\begin{equation}\label{e.crutid2}
\E \big[ m_\mu(y,\Re_{\mu,t} ) \,\big\vert\, \F_{\mu,t}' \big](\omega) \approx \sum_{\scriptsize{K \text{ compact}}} \E \left[ m_\mu(y,K )\right] \indc_{\{\Re_{\mu,t}=K\}}.
\end{equation}
Here things are \emph{much} simpler in the first-order case because the sublevel sets strongly localize the maximal subsolutions and in particular~\eqref{e.crutid1} and~\eqref{e.crutid2} hold with equality; see~\cite[(4.9)]{ACS} and~\cite[Lemma~3.1]{ACS}. The sum on the right of~\eqref{e.crutid2} is interpreted by partitioning the set of compact Borel sets into a finite number of sets of small diameter in the Hausdorff distance (see below).

Unfortunately, the second-order term destroys this strong localization property. An intuition is provided by the interpretation of $m_\mu(y,0 )$ in terms of optimal control theory. Under this interpretation, in the first-order case, once the environment is fixed the optimal trajectories are deterministic and stay confined within the sublevel sets. In contrast, the diffusive term requires the maximal subsolutions to ``see all of the environment" because optimal trajectories are stochastic and thus may venture far outside the sublevel sets before eventually returning to $B_1$. What saves us is a continuous dependence-type estimate (Lemma~\ref{l.seige}, below) which implies that the maximal subsolutions possess a weaker form of the localization property. This allows us to replace $\Re_{\mu,t}$ by a close approximation of it, denoted by $\Se_{\mu,t}$, which does have the property of being measurable with respect to the $\sigma$-algebra generated by the environment inside itself. The rigorous justifications of~\eqref{e.crutid1} and~\eqref{e.crutid2} appear below in Lemmas~\ref{l.frumpies} and~\ref{l.gumpers}, respectively.

The next ingredient we discuss is the ``dynamic programming principle." Under the control theoretic interpretation that $m_\mu(y,K )$ is measuring a ``cost imposed by the environment $\omega$ to transport a particle from $y$ to $K$," it is intuitively clear that
\begin{equation}\label{e.dpp.h}
m_\mu(y,0 ) \indc_{\{y\not\in \Re_{\mu,t}\}} \approx \left(  t + m_\mu(y,\Re_{\mu,t} ) \right) \indc_{\{y\not\in \Re_{\mu,t}\}}.
\end{equation}
In the first-order case, the quantities in~\eqref{e.dpp.h} are interpreted in such a way that~\eqref{e.dpp.h} holds with equality; see~\cite[(A.20)]{ACS}. This cannot occur in the second-order case due to the fact that a diffusion ``has more than one way of going between points" (here we are quoting Sznitman~\cite[Example~1.1]{SzEE}). Nevertheless, we show in Lemma~\ref{l.dpp} that~\eqref{e.dpp.h} holds provided we allow for a suitably small error.

We now assemble the above imprecise ingredients into an imprecise (but hopefully illuminating) proof of~\eqref{e.keyinc}. Below we make free use of~\eqref{e.dpp.h}, ~\eqref{e.crutid1} and~\eqref{e.crutid2} and are not troubled by the difficulties in interpreting the sum on the right of~\eqref{e.crutid2}. We also allow the constants to depend on~$\mu$. 

\begin{proof}[{Heuristic proof of~\eqref{e.keyinc} and~\eqref{e.keyinc2}}]
The reason that~\eqref{e.keyinc2} should hold is straightforward: according to~\eqref{e.movefron}, if $T \geq L_\mu |y|$, then we have $\indc_{\{ y\in \Re_{\mu,T}\}} \equiv 1$ . Thus~\eqref{e.crutid1} yields~\eqref{e.keyinc2}. 

To obtain~\eqref{e.keyinc} it suffices to prove the following two estimates for fixed $0< s \leq t$:
\begin{equation}\label{e.keyinc.a}
|X_t-X_s| \leq \left|  \E \big[ m_\mu(y,\Re{\mu,t} ) \,\big\vert\, \F_{\mu,t}' \big] - \E \big[ m_\mu(y,\Re{\mu,s} ) \,\big\vert\, \F_{\mu,s}' \big] \right| + |s-t| 
\end{equation}
and
\begin{equation}\label{e.keyinc.b}
\left|  \E \big[ m_\mu(y,\Re_{\mu,t} ) \,\big\vert\, \F_{\mu,t}' \big] - \E \big[ m_\mu(y,\Re_{\mu,s} ) \,\big\vert\, \F_{\mu,s}' \big] \right| \leq C(1+|s-t|). 
\end{equation}

To get~\eqref{e.keyinc.a}, we write
\begin{equation*}\label{}
X_t = \E \big[ m_\mu(y,\Re_{\mu,t} ) \,\big\vert\, \F_{\mu,t}' \big]  \indc_{\{y \in \Re_{\mu,t}\}} +  \E \big[ m_\mu(y,\Re_{\mu,t} ) \,\big\vert\, \F_{\mu,t}' \big]  \indc_{\{y \not\in \Re_{\mu,t}\}}, 
\end{equation*}
subtract from this a similar expression for $X_s$ and apply~\eqref{e.dpp.h}.

To get~\eqref{e.keyinc.b}, we use~\eqref{e.crutid1} and~\eqref{e.crutid2} to see that
\begin{multline}
\label{e.doubsum}
\E \big[ m_\mu(y,\Re_{\mu,t} ) \,\big\vert\, \F_{\mu,t}' \big] - \E \big[ m_\mu(y,\Re_{\mu,s} ) \,\big\vert\, \F_{\mu,s}' \big] \\ =\sum_{\scriptsize{K,\, K' \text{ compact}}} \left( \E \left[ m_\mu(y,K )\right] - \E \left[ m_\mu(y,K' )\right] \right) \indc_{\{\Re_{\mu,t}=K\}}\indc_{\{\Re_{\mu,s}=K'\}}.
\end{multline}
Next we notice that the Lipschitz estimates imply that, either $\indc_{\{\Re_{\mu,t}=K\}}\indc_{\{\Re_{\mu,s}=K'\}} \equiv 0$ or else $K$ and $K'$ are closer than $C|s-t|$, measured in the Hausdorff distance. Using the Lipschitz estimate again, we then get
\begin{equation*}\label{}
\left| \E \left[ m_\mu(y,K )\right] - \E \left[ m_\mu(y,K' )\right] \right| \indc_{\{\Re_{\mu,t}=K\}}\indc_{\{\Re_{\mu,s}=K'\}} \leq C (1+ |s-t|) \indc_{\{\Re_{\mu,t}=K\}}\indc_{\{\Re_{\mu,s}=K'\}}. 
\end{equation*}
Inserting this into the right side of~\eqref{e.doubsum} and summing over all compact $K$ and $K'$ yields~\eqref{e.keyinc}.
\end{proof}

The rigorous version of the heuristic proof above is given in Subsection~\ref{ss.fluc}, and it relies on the precise justification of~\eqref{e.crutid1},~\eqref{e.crutid2} and~\eqref{e.dpp.h}, found in Lemmas~\ref{l.frumpies},~\ref{l.gumpers} and~\ref{l.dpp}, respectively, in Subsection~\ref{e.justify}. In order to prove the latter, we first address the problems surrounding the lack of strong localization of $\Re_{\mu,t}$. This is the subject of the next subsection.

\subsection{Localizing the sublevel sets of the maximal subsolutions}
\label{ss.localize}
We consider maximal solutions of the inequality
\begin{equation}\label{e.subsol}
\left\{ \begin{aligned}
& -\tr\left( A(y ) D^2w \right) + H(Dw,y ) \leq \mu &  \mbox{in} & \ U,\\
& w \leq 0 & \mbox{on} & \ \overline B_1, 
\end{aligned} \right.
\end{equation}
where $U$ is an open subset of $\Rd$ with $\overline B_1 \subseteq U$. The maximal subsolutions are defined by
\begin{equation}\label{e.mmuUdef}
m_\mu^U(y ) := \sup \left\{ w(y) \, : \,  w\in \USC(U) \ \mbox{satisfies} \ \eqref{e.subsol} \right\}.
\end{equation}
The advantage of this quantity over $m_\mu$ is that it depends only on the environment in $U$. Indeed, it is immediate from~\eqref{e.mmuUdef} that, for each $y\in U$, the random variable $m_\mu^U(y )$ is $\mathcal G(U)$--measurable. In order to estimate the dependence of the original $m_\mu$ on the environment outside of one of its level sets, we compare $m_\mu$ to $m_\mu^U$ for $U$'s only a little larger than the sublevel sets. This is the purpose of this subsection, and the key estimates are obtained below in Corollary~\ref{c.pillage}. 

\smallskip

We continue with the basic properties of $m_\mu^U$ needed in the arguments below: proofs of these statements can be found in \cite{AT}. It follows from the definition that
\begin{equation*}\label{}
\overline B_1 \subseteq U \subseteq V \qquad \mbox{implies that} \qquad m_\mu^{V}(\cdot ) \leq m_\mu^U(\cdot ) \quad \mbox{in} \ U.
\end{equation*}
In particular, since $m_\mu^{\Rd}(\cdot ) = m_\mu(\cdot,0 )$,
\begin{equation}\label{e.mmuUub}
m_\mu(\cdot,0 ) \leq m_\mu^U(\cdot ) \quad \mbox{in} \ U.
\end{equation}
We note that $m_\mu^U(y ) =+\infty$ if $y$ does not belong to the connected component of $U$ containing~$\overline B_1$. We also define $m_\mu^K$ for any Borel set $K\in \mathcal B$ which contains $\overline B_1$ in its interior by setting $m_\mu^K:= m_\mu^{{\rm int} K}$. 
As shown in~\cite{AT}, by Perron's method, $m^U_\mu(\cdot )$ satisfies
\begin{equation*}\label{}
-\tr\left( A(y ) D^2m^U_\mu \right) + H\big(Dm^U_\mu,y \big) = \mu \quad \mbox{in} \ \left\{ y\in U \,:\, m_\mu^U(y ) < +\infty \right\} \setminus \overline B_1.
\end{equation*}
Therefore, the interior Lipschitz estimates apply: let $\tilde U$ be the connected component of $U$ containing $B_1$ and let
\begin{equation*}\label{}
U_h := \left\{ x\in \tilde U \, : \, \dist(x,\partial \tilde U) > h \right\},
\end{equation*}
we have that $m_\mu(\cdot )$ is Lipschitz on $\overline U_1$ and
\begin{equation}\label{e.flughagen}
\esssup_{y\in U_1} \big|Dm^U_\mu(y )\big| \leq L_\mu.
\end{equation}
The Lipschitz constant of $m_\mu$ on $\overline U_1$ may depend on the geometry of $U_1$ and in general can be much larger than $L_\mu$. However,~\eqref{e.flughagen} implies that the Lipschitz constant of $m_\mu$ is at most $L_\mu$ relative to each of its sublevel sets which are contained in $U_1$. That is, for every $t>0$,
\begin{multline} \label{e.lipflug}
\left\{ y\in U \, :\, m^U_\mu(y ) \leq t \right\} \subseteq \overline U_1 \quad \mbox{implies that} \\
\mbox{for every} \ y,z\in \left\{ y\in U \, :\, m^U_\mu(y ) \leq t \right\}, \quad \left| m_\mu^U(y ) - m^U(z ) \right| \leq L_\mu|y-z|.
\end{multline}
Here is the proof: modify and extend~$m^U_\mu(\cdot )$ to the complement of~$\{ m^U_\mu(\cdot ) \leq t \}$ by setting it equal to~$t$ there. It is then clear that this modified function has gradient bounded by $L_\mu$ in~$\Rd$, and is therefore globally Lipschitz with constant~$L_\mu$. Since we did not alter any of the values in the~$t$--sublevel set, we obtain the conclusion of~\eqref{e.lipflug}.

Observe that~\eqref{e.mmuUub} and~\eqref{control2} yield the lower growth estimate
\begin{equation}\label{e.grthUw}
l_\mu ( |y|-1 ) \leq m^U_\mu(y ) \quad \mbox{for every} \ y \in U
\end{equation}
and~\eqref{e.lipflug} the upper growth estimate
\begin{equation}\label{e.grthUwup}
m^U_\mu(y ) \leq L_\mu(|y|-1) \quad \mbox{provided that} \quad  y \in \left\{ z\in U \, :\, m^U_\mu(z ) \leq t \right\} \subseteq \overline U_1.
\end{equation}

We require a slightly stronger version of~\eqref{e.grthUw}, contained in the following lemma, which says that the sublevel sets of $m_\mu^U$ grow at a rate bounded uniformly from above. The proof is the same as for \eqref{e.movefron}, so we omit it. 

\begin{lem}
For every $0 < s \leq t$,
\begin{equation} \label{e.levsflug}
\left\{ y\in U \, :\, m^U_\mu(y ) \leq t \right\} \subseteq \left\{ y\in U \, :\, m^U_\mu(y ) \leq s \right\} + \overline B_{2+(t-s)/l_\mu}.
\end{equation}
\end{lem}

Following the arguments of  Proposition \ref{p.pptRe}, one can also easily check that sublevel sets of $m^U_\mu(\cdot )$ are connected: more precisely,
$$
\mbox{if \ \ $\{m^U_\mu(\cdot )\leq t\} \subseteq \overline U_1$, \quad then \ \
$\{m^U_\mu(\cdot )\leq t\}$ \ \  is connected.}
$$

We next show that $m_\mu^U$ is a very good approximation for $m_\mu$ in each of the sublevel sets of $m_\mu^U$ which are contained in $\overline U_1$. This result is central to the localization result and hence the analysis in this paper. The proof relies on a novel change of the dependent variable. 

\begin{lem} \label{l.seige}
Assume $U \subseteq \Rd$ and $t>0$ are such that
\begin{equation}\label{e.encamped}
 \left\{ y\in U \,: \, m_\mu^U(y ) \leq t  \right\} \subseteq \overline U_1.
\end{equation}
Then
\begin{multline}\label{e.seige}
m_\mu^U(y ) - m_\mu(y,0 ) \leq \frac{4\Lambda L_\mu^3}{\mu l_\mu} \exp\left( \frac{4L_\mu}{l_\mu} \right) \exp\left( - \frac{\mu }{\Lambda L_\mu^2} ( t - m_\mu(y,0 ) )\right) \\ \mbox{for every} \quad y\in \left\{ z\in U\,:\, m_\mu^U(z ) \leq t  \right\}.
\end{multline}
\end{lem}
\begin{proof}
We divide the proof into two steps. In the first step we perturb~$m_\mu$ in order to permit a comparison to~$m_\mu^U$. The perturbation is strong near $\partial U$, which forces~$m_\mu$ to be larger than~$m_\mu^U$ on the boundary of~$U$. On the other hand, the perturbation is very small in the~$s$--sublevel set of~$m_\mu^U$ for $t-s\gg1$, which allows for the comparison to be useful. In the second step of the argument, we perform the comparison and eventually deduce~\eqref{e.seige}. Throughout we simplify the notation by writing $m_\mu(y) = m_\mu(y,0)$.

\smallskip

\emph{Step 1.}
We set $w(y):= \varphi\left( m_\mu(y ) \right)$, where $\varphi:\R_+ \to \R_+$ is given by
\begin{equation*}\label{}
\varphi(s) := s + k  \exp\left( \frac{\mu}{\Lambda L_\mu^2}(s-t+k) \right) \quad k:= \max \left\{ m^U_\mu(y ) - m_\mu(y,0 ) \,: \, m^U_\mu(y ) \leq t \right\}.
\end{equation*}
We claim that $w$ satisfies
\begin{equation}\label{e.pert1}
-\tr\left( A(y ) D^2w \right) + H(Dw,y ) \geq \mu \quad \mbox{in} \ U_1 \setminus \overline B_1
\end{equation}
as well as  
\begin{equation}\label{e.pert2}
w\geq 0\quad \mbox{on} \ \overline  B_1 \qquad \mbox{and} \qquad w \geq m_\mu^U \quad \mbox{on} \ \partial \left\{ y\in  U \,:\, m^U_\mu(y ) \leq t \right\}.
\end{equation}
To check~\eqref{e.pert1}, we perform formal computations assuming that $m_\mu$ is smooth. While we have only that $m_\mu$ is Lipschitz, in general, it is routine to make the argument rigorous in the viscosity sense by doing the  same computation on a smooth test function. Computing the derivatives of $w$, we find
\begin{align*}
Dw(y) & = \left( 1 + \frac{\mu k}{\Lambda L_\mu^2} \exp\left( \frac{\mu}{\Lambda L_\mu^2} \left( m_\mu(y) - t +k \right) \right)  \right) Dm_\mu(y),  \\
D^2w(y) & = \left( 1 + \frac{\mu k}{\Lambda L_\mu^2} \exp\left( \frac{\mu}{\Lambda L_\mu^2} \left( m_\mu(y) - t +k \right) \right)  \right) D^2m_\mu(y) \\ & \qquad +\frac{\mu^2k}{\Lambda^2 L_\mu^4}\exp\left( \frac{\mu}{\Lambda L_\mu^2} \left( m_\mu(y) - t +k \right) \right) Dm_\mu(y) \otimes Dm_\mu(y).
\end{align*}
Before we attempt to evaluate the left side of~\eqref{e.pert1}, we note that, since $H$ is convex and $H(0,y) \leq 0$, we have, for every $\lambda \geq 0$,
\begin{equation*} \label{}
H( (1+\lambda)p,y) \geq (1+\lambda)H(p,y).
\end{equation*}
In order to bound the trace of $A(y)$ and the second term in the expression for $D^2w$, we recall that $|A(y)| \leq \Lambda$ from the assumption~\eqref{e.sigbnd} and that $|Dm_\mu(y)|\leq L_\mu$ for every $y\in U_1$, which follows from~\eqref{e.encamped} and~Proposition~\ref{p.lipschitz}. Using these and the equation for $m_\mu$, we obtain, for every $y\in U_1\setminus B_1$,
\begin{align*}
-\tr\left( A(y ) D^2w \right) + H(Dw,y ) & \geq \left( 1 + \frac{\mu k}{\Lambda L_\mu^2} \exp\left( \frac{\mu}{\Lambda L_\mu^2} \left( m_\mu(y) - t +k \right) \right)  \right) \mu \\
& \qquad  - \frac{\mu^2k}{\Lambda^2 L_\mu^4}\exp\left( \frac{\mu}{\Lambda L_\mu^2} \left( m_\mu(y) - t +k \right) \right) \Lambda L_\mu^2 \\
& = \mu.
\end{align*}
To check~\eqref{e.pert2}, we use the definition of $k$ and $\varphi$ and the fact that $$\partial \left\{ y\in  U \,:\, m^U_\mu(y ) \leq t \right\} \subseteq \left\{ y\in  U \,:\, m^U_\mu(y ) = t \right\},$$ which follows from the continuity of $m_\mu^U$, to deduce, for each $y \in \partial \left\{ z\in  U \,:\, m^U_\mu(z ) \leq t \right\} $,
\begin{align*} \label{}
w(y) & \geq m_\mu(y) + k\exp\left(  \frac{\mu}{\Lambda L_\mu^2} \left( m_\mu(y) - t +k \right) \right) \\
& \geq m_\mu(y) + k\exp\left(  \frac{\mu}{\Lambda L_\mu^2} \left( m^U_\mu(y) - t \right) \right) \geq m_\mu(y) + k \geq m_\mu^U(y).
\end{align*}
The nonnegativity of $w$ on $\overline B_1$ is obvious. This completes the proof of~\eqref{e.pert1} and~\eqref{e.pert2}.

\smallskip

\emph{Step 2.}
By  comparison principle (c.f.~\cite{AT}), we obtain that
\begin{equation} \label{e.cmppert}
w \geq m_\mu^U \quad \mbox{in} \ \{ y \in   U \,:\, m^U_\mu \leq t\}.
\end{equation}
Indeed, while the comparison principle requires a strict inequality, we may compare $w$ to $(1-\ep)m_\mu^U$ for $\ep > 0$ and then send $\ep  \to 0$, using the observation that $(1-\ep) m_\mu^U$ is a strict subsolution of~\eqref{e.metbas} in $U$ by the convexity of $H$ and the fact that $0$ is a strict subsolution for $\mu > 0$. See~\cite[Lemma 2.4]{AT} for details.

Expressing~\eqref{e.cmppert} in terms of $m_\mu$ yields
\begin{equation}\label{e.gumpt}
m_\mu(y ) - m_\mu^U(y ) \geq - k \exp\left( \frac{\mu}{\Lambda L_\mu^2}(m_\mu(y )-t+k) \right) \quad \mbox{in} \ \left\{ x\in   U  \,: \, m_\mu^U(x ) \leq t \right\}.
\end{equation}
We complete the argument by using the Lipschitz estimate and~\eqref{e.gumpt} to estimate $k$ from above, and then feed the result back into~\eqref{e.gumpt}. We first show the following rough bound on $k$: 
\be\label{e.roughbdk}
k\leq t(1-l_\mu/L_\mu).
\ee 
Since $k= t-\min\left\{m_\mu(y,0 )\,: \, m^U_\mu(y ) =  t\right\}$, we may select $y_0\in U$ such that $m^U_\mu(y_0 ) = t$ and $k = t - m_\mu(y_0 )$. 
Then \eqref{e.grthUwup} implies that $|y_0|\geq t/L_\mu+1$ and \eqref{control2} gives that 
$(l_\mu/L_\mu) t \leq m_\mu(y_0)= t-k$, which yields \eqref{e.roughbdk}. 

\smallskip

Fix $t_1\in [0, t-k]$ to selected below. By~\eqref{e.movefron}, there exists $y_1\in \Rd$ such that $m_\mu(y_1,0 )=t_1$ and $|y_1-y_0|\leq (t-k-t_1)/l_\mu+2$. Using~\eqref{e.gumpt} at~$y_1$, and the Lipschitz estimate \eqref{e.lipflug}, we deduce
\be\label{kjhbefvz}
-k \exp\left( \frac{\mu}{\Lambda L_\mu^2}(t_1-t+k) \right)\leq t_1 - m_\mu^U(y_1 ) \leq t_1-t+ \frac{L_\mu}{l_\mu}(t-k-t_1)+2L_\mu.
\ee
Fix $\ep:=\exp(-1)$ and set $t_1:=  t - k - \mu^{-1}(\Lambda L_\mu^2)$. Observe that, in view of~\eqref{e.roughbdk}, we have $t_1\geq 0$ provided that $t\geq (\Lambda L_\mu^3)/(\mu l_\mu)$. Then \eqref{kjhbefvz} gives
$$
k \leq \frac{1}{1-\ep} \left( -\frac{\Lambda L_\mu^2}{\mu}+ \frac{\Lambda L_\mu^3}{\mu l_\mu}+2L_\mu\right)\leq 4 \frac{\Lambda L_\mu^3}{\mu l_\mu}.
$$
Inserting this into~\eqref{e.gumpt} yields~\eqref{e.seige} for $t\geq (\Lambda L_\mu^3)/(\mu l_\mu)$. We conclude by noting that~\eqref{e.seige}  always holds for $t\leq \Lambda L_\mu^3/ (\mu l_\mu)$.
\end{proof}

We next use the Lipschitz and growth estimates to translate the previous result that $m_\mu \approx m_\mu^U$ into a result asserting that, under the same hypotheses, the corresponding level sets are close. We define a constant $a_\mu$, used in the rest of this section, by
\begin{equation} \label{e.amu}
a_\mu:=1+ \frac{\Lambda L_\mu^2}{\mu l_\mu}\left( \frac{4L_\mu}{l_\mu} + \log\left( \frac{4\Lambda L_\mu^3}{\mu l_\mu^2} \right) \right).
\end{equation}
We note that, for $C>0$ depending only on $(q,\Lambda)$ and an upper bound for $\mu$, 
\begin{equation} \label{e.amubnd}
a_\mu \leq C l_\mu^{-2} \mu^{-1} \leq C\mu^{-3}.
\end{equation}

\begin{cor}\label{c.pillage}
Let $t>0$ and assume $U \subseteq \Rd$ is  such that
\begin{equation}\label{e.scribb}
\left\{ x\in \Rd \,: \, m_\mu^U(x ) \leq t  \right\} \subseteq \overline U_{2+a_\mu}.
\end{equation}
Then 
\begin{equation} \label{e.stoucha}
0 \leq m_\mu^U(t ) - m_\mu(t,0 ) \leq l_\mu \quad \mbox{for every} \ y\in \left\{ z\in U\,:\, m_\mu^U(z ) \leq t \right\}
\end{equation}
and
\begin{equation}\label{e.scrubb}
\left\{ y\in U\,:\, m_\mu^U(y ) \leq t \right\} \subseteq \Re_{\mu,t}  \subseteq\left\{ y\in U\,:\, m_\mu^U(y ) \leq t \right\} + \overline B_3.
\end{equation}
\end{cor}

\begin{proof} 
Define the constant
\begin{equation*} \label{}
h:=  \frac{\Lambda L_\mu^2}{\mu}\left( \frac{4L_\mu}{l_\mu} + \log\left( \frac{4\Lambda L_\mu^3}{\mu l_\mu^2} \right) \right)
\end{equation*}
and observe that, due to~\eqref{e.levsflug},~\eqref{e.scribb} and the fact that $a_\mu = 1+ h/l_\mu$,
\begin{equation} \label{e.scroub}
\left\{ y\in U\,:\, m_\mu^U(y ) \leq t + h \right\} \subseteq \left\{ y\in U \,:\, m_\mu^U (y ) \leq t \right\} + \overline B_{2+h/l_\mu} \subseteq \overline U_1.
\end{equation}
Therefore, applying~\eqref{e.seige} with $t+h$ in place of $t$, we find that
\begin{align*} \label{}
l_\mu & = \frac{4\Lambda L_\mu^3}{\mu l_\mu} \exp\left( \frac{4L_\mu}{l_\mu} \right) \exp \left(- \frac{\mu }{\Lambda L_\mu^2} h \right) && (\mbox{definition of } h)\\ 
& \geq \frac{4\Lambda L_\mu^3}{\mu l_\mu} \exp\left( \frac{4L_\mu}{l_\mu} \right) \exp\left( - \frac{\mu }{\Lambda L_\mu^2} ( t +h - m_\mu(\cdot ) )\right) && ( \mbox{in} \  \left\{ m_\mu^U(\cdot ) \leq t  \right\} ) \\
& \geq m^U_\mu(\cdot ) - m_\mu(\cdot,0 ) && ( \mbox{by~\eqref{e.seige},~\eqref{e.scroub}}).
\end{align*}
This completes the proof of~\eqref{e.stoucha}.
 
Next we show that \eqref{e.scrubb} holds. The first inclusion is immediate from~\eqref{e.mmuUub}. To show the second inclusion, we observe that, according to~\eqref{e.stoucha}, for every $\ep > 0$,
\begin{equation*}\label{}
m_\mu(\cdot,0 )> t-l_\mu \quad \mbox{on} \ \partial \left\{ y\in U \,:\, m^U_\mu(y ) \leq t + \ep \right\}.
\end{equation*}
Since $\Re_{\mu,t-l_\mu}$ is connected, after sending $\ep \to 0$ we conclude that  
\begin{equation*}\label{}
\Re_{\mu,{t-l_\mu}} \subseteq \left\{ y\in U \,:\, m^U_\mu(y ) \leq t \right\}.
\end{equation*}
The result now follows from \eqref{e.movefron} which implies $\Re_{\mu,t} \subseteq \Re_{\mu,t-l_\mu} + \overline B_3$. 
\end{proof}


\subsection{The rigorous justification of~\eqref{e.crutid1},~\eqref{e.crutid2} and~\eqref{e.dpp.h}}\label{e.justify}

Using the estimates in Corollary~\ref{c.pillage}, which assert that $m_\mu\approx m_\mu^U$ in sublevel sets of $m^U_\mu$ which are a unit distance from the boundary of~$U$, we construct a random ``moving front" $\Se_{\mu,t}$, close to $\Re_{\mu,t}$, which is ``measurable with respect to the $\sigma$-algebra generated by the environment inside itself." This is the key to obtaining the rigorous versions of~\eqref{e.dpp.h},~\eqref{e.crutid1} and~\eqref{e.crutid2} we need to complete the proof of Proposition~\ref{p.fluc}.

In order to build $\Se_{\mu,t}$ as well as interpret sums like on the right side of~\eqref{e.crutid2}, it is convenient to introduce a discretization of the set of compact subsets of $\Rd$ which contain $\overline B_1$. We denote the latter by $\mathcal K$ and endow it will the \emph{Hausdorff metric} $\dist_H$ defined by
\begin{align*}\label{}
\dist_H(E,F) := & \inf_{x\in E} \sup_{y\in F} |x-y| \vee \inf_{y\in F} \sup_{x\in E} |x-y| \\ 
= & \inf \left\{ \ep > 0 : E \subseteq F+B_\ep \ \mbox{and} \ F \subseteq E+B_\ep \right\}.
\end{align*}
The metric space $(\mathcal K,\dist_H)$ is locally compact (see Munkres~\cite{Munk}) and thus there exists a partition $(\Gamma_i)_{i\in \N}$ of $\mathcal K$ into Borel subsets of $\mathcal K$ of diameter at most $1$. That is, $\Gamma_i \subseteq \mathcal K$ and $\diam_H(\Gamma_i) \leq 1$ for all $i\in \N$ and $\Gamma_i \cap \Gamma_j = \emptyset$ if $i,j\in\N$ are such that $i\neq j$.

For each $i\geq 1$, we take $K_i$ to be the closure of the union of the elements of $\Gamma_i$. Observe that
\begin{equation}\label{e.caKK}
K \in \Gamma_i \qquad \mbox{implies that} \qquad K \subseteq K_i  \subseteq K + \overline B_{1}.
\end{equation}
We introduce, for each $i\in \N$, compact sets
\begin{equation*} \label{}
K_i \subseteq K_i' \subseteq K_i'' \subseteq \widetilde K_i
\end{equation*}
by setting (with the constant $a_\mu>0$ defined in~\eqref{e.amu})
\begin{equation*}\label{}
K_i' := K_i + \overline B_{2+a_\mu}, \qquad K_i'':= K_i' + \overline B_{7} \qquad \mbox{and} \qquad \widetilde K_i:= K_i''+ \overline B_{4}.
\end{equation*}
Here are the reasons we must introduce so many sets: 
\begin{itemize}

\item $K_i'$ provides extra room so that we can apply the localization results of the Subsection~\ref{ss.localize};

\item $K_i''$ must be a little larger than $K_i'$ so that the partition $\{ E_i(t) \}_{i\in \N}$ of $\Omega$ that we construct below has the property that $E_i(t) \in \G(K_i'')$, which forces the moving front $\Se_{\mu,t}$ we build to localize properly;

\item $\widetilde K_i$ is larger than $K_i''$ so that $\dist(K_i''+\overline B_3\,,\Rd\setminus \widetilde K_i) = 1$ and thus the finite range of dependence hypothesis yields
\begin{equation}\label{e.KtKindp}
\G\left(K_i'' +\overline B_3 \right) \qquad \mbox{and} \qquad \G\left( \Rd \!\setminus \! \widetilde K_i \right) \qquad \mbox{are \   independent.}
\end{equation}
\end{itemize}

We next use the partition $\{ \Gamma_i \}_{i\in\N}$ of $\mathcal K$ to construct, for each $t>0$, a partition of $\Omega$ into disjoint events $\{ E_i(t) \}_{i\in \N}$ which are approximations of the event  $\{ \omega\in\Omega\,:\,\Re_{\mu,t}(\omega) \in \Gamma_i\}$:
\begin{equation}\label{e.Eitintn}
`` \ E_i(t) \approx \left\{\omega\in \Omega\,:\, \Re_{\mu,t}(\omega) \in \Gamma_i \right\}. \ \mbox{"}
\end{equation}
The reason we do not define $E_i(t)$ with equality in~\eqref{e.Eitintn} is due to the fact that, as we have explained above, the $\Re_{\mu,t}$'s do not properly ``localize."  
The events $E_i(t)$ we introduce are better localized in the sense made precise by Lemma~\ref{l.loopers}.

To construct $E_i(t)$, we first make each of the $\Gamma_i$'s slightly larger by setting
\begin{equation*}\label{}
\widetilde \Gamma_i := \left\{ K \in \mathcal K \, : \, K \subseteq K_i \subseteq K+\overline B_4 \right\}
\end{equation*}
and then define, for every $t>0$ and $i\in \N$,
\begin{equation*}\label{}
F_i(t):=\left\{ \omega\in \Omega \,:\, \left\{ y\in K'_i\,:\, m^{K'_i}_\mu(y ) \leq t \right\} \in \widetilde \Gamma_i \right\}.
\end{equation*}
To make a partition of $\Omega$, we modify the $F_i(t)$'s by setting
\begin{equation*}\label{}
E_1(t) := F_1(t) \qquad \mbox{and} \quad E_{i+1}(t) := F_{i+1}(t) \setminus \bigcup_{j\leq i} E_{j}(t) \qquad \mbox{for every} \ i \in\N.
\end{equation*}

We now collect some properties of the families $\{F_i(t)\}$ and $\{E_i(t)\}$. 

\begin{lem}\label{lem:pptFiEi} Fix $t>0$. Then the family $(E_i(t))$ is a measurable partition of $\Omega$. Moreover, if $\omega\in F_i(t)$ and $y\in K_i$, we have 
\begin{equation}\label{e.keylocal2}
 \Re_{\mu,t} \subseteq K_i + \overline B_3,  \qquad  K_i \subseteq \Re_{\mu,t} + \overline B_4,\qquad 
 \dist_H(\partial \Re_{\mu,t}, \partial K_i)\leq 7
\end{equation}
and 
\begin{equation}\label{e.keylocal1}
 \left|m^{K_i'}_\mu(y )-m_\mu(y,0 )\right|  \leq 9L_\mu.
\end{equation}
 \end{lem}

\begin{proof} To prove that $\{ E_i(t)\}_{i\in \N}$ is a partition of $\Omega$, we note that, for any $\omega\in \Omega$,  
\begin{equation}\label{e.Rmutob}
\Re_{\mu,t}(\omega) \in \Gamma_i \quad \mbox{implies that} \quad \omega \in F_i(t). 
\end{equation}
Indeed, notice that $\{ m^{K_i'}_\mu(\cdot ) \leq t\} \subseteq \Re_{\mu,t}$ by~\eqref{e.mmuUub}, and thus $\Re_{\mu,t}(\omega) \in \Gamma_i$ implies that 
$$
\{y\in K_i' \,:\, m^{K_i'}_\mu(y ) \leq t\} \subseteq K_i\subseteq (K_i')_{2+a_\mu}.
$$ Corollary~\ref{c.pillage} and the fact that $\Re_{\mu,t}(\omega) \in \Gamma_i$ then imply  that
\begin{equation*}\label{}
K_i \subseteq \Re_{\mu,t} + \overline B_1 \subseteq \big\{ m^{K_i'}_\mu(\cdot ) \leq t \big\} + \overline B_4,
\end{equation*}
so that~\eqref{e.Rmutob} holds. Since $\{ \Gamma_i \}_{i\in\N}$ is a partition of $\mathcal K$,  we deduce from~\eqref{e.Rmutob} that
\begin{equation*}\label{}
\Omega = \bigcup_{i\in \N} \{ \omega\in \Omega \,:\, \Re_{\mu,t} \in \Gamma_i \}\subseteq   \bigcup_{i\in \N} F_i(t)\subseteq\Omega.
\end{equation*}
Then it is clear that $\{ E_i(t)\}_{i\in \N}$ is a partition of $\Omega$.

\smallskip

Next we show~\eqref{e.keylocal2}. Fix $\omega\in F_i(t)$. 
Owing to the definition of $K_i'$, we have 
$$ \left\{ y\in K_i' \,:\, m^{K_i'}_\mu(y ) \leq t \right\}  \subseteq (K_i')_{2+a_\mu}.
$$ In particular, we can apply Corollary \ref{c.pillage} to $U:={\rm int}(K_i')$ to deduce, by the definition of $F_i(t)$ and \eqref{e.scrubb}, that
\be\label{khbslig:,kjbe}
K_i\subseteq \left\{ y\in U\,:\, m_\mu^U(y ) \leq t \right\} +\overline B_4\subseteq \Re_{\mu,t}(\omega) +\overline B_4, 
\ee
and
$$
\Re_{\mu,t}(\omega)  \subseteq\left\{ y\in U\,:\, m_\mu^U(y ) \leq t \right\} + \overline B_3\subseteq K_i+\overline B_3.
$$
To show the last statement of \eqref{e.keylocal2}, let us notice that, by definition of $\tilde \Gamma_i$, 
$$
\left\{ y\in U\,:\, m_\mu^U(y ) \leq t \right\}\subseteq K_i\subseteq \left\{ y\in U\,:\, m_\mu^U(y ) \leq t \right\}+\overline B_4.
$$
Hence  
$$
\dist_H\left( \partial \left\{ y\in U\,:\, m_\mu^U(y ) \leq t \right\}, \partial K_i\right)\leq 4.
$$
On another hand, \eqref{e.scrubb}  implies that
$$
\dist_H\left(\partial \Re_{\mu,t}, \partial \left\{ y\in U\,:\, m_\mu^U(y ) \leq t \right\}\right) \leq 3. 
$$
Therefore
$
\dist_H\left(\partial \Re_{\mu,t}(\omega), \partial K_i\right) \leq 7
$
and the proof of \eqref{e.keylocal2} is complete. 

\smallskip

Let us finally check that \eqref{e.keylocal1} holds: by \eqref{e.stoucha} we have
\be\label{jkqbsflgjdnf:o}
0 \leq m_\mu^U(\cdot ) - m_\mu(\cdot,0 ) \leq l_\mu \quad \mbox{in} \ \left\{ y\in U\,:\, m_\mu^U(y ) \leq t \right\}.
\ee
If $x\in K_i\setminus \left\{ y\in U\,:\, m_\mu^U(y ) \leq t \right\}$, then, by definition of $F_i(t)$, there exists $x_0\in U$ such that 
$ m_\mu^U(x_0 ) \leq t$ and $|x-x_0|\leq 4$. We conclude by Lipschitz estimate and \eqref{jkqbsflgjdnf:o} that
\begin{equation*}
\left|m_\mu^U(x ) - m_\mu(x,0 )\right| \leq 8L_\mu +
\left|m_\mu^U(x_0 ) - m_\mu(x_0,0 )\right| \leq 9L_\mu. \qedhere
\end{equation*}
\end{proof}

We next verify that $E_i(t)$ is ``localized" in the sense that it belongs to $\mathcal G(K_i'')$. 

\begin{lem} \label{l.loopers}
For every $0 < s \leq t$ and $i,j \in \N$,
\begin{equation} \label{e.loopers}
F_i(s) \cap F_j(t) \neq \emptyset \qquad \mbox{implies that} \qquad E_i(s) \in \G(K_j'').
\end{equation}
In particular, $E_i(t) \in \G(K_i'')$ for every $i\in\N$.
\end{lem}
\begin{proof}
We first claim that
\begin{equation} \label{e.fishfishfish}
0 < s \leq t \quad \mbox{and} \quad F_i(s) \cap F_j(t) \neq \emptyset \qquad \mbox{implies that} \qquad K_i' \subseteq K_j''.
\end{equation}
Here is the proof of~\eqref{e.fishfishfish}: if $0 < s \leq t$ and $\omega\in F_i(s)\cap F_j(t)$, then by~\eqref{e.keylocal2} we have
\begin{equation*} \label{}
K_i  \subseteq \Re_{\mu,s}(\omega) +\overline  B_{4} \subseteq \Re_{\mu,t}(\omega) + \overline B_{4} \subseteq K_j+\overline B_{7}.
\end{equation*}
Hence $K_i' \subseteq K_j' + \overline B_{7} = K_j''$.

It is clear that $F_i(s)\in \G(K_i')$ for every $i\in \N$. If $0< s\leq t$ and $i,j\in \N$ are such that $F_i(s) \cap F_j(t) \neq \emptyset$, then we have $F_i(s)\in \G(K_j'')$ by~\eqref{e.fishfishfish}. Observe also that~\eqref{e.fishfishfish} yields the expression
\begin{equation*}\label{}
E_i(s) = F_i(s) \setminus \bigcup_{n \in \beta(i)} F_n(s),
\end{equation*}
where we have set $\beta(i):= \left\{ 1 \leq n < i \,:\, K_n' \subseteq K_i'' \right\}$.
The lemma follows.
\end{proof}

We now introduce the ``moving front" $\Se_{\mu,t}$ and the filtration $\F_{\mu,t}$ that we use in the proof of Proposition~\ref{p.fluc}. We define $\Se_{\mu,t}$ by
\begin{equation}\label{defSemu}
\Se_{\mu,t}(\omega) := K_i \qquad {\rm if }\; \omega\in E_i(t). 
\end{equation}
Observe that~\eqref{e.keylocal2} implies that $\Se_{\mu,t}$ is a good approximation of $\Re_{\mu,t}$:
\begin{equation}\label{e.ReSeRe}
\Re_{\mu,t} \subseteq  \Se_{\mu,t} +\overline B_3,\qquad \Se_{\mu,t} \subseteq \Re_{\mu,t}+ \overline{ B}_4
\qquad {\rm and }\qquad \dist_H\left(\partial \Re_{\mu,t},\partial  \Se_{\mu,t}\right)\leq 7.
\end{equation}
The filtration $\F_{\mu,t}$ is defined as the $\sigma$-algebra generated by events of the form
\begin{equation}\label{e.defFmut}
G \cap E_i(s), \qquad \mbox{where} \quad 0< s \leq t, \ i\in \N, \ G \in \mathcal G(K_i'').
\end{equation}
We also set $\F_{\mu,0} = \{ \emptyset, \Omega \}$.

Observe that $\{ \F_{\mu,t}\}_{t\geq 0}$ is indeed a filtration, since by definition it is increasing in $t$. It is clear from Lemma~\ref{l.loopers} that, for every $y\in \Rd$ and $t \geq 0$,
\begin{equation}\label{e.Semeasu}
\left\{ \omega\in \Omega \,:\, y\in \Se_{\mu,t} (\omega)\right\} \in \F_{\mu,t}.
\end{equation}
For convenience, below we write $\{ y \in \Se_{\mu,t} \}$ to denote the event $\{ \omega\in \Omega\,:\,  y \in \Se_{\mu,t}(\omega)\}$.

\smallskip

The next lemma is the rigorous justification of~\eqref{e.dpp.h}, which follows relatively easily from~\eqref{e.ReSeRe} and the Lipschitz estimates.
 
\begin{lem}\label{l.dpp}
For any $t>1$, 
\begin{equation}\label{e.dpp}
\left| m_\mu(y,0 )-\left(t+m_\mu(y,\Se_{\mu,t} )\right)\right|\indc_{\{ y\not\in \Se_{\mu,t}\}}(\omega) \leq 8L_\mu\indc_{\{y\not\in \Se_{\mu,t}\}}(\omega). 
\end{equation}
\end{lem}
\begin{proof} 
By the maximality of $m_\mu(\cdot,0 )$, we have
\begin{equation}\label{e.climb1}
\inf_{z\in \partial \Se_{\mu,t}} m_\mu(z,0 ) + m_\mu(\cdot,\Se_{\mu,t} ) \leq m_\mu(\cdot,0 ) \quad \mbox{in} \ \Rd\setminus \Se_{\mu,t}.
\end{equation}
Indeed, for every $\ep > 0$, the function
\begin{equation*}\label{}
w(y):= \begin{cases} 
m_\mu(y,0 ) & y \in \Se_{\mu,t},\\
\displaystyle \max\left\{ m_\mu(y,0 ), \inf_{z\in \partial \Se_{\mu,t}} m_\mu(z,0 ) + m_\mu(y,\Se_{\mu,t} ) - \ep \right\} & y\not \in \Se_{\mu,t},
\end{cases}
\end{equation*}
is a global subsolution which vanishes on $\overline B_1$, hence $w \leq m_\mu(\cdot,0 )$ by maximality. Sending $\ep \to 0$ yields~\eqref{e.climb1}. On the other hand, the maximality of $m_\mu(\cdot,\Se_{\mu,t} )$ yields
\begin{equation*}\label{}
m_\mu(\cdot,\Se_{\mu,t} ) \geq m_\mu(\cdot,0 ) - \sup_{z\in \partial \Se_{\mu,t}}  m_\mu(z,0 ) -L_\mu \quad \mbox{in} \ \Rd
\end{equation*}
since the right side is a subsolution which is nonpositive on $\Se_{\mu,t}+B_1$ thanks to~\eqref{control2}. We conclude that
\begin{equation}\label{e.climb2}
\inf_{z\in \partial \Se_{\mu,t}} m_\mu(z,0 )\leq  m_\mu(\cdot,0 ) - m_\mu(\cdot,\Se_{\mu,t} ) \leq  \sup_{z\in \partial \Se_{\mu,t}}  m_\mu(z,0 ) +L_\mu \quad \mbox{in} \ \Rd\setminus \Se_{\mu,t}
\end{equation}
Next we observe that, according to~\eqref{e.ReSeRe}, the Lipschitz estimate and the fact that $m_\mu(\cdot,0 ) \equiv t$ on $\partial \Re_{\mu,t}$, we have  
\begin{equation}\label{e.climb3}
\sup_{z\in \partial\Se_{\mu,t}} \left| m_\mu(z,0 ) -  t \right| \leq 7L_\mu.
\end{equation}
Combining~\eqref{e.climb2} and~\eqref{e.climb3} gives
\begin{equation*}\label{}
\left| m_\mu(\cdot,0 )-\left(t+m_\mu(\cdot,\Se_{\mu,t} )\right)\right| \leq 8L_\mu \quad \mbox{in} \ \Rd\setminus \Se_{\mu,t},
\end{equation*}
which yields the lemma.
\end{proof}

We next give the rigorous justification of~\eqref{e.crutid1}, which asserts that $m_\mu(y,0 ) \indc_{\{ y\in \Se_{\mu,t}\}}$ is nearly $\F_{\mu,t}$--measurable. 
\begin{lem} \label{l.frumpies}
For every $y\in \Rd$ and $0 < s \leq t$, 
\begin{equation}\label{e.frumpies}
\left| m_\mu(y,0 )\indc_{\{ y\in \Se_{\mu,s}\}} -
\E\big[ m_\mu(y,0 )\, |\, \F_{\mu,t}\big]\indc_{\{ y\in \Se_{\mu,s}\}} \right|\leq 18L_\mu.
\end{equation}
\end{lem}
\begin{proof}
Fix $y\in \Rd$, $0 < s \leq t$ and define a random variable $Z$ by
\begin{equation*}\label{}
Z(\omega):= \sum_{i\in \N,\, y\in K_i} m_\mu^{K_i'}(y ) \indc_{ E_i(s) } (\omega). 
\end{equation*}
It is clear that $m_\mu^{K_i'}(y )$ is $\mathcal{G}(K_i'')$--measurable. According to Lemma~\ref{l.loopers}, $E_i(s) \in \F_{\mu,s}$, and hence $Z$ is $\F_{\mu,s}$--measurable by the definition~\eqref{e.defFmut} of the filtration. As
\begin{equation*}\label{}
\left\{\omega\in \Omega\,:\, y\in \Se_{\mu,s} \right\} = \bigcup_{i\in \N,\, y\in K_i} E_i(s),
\end{equation*}
we have, from~\eqref{e.keylocal1}, that for every $\omega\in \Omega$,
\begin{align*}
\left| Z(\omega)  -  m_\mu(y,0 ) \indc_{\{ y\in \Se_{\mu,s}\}}(\omega) \right| \leq \sum_{i\in \N,\ y\in K_i} \left| m_\mu^{K_i'}(y ) - m_\mu(y,0 ) \right|\indc_{ E_i(s) } (\omega) \leq 9L_\mu.
\end{align*}
Using~\eqref{e.Semeasu} and that $Z$ is $\F_{\mu,t}$--measurable, we find that
\begin{align*}\label{}
\lefteqn{ \left| m_\mu(y,0 )\indc_{\{ y\in \Se_{\mu,s}\}} -
\E\big[ m_\mu(y,0 )\, |\, \F_{\mu,t}\big]\indc_{\{ y\in \Se_{\mu,s}\}} \right| } \qquad \qquad & \\
& \leq \left|  m_\mu(y,0 )\indc_{\{ y\in \Se_{\mu,s}\}}
-Z\right|
+ \left| \E\big[ Z\ |\ \F_{\mu,t}\big] -
\E\big[ m_\mu(y,0 )\ |\ \F_{\mu,t}\big]\indc_{\{ y\in \Se_{\mu,s}\}} \right|  \\
& = \left|  m_\mu(y,0 )\indc_{\{ y\in \Se_{\mu,s}\}}
-Z\right|
+  \left|\E\left[ Z -m_\mu(y,0 )\indc_{\{ y\in \Se_{\mu,s}\}}  \, \big\vert \, \F_{\mu,t}\right]\right|\\
& \leq  2 \sup_{\Omega} \left|  m_\mu(y,0 )\indc_{\{ y\in \Se_{\mu,s}\}}
-Z\right| \\
& \leq 18L_\mu,
\end{align*}
as desired.
\end{proof}

As a consequence of Lemma~\ref{l.frumpies}, we obtain~\eqref{e.keyinc2}, albeit with a small error: the statement is that for every~$T\geq L_\mu(|y|-1)$,
\begin{equation}\label{e.frumpies1}
\left|m_\mu(y,0 )-\E\big[m_\mu(y,0 )\, |\, \F_{\mu,T}\big]\right|\leq 18L_\mu.
\end{equation}
Here is the proof: by the definition of $F_i(t)$, we have
\begin{equation*}\label{}
\omega\in F_i(t) \quad \mbox{implies that} \quad \left\{ y\in K_i' \,:\, m^{K_i'}_\mu(y ) \leq t \right\} \subseteq K_i. 
\end{equation*}
In light of~\eqref{e.grthUwup} and the definition of $\Se_{\mu,t}$, we find that 
\begin{equation*}\label{}
T \geq L_\mu(R-1) \quad \mbox{implies that} \quad \overline B_R \subseteq \Se_{\mu,T}.
\end{equation*}
Hence $T\geq L_\mu(|y|-1)$ implies that $\indc_{\{ y\in \Se_{\mu,T}\}} \equiv 1$, and so~\eqref{e.frumpies1} follows from~\eqref{e.frumpies}. 

\smallskip

The next lemma provides the rigorous justification of~\eqref{e.crutid2}.

\begin{lem} \label{l.gumpers}
For every $t>0$,
\begin{equation}\label{e.gumpers}
0 \leq \E \big[ m_\mu(y,\Se_{\mu,t} ) \,\big\vert\, \F_{\mu,t} \big] - \sum_{i\in \N} \E \left[ m_\mu(y,\tilde K_i )\right] \indc_{E_i(t)} \leq L_\mu(a_\mu+13).
\end{equation}
\end{lem}
\begin{proof}
\emph{Step 1.} We first show that, for every $i\in \N$ and $t>0$,
\begin{equation}\label{e.lumpies}
A \in \F_{\mu,t} \qquad \mbox{implies that} \qquad A \cap E_i(t) \in \G( K_i'' + \overline B_3).
\end{equation}
By the definition of $\F_{\mu,t}$, it suffices to fix $0<s\leq t$, $j\in \N$, $G\in \mathcal G(K_j'')$ and to take $A = G \cap E_j(s)$. Then 
\begin{equation*}
A \cap E_i(t) = G \cap E_j(s) \cap E_{i}(t).
\end{equation*}
By~\eqref{e.loopers} and~\eqref{e.fishfishfish}, we see that either this set is empty or else $K_j' \subseteq K_i''$and $E_j(s) \in \mathcal G(K_i'')$. This also yields that $G\in \mathcal G(K_j'') \subseteq \mathcal G(K_i''+\overline B_3)$ and $E_j(s) \cap E_i(t) \in \mathcal G(K_i'')$, and therefore we obtain $A\cap E_i(t) \in  \mathcal G(K_i''+ \overline B_3)$ as desired.

\smallskip

\emph{Step 2.} We claim that
\begin{equation}\label{e.groscaca}
 \E \left[ m_\mu\big(y,\widetilde K_i \big) \indc_{E_i(t)}  \, \Big\vert\, \mathcal{F}_{\mu,t}\right] = \E\left[ m_\mu\big(y,\widetilde K_i \big) \right] \indc_{E_i(t)}.
\end{equation}
By the definition of conditional expectation, we must show that, for every $A\in \F_{\mu,t}$,
\begin{equation*}\label{}
\E \left[ m_\mu(y,\tilde K_i ) \indc_{A\cap E_i(t)} \right] = \E \left[ m_\mu(y,\tilde K_i ) \right] \Prob\Big[\, A \cap E_i(t) \Big].
\end{equation*}
Since $m_\mu(y,\tilde K_i )$ is $\G(\Rd\!\setminus \!\tilde K_i)$-measurable, this follows from~\eqref{e.KtKindp} and~\eqref{e.lumpies}.

\smallskip

\emph{Step 3.} The conclusion, using~\eqref{e.groscaca} and the Lipschitz estimates. Observe that
\begin{equation*}\label{}
\E \big[ m_\mu(y,\Se_{\mu,t} )  \,\big\vert\, \F_{\mu,t} \big]  = \sum_{i\in \N} \E \big[ m_\mu(y,\Se_{\mu,t} ) \indc_{E_i(t)}  \,\big\vert\, \F_{\mu,t} \big]  =  \sum_{i\in \N} \E \big[ m_\mu(y,K_i ) \indc_{E_i(t)}  \,\big\vert\, \F_{\mu,t} \big]
\end{equation*}
and, according to the Lipschitz estimates \eqref{lips},
\begin{equation*}\label{}
0 \leq \left(  m_\mu(y,K_i ) - m_\mu(y,\tilde K_i )\right) \indc_{E_i(t)}  \leq  L_\mu (a_\mu+13) \indc_{E_i(t)}.
\end{equation*}
Combining the previous two lines and applying~\eqref{e.groscaca} yields~\eqref{e.gumpers}.
\end{proof}

\subsection{The fluctuations estimate}
\label{ss.fluc}

We now present the proof for Proposition~\ref{p.fluc}, which follows the heuristic argument given in Subsection~\ref{ss.heur}.

\begin{proof}[{Proof of Proposition~\ref{p.fluc}}]
We break the proof into three steps. Throughout, $C$ denotes a constant which may change from line to line and depends only on $(q,\Lambda,\mu_0)$.

\smallskip

\emph{Step 1}. The application of Azuma's inequality.
Fix $y\in \Rd$ and define a $\F_{\mu,t}$--adapted martingale $\{ X_t \}_{t\geq 0}$ by
\begin{equation}\label{}
X_t : = \E \big[ m_\mu(y,0 ) \, \vert \, \mathcal{F}_{\mu,t} \big] - M_\mu(y).
\end{equation}
Here $\{ \F_{\mu,t}\}$ is the filtration defined in the previous subsection and we recall that $M_\mu(y)$ is defined by~\eqref{e.meandef}. Observe that $X_0 \equiv 0$ and, according to~\eqref{e.frumpies1}, for all $t\geq L_\mu|y|$,
\begin{equation} \label{e.deritx}
\left| X_t(\omega) - \left(m_\mu(y,0 ) - M_\mu(y)\right) \right| \leq 18L_\mu.
\end{equation}
The main step in our argument is to show
\begin{equation}\label{e.stepbnd}
\esssup_{\omega\in\Omega} \left|X_t(\omega) - X_s(\omega) \right|  \leq (2a_\mu+98)L_\mu + \frac{2L_\mu}{l_\mu}|t-s|.
\end{equation}
We admit~\eqref{e.stepbnd} for a moment and use it to complete the proof of the proposition. Azuma's inequality applied to the discrete martingale sequence $Y_n:= X_{\alpha n}$, for $n\in \N$ and with $\alpha:=  l_\mu (a_\mu+49)$ yields, in light of~\eqref{e.stepbnd},
\begin{equation*}\label{}
\P\big[ \left|Y_n \right|>\lambda  \big] \leq \exp\left( - \frac{\lambda^2}{8(2a_\mu+98)^2L_\mu^2 n}\right).
\end{equation*}
Note that by~\eqref{e.amubnd} we have $l_\mu a_\mu \leq C\mu^{-2}$. Take $n:= \left\lceil \frac{L_\mu|y|}{\alpha} \right\rceil +1$ and observe that for $|y| \geq  C\mu^{-2} \geq Cl_\mu(a_\mu+1)$, we have
\begin{equation*}\label{}
n = \left\lceil \frac{ 2L_\mu|y|}{ l_\mu(2a_\mu+98)} +1\right\rceil  \leq \frac{C|y|}{l_\mu(a_\mu+1)}.
\end{equation*}
The above estimate, together with~\eqref{e.deritx}, yields~\eqref{e.fluc} for every $\lambda \geq C$.

\smallskip

We have left to prove~\eqref{e.stepbnd}, which is a consequence of the following two inequalities:
\begin{equation}\label{e.stepbone}
 |X_t-X_s| \leq \big| \E\left[  m_\mu(y,\Se_{\mu,t} ) \, \vert \, \mathcal{F}_{\mu,t} \right] - \E\left[ m_\mu(y, \Se_{\mu,s} ) \, \vert \, \mathcal{F}_{\mu,s} \right] \big| + \frac{L_\mu}{l_\mu}|s-t| + 62 L_\mu 
\end{equation}
and
\begin{equation}\label{e.stepbrain}
\big| \E\left[  m_\mu(y,\Se_{\mu,t} ) \, \vert \, \mathcal{F}_{\mu,t} \right] - \E\left[ m_\mu(y, \Se_{\mu,s} ) \, \vert \, \mathcal{F}_{\mu,s} \right] \big| \leq \frac{L_\mu}{l_\mu}|s-t| + (2a_\mu+36)L_\mu.
\end{equation}
These are proved in the next two steps.

\smallskip

\emph{Step 2.} The proof of~\eqref{e.stepbone}.
For every $0< s \leq t$, we have $\{ y\in \Se_{\mu,s}\} \in \mathcal{F}_{\mu,s}\subseteq \mathcal{F}_{\mu,t}$ and hence 
\begin{equation} \label{e.Xtexp}
X_t = \E\big[ m_\mu(y,0 ) \, \vert \, \mathcal F_{\mu,t} \big]\indc_{\{ y\in \Se_{\mu,s} \}}+
\E\big[ m_\mu(y,0 )\indc_{\{ y\notin \Se_{\mu,s} \}} \, \vert \, \mathcal F_{\mu,t} \big]- M_\mu(y)
\end{equation}
and
\begin{equation} \label{e.Xsexp}
X_s = \E\big[ m_\mu(y,0 ) \, \vert \, \mathcal F_{\mu,s} \big]\indc_{\{ y\in \Se_{\mu,s} \}}+
\E\big[ m_\mu(y,0 )\indc_{\{ y\notin \Se_{\mu,s} \}} \, \vert \, \mathcal F_{\mu,s} \big]- M_\mu(y).
\end{equation}
According to~\eqref{e.frumpies}, we have 
\begin{equation*}
\left| \E\big[ m_\mu(y,0 ) \, \vert \, \mathcal F_{\mu,t} \big]\indc_{\{ y\in \Se_{\mu,s} \}}
-\E\big[ m_\mu(y,0 ) \, \vert \, \mathcal F_{\mu,s} \big]\indc_{\{ y\in \Se_{\mu,s} \}}\right|\leq 36L_\mu.
\end{equation*}
Therefore, subtracting~\eqref{e.Xsexp} from~\eqref{e.Xtexp}, we obtain
\begin{equation*}
\left|X_t - X_s\right| \leq   \left|\E\big[ m_\mu(y,0 )\indc_{\{ y\notin \Se_{\mu,s} \}} \, \vert \, \mathcal F_{\mu,t} \big]
- \E\big[ m_\mu(y,0 )\indc_{\{ y\notin \Se_{\mu,s} \}} \, \vert \, \mathcal F_{\mu,s} \big]\right|+36L_\mu. 
\end{equation*}
Using~\eqref{e.dpp} to estimate the right side of the above inequality, we find
\begin{equation}\label{e.kittens}
\left|X_t - X_s\right| \leq  \left|\E\big[ m_\mu(y,\Se_{\mu,s} )  \, \vert \, \mathcal F_{\mu,t} \big]
- \E\big[ m_\mu(y,\Se_{\mu,s} )\, \vert \, \mathcal F_{\mu,s} \big]\right|\indc_{\{ y\notin \Se_{\mu,s}\}}+ 52L_\mu.
\end{equation}
According to~\eqref{e.ReSeRe} and~\eqref{e.movefron}, 
\begin{align}\label{e.twofronts}
\dist_H(\Se_{\mu,t}, \Se_{\mu,s}) & \leq \dist_H(\Se_{\mu,t},\Re_{\mu,t}) + \dist_H(\Re_{\mu,t},\Re_{\mu,s}) + \dist_H(\Re_{\mu,s},\Se_{\mu,s}) \\ & \leq  
\frac{|s-t|}{l_\mu} +10. \nonumber
\end{align}
By the Lipschitz estimate \eqref{lips}, this yields
\begin{equation*}
\left|m_\mu(y,\Se_{\mu,t} )- m_\mu(y,\Se_{\mu,s} )\right|\leq   \frac{L_\mu}{l_\mu}|s-t| +10L_\mu.
\end{equation*}
After inserting this into \eqref{e.kittens}, we get \eqref{e.stepbone}.

\smallskip

\emph{Step 3.} The proof of~\eqref{e.stepbrain}.
We use the discrete approximation constructed in the previous subsection: by Lemma~\ref{l.gumpers}, we have
\begin{align*}\label{}
\lefteqn{   \left| \E\big[ m_\mu(y, \Se_{\mu,t} ) \, \big\vert \, \mathcal{F}_{\mu,t} \big] - \E\big[ m_\mu(y, \Se_{\mu,s} ) \, \big\vert \, \mathcal{F}_{\mu,s} \big] \right|} \qquad \qquad &    \\ 
& \leq \bigg| \sum_{i\in \N} \E \left[ m_\mu(y,\tilde K_i )\right] \indc_{E_i(t)} - \sum_{j\in \N} \E \left[ m_\mu(y,\tilde K_j )\right] \indc_{E_j(s)} \bigg| \\
& \qquad + 2L_\mu(1+a_\mu+12) \\
& \leq  \sum_{i,j\in \N} \E \left[ \left| m_\mu(y,\tilde K_i ) - m_\mu(y,\tilde K_j )\right| \right]\indc_{E_i(t) \cap E_j(s)} + 2L_\mu(a_\mu+13).
\end{align*}
Next we recall from~\eqref{e.twofronts} that
\begin{equation*}\label{}
E_i(t) \cap E_j(s) \neq \emptyset \quad \mbox{implies that} \quad \dist_H(\widetilde K_i,\widetilde K_j) \leq \dist_H(K_i,K_j) \leq \frac{|s-t|}{l_\mu} + 10.
\end{equation*}
Using the Lipschitz estimate \eqref{lips}, this yields
\begin{multline*}\label{}
\sum_{i,j\in\N} \E \left| m_\mu(y,\tilde K_i ) - m_\mu(y,\tilde K_j )\right|\indc_{E_i(t) \cap E_j(s)} \\ \leq  \left( \frac{L_\mu}{l_\mu}|s-t| + 10L_\mu\right) \sum_{i,j\in\N} \indc_{E_i(t) \cap E_j(s)}= \frac{L_\mu}{l_\mu}|s-t| + 10L_\mu.
\end{multline*}
This completes the proof of~\eqref{e.stepbrain}.
\end{proof}

\section{Estimate of the statistical bias}\label{s.bias}

The main result of this section is an estimate of the nonrandom error, that is, for the difference between $\E \left[ m_\mu(y,0 ) \right]$ and $\overline m_\mu(y)$ for  $|y| \gg 1$.  

\begin{prop}\label{p.bias}
Fix $\mu_0\geq 1$. There exists $C>0$, depending only on $(d,q,\Lambda,\mu_0)$ such that, for every $y\in \R^d$ and $0<\mu\leq \mu_0$,
\begin{equation}\label{meanseq}
M_\mu(y) \leq \overline m_\mu(y) +  C\left( \frac{|y|^{\frac23}}{\mu^2} + \frac{|y|^\frac13}{\mu^{4}} \right)  \log\left(2+\frac{|y|}{\mu}\right).
\end{equation}
\end{prop}

The section is devoted to the proof of this result. Before we begin, we first show that Propositions~\ref{p.fluc} and~\ref{p.bias} imply Theorem \ref{mpEE}.

\begin{proof}[Proof of Theorem \ref{mpEE}.] The estimate~\eqref{hard} is a straightforward consequence of Proposition~\ref{p.fluc} and Proposition~\ref{p.bias} if $|y|\geq C\mu^{-2}$. If $|y|< C\mu^{-2}$, the result also holds for $\lambda\geq 2CL_\mu \mu^{-2}=C\mu^{-2}$ because, from the Lipschitz estimates:
$$
m_\mu(y,0 ) - \overline m_\mu(y)\leq 2L_\mu |y| \leq 2CL_\mu \mu^{-2} \leq \lambda.
$$
Let us now check estimate~\eqref{easy}. Owing to the stationary of $m_\mu$ and its subadditivity (see~\eqref{subadd}), we have
\begin{align*}
M_\mu(y+z)  = & \E \left[ m_\mu(y+z,0 ) \right] \leq \E\left[ m_\mu(y,0 )\right] + \E\left[ m_\mu(y+z,y )\right] +L_\mu\\ 
=& M_\mu(y) + M_\mu(z)+L_\mu.
\end{align*}
Hence $M_\mu(\cdot)+L_\mu$ is a subadditive quantity and we have, for every $y\in \Rd$,
\begin{equation}\label{tyi}
M_\mu(y)+L_\mu \geq \inf_{t\geq1} t^{-1} (M_\mu(ty)+L_\mu) = \lim_{t\to \infty} t^{-1} (M_\mu(ty)+L_\mu) = \overline m_\mu(y).
\end{equation}
The estimate~\eqref{easy}  is immediate from~\eqref{tyi} and~\eqref{e.fluc}  if $|y|\geq C\mu^{-2}$. If $|y|< C\mu^{-2}$ but $\lambda\geq C\mu^{-2}$, the estimate also thanks to the same argument as above. 
\end{proof}

\subsection{Introduction of the approximating quantity}
Throughout the rest of this section, we fix $\mu_0\geq 1$ and $0< \mu \leq \mu_0$ and let $C$ and $c$ denote positive constants which may depend on $(q,\Lambda,\mu_0)$, but vary in each occurrence. 

\smallskip

The proof of Proposition~\ref{p.bias} is based on the introduction of an approximately \emph{superadditive} quantity which approximates $M_\mu(y)$. Since the latter is nearly subadditive, we deduce an estimate (which depends on the quality of the approximation) for the difference of $M_\mu(y)$ and $\overline m_\mu(y)$. The key idea, which goes back to Alexander~\cite{A}, is that $t\to \E\left[m_\mu(H_t,0 )\right]$ is almost a superadditive quantity (here $H_t$ is a plane defined below). However, we cannot use this quantity directly, and must settle for a further approximation. In this subsection, we introduce the relevant approximating quantity and make some preliminary estimates. 

We fix a unit direction $e \in \partial B_1$. For convenience, we assume~$e=e_\d:=(0,\ldots,0,1)$. For each $t> 0$, define the plane
\begin{equation}\label{}
H_t := te + \{ e \}^\perp = \left\{ (x',t) \, : \, x'\in \R^{\d-1} \right\}
\end{equation}
and the discrete version
\begin{equation}\label{}
\widehat H_t : = \left\{ (n,t) \, : \, n\in \Z^{\d-1} \right\}.
\end{equation}
We also denote, for $t> 0$, the halfspaces
\begin{equation}\label{}
H_t^+=\left\{(x',x_d)\in \Rd \, : \, x_d\geq t \right\} \qquad \mbox{and} \qquad H_t^-=\left\{(x',x_d)\in \Rd \, : \, x_d\leq t \right\}.
\end{equation}
For $y\in H^-_t$ we set 
$$
m_\mu(H_t,y ):=\min_{z\in H_t} m_\mu(z,y )= \min\left\{ s\geq 0, \; \Re_{\mu,s}(y)\cap H_t\neq \emptyset\right\}. 
$$
Define, for each $\sigma,t>0$, the quantities
\begin{equation}\label{defgt}
G_{\mu,\sigma} (t):= \sum_{y\in \widehat{H}_t}  \E \left[ \exp\left( -\sigma m_\mu(y,0 ) \right) \right] \quad \mbox{and} \quad g_{\mu,\sigma} (t):= - \frac{1}{\sigma} \log G_{\mu,\sigma}(t).
\end{equation}
We will see below in~Lemma~\ref{goodappr} that $g_{\mu,\sigma}(t)$ is a good approximation of~$\E\left[m_\mu(H_t,0 )\right]$, at least for appropriate choice of the parameter~$\sigma$. 

\smallskip

We first show that we can restrict the sum for $G_{\mu,\sigma} (t)$ to a finite number of indices and still obtain a good approximate of $G_{\mu,\sigma}$. This follows from~\eqref{control2}, which implies that far away points cannot make up a large proportion of the sum in the definition of $G_{\mu,\sigma}$. As the argument is nearly identical to that of~\cite[Lemmas~5.2]{ACS}, we omit the proof. 

\begin{lem} \label{chop}
There exists $C> 0$ such that, for each $t\geq 1$, $0< \sigma \leq 1$ and $R\geq (L_\mu/l_\mu)t$,
\begin{equation}\label{chopin}
G_{\mu,\sigma}(t) \leq C\sigma^{1-\d}  \sum_{y\in \widehat H_t \cap B_R }  \E\left[ \exp\left( -\sigma m_\mu(y,0 ) \right) \right].  
\end{equation}
\end{lem}

We next show that $g_{\mu,\sigma}(t)$ well--approximates~$\E \left[ m_\mu(H_t,0)\right]$. 

\begin{lem} \label{gooda}
There exists $C> 0$ such that, for every $t>C\mu^{-2}$ and $0 < \sigma \leq 1$,
\begin{equation}\label{goodappr}
\E \left[ m_\mu(H_t,0 ) \right] - C \left( \frac{\sigma  t}{\mu^5} + \frac1\sigma \log \left(2+ \frac t{\sigma\mu} \right) \right)\leq g_{\mu,\sigma}(t)  \leq \E\left[ m_\mu(H_t,0 ) \right] + C.
\end{equation}
\end{lem}

\begin{proof} 
The upper bound in~\eqref{goodappr} is relatively easy and is essentially the same as in the proof of the analogous bound in~\cite[Lemma 5.3]{ACS}. 

\smallskip

To obtain the lower bound, we use both \eqref{e.fluc} and \eqref{chopin}. 
We have
\begin{align*}
\lefteqn{\E \left[ \exp\left( -\sigma m_\mu(y,0  \right) \right] = \int_0^\infty \sigma \exp(-\sigma s) \Prob \left[ m_\mu(y,0 ) \leq s \right] \, ds} \qquad \qquad & \\
& \leq \exp\left( -\sigma M_\mu(y) \right) + \int_0^{M_\mu(y)} \sigma\exp(-\sigma s) \Prob \left[ m_\mu(y,0 ) \leq s \right] \, ds\\
& = \left( 1 + \int_0^{M_\mu(y)} \sigma\exp(\sigma \lambda) \Prob\left[ m_\mu(y,0 ) - M_\mu(y) \leq -\lambda \right]\, d\lambda \right)\exp\left( -\sigma M_\mu(y) \right).
\end{align*}
Applying \eqref{e.fluc} for $|y| >C\mu^{-2}$ and using $\sigma\in (0, 1]$, we obtain
\begin{equation*}\label{asdf}
\E \left[ \exp\left( -\sigma m_\mu(y,0  \right) \right] \leq \left( 1 + \int_0^{M_\mu(y)} \exp\left(\sigma \lambda - \frac{\mu^4 \lambda^2}{C |y|}\right) \, d\lambda \right)\exp\left( -\sigma M_\mu(y) \right).
\end{equation*}
We estimate the integrand above by 
\begin{equation*}\label{}
\sigma \lambda - \frac{\mu^4 \lambda^2}{C|y|} = -\frac{\mu^4}{C|y|} \left( \lambda - \frac{\sigma C |y|}{2\mu^4} \right)^2 + \frac{1}{4\mu^4} \sigma^2 C|y| \leq \frac{1}{4\mu^4} \sigma^2 C|y|
\end{equation*}
and so we get
\begin{equation}\label{commybn}
\E \left[ \exp\left( -\sigma m_\mu(y,0  \right) \right]  \leq \left( 1 + M_\mu(y) \exp\left(\frac{1}{4\mu^4} \sigma^2 C|y| \right) \right)\exp\left( -\sigma M_\mu(y) \right).
\end{equation}
Fix $t\geq C\mu^{-2}$. Summing~\eqref{commybn} over~$y\in \widehat H_t\cap B_R$, taking $R:=(L_\mu / l_\mu)t$ and applying~\eqref{chopin} (which holds because $|y|\geq t\geq C\mu^{-2}$), we get
\begin{align*}
G_{\mu,\sigma}(t) & \leq C\sigma^{1-\d} \sum_{y\in \widehat H_t \cap B_R} \left( 1 + M_\mu(y) \exp\left( \frac{1}{4\mu^4} \sigma^2 C|y|  \right) \right) \exp\left( -\sigma M_\mu(y) \right) \\
& \leq C\sigma^{1-\d} \sum_{y\in \widehat H_t \cap B_R} \left( 1 + M_\mu(y) \exp\left( \frac{1}{4\mu^4} \sigma^2 C|y| \right) \right) \exp\left( -\sigma \E \left[ m_\mu(H_t,0 ) \right]   \right) \\
& \leq C \sigma^{1-\d} R^{\d-1} \exp\left( -\sigma\E\left[ m_\mu(H_t,0 ) \right] \right) \left( 1 + L_\mu t \exp\left( \frac{1}{4\mu^4} \sigma^2 CR\right) \right).
\end{align*}
By~\eqref{lmuLmu} we have $R\leq Ct/\mu$, and thus we deduce that
\begin{equation*}
G_{\mu,\sigma}(t) \leq C t^\d \sigma^{1-\d} \mu^{1-\d} \exp\left( -\sigma\E\left[ m_\mu(H_t,0 ) \right] + \frac{ C \sigma^2 t}{\mu^5}\right).
\end{equation*}
Taking logarithms, dividing by $-\sigma$ and rearranging yields the lemma. 
\end{proof}

In the next lemma, we compare $m_\mu(\cdot, H_t )$ with $\min_{z\in \widehat H_t} m_\mu(\cdot, z )$. In the first-order case, these quantities are equal and so there is nothing to prove (in fact, the first quantity is defined in terms of the second, see~\cite[(3.26)]{ACS}). This is because of the peculiarity that, in the first-order case, the convexity of $H$ ensures that the \emph{minimum} of a family of solutions is a \emph{sub}solution (and hence also a solution). Of course, this is not true in the second-order case, but nevertheless we are able to use the convexity of $H$ to show that a \emph{perturbation} of the second quantity is still a subsolution, and hence the desired estimate follows by the definition of the first. The idea of using convexity to perturb a supersolution into a subsolution is a natural one and was used previously for example in~\cite[Lemma 6.14]{LS2}.

\begin{lem}\label{lem:compamminm} There exists a constant $C>0$ such that, for $t>0$ and $y\in H^+_{t+1}$,  
$$
\min_{z\in \widehat H_t \cap B_{L_\mu|y|/l_\mu}}m_\mu(y,z )\leq m_\mu(y,H_t ) +
C\left(1+\frac{|y|}{\mu}\log\left(2+\frac{|y|}{\mu}\right)\right)^{\frac12}.
$$
\end{lem}

\begin{proof} 
With $\theta > 0$ selected at the end of the argument, we introduce the test function
\begin{equation*}
Z(y) : =  -\frac{1}{\theta}\log\left(\sum_{z\in \widehat H_t} \exp\left( -\theta m_\mu(y,z )\right) \right).
\end{equation*}
The main point of the proof is that~$Z$ is both a good approximation of 
 $\min_{z\in \widehat H_t} m_\mu(\cdot, z )$ as well as (almost) a subsolution of~\eqref{e.metbas} in $H_{t+t}^+$.
 
\emph{Step 1.} We show that $Z$ is a good approximation of~$\min_{z\in \widehat H_t} m_\mu(\cdot, z )$. Arguing as in the proof of Lemma~\ref{chop}, we can show that there exists $C> 0$ such that, for every $t> 0$,  $y\in H_{t+1}^+$ and $R\geq (L_\mu/l_\mu)|y|$,
\begin{align*}
\sum_{z\in \widehat H_t} \exp\left( -\theta m_\mu(y,z )\right)
& \leq  \left(1+ C\theta^{1-\d} \right) \sum_{z\in \widehat H_t \cap B_R }  \exp\left( -\theta m_\mu(y,z ) \right)\\
& \leq R^{\d-1}\left(1+ C\theta^{1-\d} \right) \max_{z\in \widehat H_t \cap B_R }  \exp\left( -\theta m_\mu(y,z ) \right).
\end{align*}
In particular, 
\begin{equation}\label{jhrfgsdjh}
Z(y ) \geq \min_{z\in \widehat H_t  \cap B_R}m_\mu(y,z )  -\frac{1}{\theta} \log \left(R^{\d-1}\left(1+ C\theta^{1-\d} \right)\right).
\end{equation}

\smallskip

\emph{Step 2.} We show that $Z$ is (almost) a subsolution of~\eqref{e.metbas} in $H_t^+$. The claim is that
\begin{equation} \label{e.Zclaim}
-\tr\left( A(y ) D^2Z(y ) \right) +H(DZ(y ),y ) \leq  \mu + 2 \Lambda L_\mu^2\theta \quad \mbox{in} \ H_{t+1}^+.
\end{equation}
As in the proof of Lemma~\ref{e.seige}, we check this inequality assuming that the functions are smooth. A rigorous argument in the viscosity sense follows by performing essentially identical computations on a smooth test function. It is convenient to work with  
\begin{align*}
W(y )=  \exp\{-\theta Z(y )\} = \sum_{z\in \widehat H_t} \exp\left(-\theta m_\mu(y,z ) \right).
\end{align*}
Computing the derivatives of $Z$ in terms of $W$, we find
\begin{align*}
D Z(y) & = -\frac{1}{\theta W(y)} DW(y),  \\
D^2 Z(y) & = -\frac{1}{\theta W(y)} D^2W(y) + \frac{1}{\theta W^2(y)} DW(y) \otimes DW(y),
\end{align*}
and computing the derivatives of $W$ in terms of $m_\mu(\cdot,z)$ yields
\begin{align*}
D W(y) & = -\theta \sum_{z\in \widehat H_t}\exp\left(-\theta m_\mu(y,z ) \right) Dm_\mu(y,z) ,  \\
D^2 W(y) & = -\theta \sum_{z\in \widehat H_t}\exp\left(-\theta m_\mu(y,z ) \right) \left( D^2m_\mu(y,z) - \theta Dm_\mu(y,z) \otimes Dm_\mu(y,z) \right),
\end{align*}
The Lipschitz bound $|Dm_\mu(\cdot,z)|\leq L_\mu$ thus gives the estimates
\begin{align*} \label{}
& \big| DW(y) \big|  \leq \theta L_\mu W(y),\quad  \Big| D^2 W(y) + \theta \sum_{z\in \widehat H_t}\exp\left(-\theta m_\mu(y,z ) \right) D^2m_\mu(y,z) \Big|  \leq \theta^2 L_\mu^2 W(y).
\end{align*}
Assembling these together, we obtain
\begin{multline*}
-\tr\left( A(y ) D^2Z(y ) \right) +H(DZ(y ),y )  
 \\ \leq \frac{1}{\theta W(y)} \tr\left( A(y) D^2 W(y) \right)  + H\left( -\frac{1}{\theta W(y)} DW(y) ,y\right) + \Lambda  L_\mu^2 \theta.
\end{multline*}
Next, we observe that
\begin{multline*} \label{}
\frac{1}{\theta W(y)} \tr\left( A(y) D^2W(y) \right) \\ \leq -\frac{1}{W(y)} \sum_{z\in \widehat H_t}\exp\left(-\theta m_\mu(y,z ) \right) \tr\left( A(y) D^2m_\mu(y,z) \right) + \Lambda L_\mu^2\theta
\end{multline*}
and, by the convexity of $H$ and the definition of $W$,
\begin{align*} \label{}
H\left( -\frac{1}{\theta W(y)} DW(y) ,y\right) &  = H \left( \frac{1}{W(y)}   \sum_{z\in \widehat H_t}\exp\left(-\theta m_\mu(y,z ) \right) Dm_\mu(y,z), y\right)  \\
& \leq \frac{1}{W(y)} \sum_{z\in \widehat H_t}\exp\left(-\theta m_\mu(y,z ) \right) H(Dm_\mu (y,z),y).
\end{align*}
Putting the last three inequalities together, we obtain
\begin{align*}
\lefteqn{-\tr\left( A(y ) D^2Z(y ) \right) +H(DZ(y ),y ) - 2\Lambda L_\mu^2\theta} \qquad & \\ &\leq \frac{1}{W(y)} \sum_{z\in \widehat H_t}\exp\left(-\theta m_\mu(y,z ) \right) \Big( - \tr\left( A(y) D^2m_\mu(y,z) \right) + H(Dm_\mu(y,z),y) \Big) \\
& = \mu.
\end{align*}
This completes the proof of~\eqref{e.Zclaim}.

\smallskip

\emph{Step 3.} The conclusion. Due to~\eqref{e.Zclaim} and the convexity of $H$ and~$H(0,y) \leq 0$, we obtain that $\zeta(y):= (1+2\Lambda L_\mu^2\theta/\mu)^{-1}Z(y )$ is a subsolution of 
\begin{equation}\label{e:ecoco}
-\tr\left( A(y ) D^2\zeta \right) +H(D\zeta,y )
\leq  \mu \quad \mbox{in} \ H^+_{t+1}.
\end{equation}
By Lipschitz estimate on $m_\mu$, we have, for every $y\in H_{t+1}$,  
$$
Z(y )\leq  \min_{z\in \widehat H_t } m_\mu(y,z )\leq L_\mu.
$$
By the maximality of  $m_\mu(\cdot,H_t )$, we obtain:
$$
(1+2\Lambda L_\mu^2\theta/\mu)^{-1}\left(Z(y )- L_\mu)\right)\leq m_\mu(y,H_t ), 
$$
which can be rearranged to read
\begin{align*}
Z(y )\leq (1+2\Lambda L_\mu^2\theta/\mu)m_\mu(y,H_t )+C
\leq m_\mu(y,H_t )+ \frac{2\Lambda L_\mu^3\theta |y|}{\mu}+C.
\end{align*}
Using \eqref{jhrfgsdjh} we get 
$$
\min_{z\in \widehat H_t \cap B_R }m_\mu(y,z )\leq m_\mu(y,H_t ) +
\frac{1}{\theta} \log \left(R^{\d-1}\left(1+ C\theta^{1-\d} \right)\right)+ \frac{\Lambda^2 L_\mu^3\theta |y|}{\mu}+C.
$$
Choosing $R:= (L_\mu/l_\mu)|y|$ and $\theta: = |y|^{-\frac12}\mu^{\frac12}(\log(|y|/\mu))^{\frac12}$ gives the lemma. 
\end{proof}

\subsection{The (almost) superadditivity of $g_{\mu,\sigma}$} 
In this subsection we show that $g_{\mu,\sigma}$ is superadditive, up to a small error, and as a corollary derive a rate of the convergence of the deterministic quantity $t^{-1} \E\left[ m_\mu(H_t,0) \right]$ to its limit $\overline m_\mu(H_1)$. The statements of these assertions appear respectively in Lemmas~\ref{superadd} and~\ref{EHtrate}, below. The arguments for these facts depend on the localization results of Section~\ref{e.justify}, which we first must adapt to the setting here. This is the purpose of Lemma~\ref{lem:mmuPresqMes}.

It is immediate that, for any $y\in H_t^+$, the random variable $m_\mu(y,H_t )$ is ${\mathcal G}(H_{t-1}^+)$--measurable. On another hand, we cannot expect $m_\mu(H_t,0)$ to be ${\mathcal G}(H_t^-)$--measurable for the same reason that $m_\mu(y,0)$ depends on the full $\sigma$--algebra~$\F$, due to the second--order term in our equation. Nevertheless, we argue, using the estimates of Subsection~\ref{e.justify}, that~$m_\mu(H_t,0 )$ is  close to being ${\mathcal G}(H_t^-)-$measurable.

\begin{lem}\label{lem:mmuPresqMes}  For every $t>1$, 
$$
\left| m_\mu(H_t,0 )- \E\big[ m_\mu(H_t,0 )\ |\ {\mathcal G}(H_{t+7}^-)\big]\right| \leq 8L_\mu.
$$
\end{lem}

\begin{proof} 
We introduce the ``stopping time"
\begin{equation*}\label{}
T:= \inf\left\{ s\geq 0 \,:\, \Se_{\mu,s} \cap H_t^+ \neq \emptyset \right\}.
\end{equation*}
Here $\Se_{\mu,s}$ is defined by~\eqref{defSemu} and we recall that it is adapted to the filtration~$\F_{\mu,s}$. 
We claim that
\begin{equation}\label{e.Tone}
T \quad \mbox{is \ \ $\mathcal G(H_{t+7}^-)$--measurable}\end{equation}
and
\begin{equation}\label{e.Ttwo}
\left|T - m_\mu(H_t,0 )\right| \leq 4L_\mu.
\end{equation}
Observe that~\eqref{e.Tone} and~\eqref{e.Ttwo} yield the lemma, since they imply
\begin{align*}\label{}
\lefteqn{ \left| m_\mu(H_t,0 )- \E\big[ m_\mu(H_t,0 )\, |\, \mathcal{G}(H_{t+7}^-)\big]\right|} \qquad \qquad &  && \\
& \leq \left| m_\mu(H_t,0 )- T\right|+ \left| \E\big[ T-m_\mu(H_t,0 )\ |\ {\mathcal G}(H_{t+7}^-)\big]\right| && \mbox{(by~\eqref{e.Tone})}\\
& \leq 8L_\mu. && \mbox{(by~\eqref{e.Ttwo})}
\end{align*}

\emph{Step 1.} The proof of~\eqref{e.Tone}. For each $\omega\in\Omega$,  we let ${\bf i}_s(\omega)$ be the unique index $i\in \N$ such that $\omega\in E_{i}(s)$. Let ${\mathcal J}_t=\{ i\in \N\ : \ K_i'\subseteq {\rm int}(H_t^-)\}$. Then, by definition, $
T(\omega) = \inf\left\{s \geq 0\ : \ {\bf i}_s(\omega)\notin  {\mathcal J}_t\right\}$. Note that 
$$
\{{\bf i}_s\in  {\mathcal J}_t\}= \bigcup_{i\in {\mathcal J}_t} \{ {\bf i}_s=i\}= \bigcup_{i\in {\mathcal J}_t}E_i(s).
$$
Recall  that, by construction, $E_i(s)\in {\mathcal G}(K_i'')$. If $i\in \mathcal J_t$, then we have $K_i''\subseteq H_{t+7}^-$ since $K_i'\subseteq H_t^-$. Therefore $\{{\bf i}_s\in  {\mathcal J}_t\}\in {\mathcal G}(H_{t+7}^-)$, which proves our claim. 

\smallskip

\emph{Step 2.} The proof of~\eqref{e.Ttwo}.
Pick $y\in H_{t+4}$ be such that $\tau:=m_\mu(H_{t+4},0 )= m_\mu(y,0 )$. According to \eqref{e.ReSeRe}, we have $\displaystyle \dist_H( \Re_{\mu,\tau}, \Se_{\mu,\tau})\leq 4$. So there exists $z\in \Se_{\mu,\tau}$ with $|z-y|\leq 4$. As $y\in H_{t+4}$, we have $z\in H_t^+$. Therefore $T\leq \tau$. By~\eqref{lips},  we have
$$
\left|m_\mu(H_{t+4},0 )-m_\mu(H_{t},0 )\right|\leq 4L_\mu,
$$
which yields that $T\leq m_\mu(H_{t},0 )+ 4L_\mu$.
On the other hand, if $x\in \Se_{\mu,T}\cap H_t$, according to~\eqref{e.ReSeRe} there exists $z\in \Re_{\mu,T(\omega)}$ such that $|z-x|\leq 4$. Hence 
$$
m_\mu(H_t,0 )\leq m_\mu(x,0 )\leq m_\mu(z,0 )+4L_\mu \leq T+4L_\mu\;.
$$
This completes the proof of~\eqref{e.Ttwo}. 
\end{proof}

\begin{lem} \label{superadd}
There is a constant $C> 0$ such that, for every $s,t> 1$ and $0< \sigma \leq 1$,
\begin{equation}\label{superaddeq}
g_{\mu,\sigma}(t+s) \geq g_{\mu,\sigma}(t) + g_{\mu,\sigma}(s) - 
C\left(\frac{1}{\sigma}+\frac{(s+t)^{\frac12}}{\mu}\right)\log\left(\frac{s+t}{\sigma\mu}\right).
\end{equation}
\end{lem}
\begin{proof} 
Fix $s,t> 1$. We first claim that, for any $y\in H_{t}^+$, 
\begin{equation}\label{ineq:mmugeqmmu}
m_\mu(y,0 ) \geq m_\mu (H_t,0 ) +m_\mu(y,H_{t} )-3L_\mu.
\end{equation}
Indeed, observe that the maps $y\to m_\mu(y,0 )$ and $y\to  m_\mu (H_t,0 ) +m_\mu(y,H_{t} )$ are both maximal solutions of \eqref{e.subsol} in $H_{t+1}^+$ with boundary conditions on $H_{t+1}$ respectively equal to $m_\mu(y,0 )$ and $m_\mu (H_t,0 )$. As, for $z\in H_{t+1}$, we have by Lipschitz estimates, 
$$
m_\mu(z,0 )\geq m_\mu(z-e_d,0 )-L_\mu \geq m_\mu (H_t,0 )-L_\mu, 
$$
we obtain by comparison (see \cite{AT})
$$
m_\mu(y,0 ) \geq m_\mu (H_t,0 ) +m_\mu(y,H_{t} )-L_\mu\qquad {\rm in }\; H_{t+1}^+.
$$
This implies \eqref{ineq:mmugeqmmu} again thanks to the Lipschitz estimates.

Using once more the Lipschitz estimate for $m_\mu$ and Lemma \ref{lem:mmuPresqMes}, \eqref{ineq:mmugeqmmu} becomes
\begin{equation*}
m_\mu(y,0 ) \geq  \E\big[m_\mu (H_t,0 )\ |\ {\mathcal G}(H_{t+7}^-)\big] +m_\mu(y,H_{t+8} ) - C.
\end{equation*}
In light of~\eqref{e.frd}, the random variables $\E\big[m_\mu (H_t,0 )\ |\ {\mathcal G}(H_{t+7}^-)\big]$ and $m_\mu(y,H_{t+8} )$ are independent and thus
\begin{multline*}
\E\left[ \exp\left( -\sigma m_\mu(y,0 ) \right) \right] \\
\leq  \exp \left( C\sigma \right) \E\left[ \exp\left(-\sigma  \E\big[m_\mu (H_t,0 )\ |\ {\mathcal G}(H_{t+7}^-)\big] \right)\right]
\E\left[ \exp\left(-\sigma m_\mu(y,H_{t+8} ) \right)\right].
\end{multline*}
Using Lemma \ref{lem:mmuPresqMes} again, we obtain
\begin{multline*}
\E\left[ \exp\left( -\sigma m_\mu(y,0 ) \right) \right]
\leq \exp \left( C\sigma \right) \E\left[ \exp\left(-\sigma  m_\mu (H_t,0 )\right)\right]
\E\left[ \exp\left(-\sigma m_\mu(y,H_{t+8} ) \right)\right].
\end{multline*}
Returning to the discrete setting, we have, for $R:=L_\mu(s+t)/l_\mu$, and thanks to the Lipschitz estimates:
\begin{equation}\label{hbecahb}
m_\mu(H_t,0 ) \geq \min_{z\in \hat{H}_t\cap B_R}  m_\mu (z,0 )-  L_\mu (d-1)^{\frac12}.
\end{equation}
On another hand Lemma \ref{lem:compamminm} implies that, for $y\in \widehat H_{t+s}\cap B_R$,   
\begin{align*}
m_\mu(y,H_{t+8} ) & \geq \min_{z\in \hat{H}_{t+8}} m_\mu(y,z )- C\left(1+\frac{R}{\mu}\log\left(2+\frac{R}{\mu}\right)\right)^{\frac12}\\
& \geq  \min_{z\in \hat{H}_{t}} m_\mu(y,z )- C\left(1+\frac{R}{\mu}\log\left(2+\frac{R}{\mu}\right)\right)^{\frac12}.
\end{align*}
Combining these inequalities, we obtain
\begin{multline*}
\E\left[ \exp\left(-\sigma m_\mu(y,0 ) \right) \right] \\
\leq \exp\left(C\sigma \left(1+\tfrac{R}{\mu}\log\left(2+\tfrac{R}{\mu}\right)\right)^{\frac12}\right) 
\!\!\sum_{ z, z'\in \hat{H}_t} \E\left[ \exp\left(-\sigma  m_\mu (z,0 )\right)\right]
\E\left[ \exp\left(-\sigma  m_\mu(y,z' ) \right) \right].
\end{multline*}
Note that, if $y\in \widehat H_{t+s}$ and  $z'\in \hat{H}_t$, then $y-z'\in \hat{H}_s$. So, 
in view of the definition of $G_{\mu, \sigma}$ and the stationarity of $m_\mu$, we have
\begin{equation*}
\sum_{ z'\in \hat{H}_t} 
\E\left[ \exp\left(-\sigma  m_\mu(y,z' ) \right) \right]=
\sum_{ z'\in \hat{H}_t} 
\E\left[ \exp\left(-\sigma  m_\mu(y-z',0 ) \right) \right]=
G_{\mu,\sigma}(s).
\end{equation*}
Therefore
\begin{equation*}
\E\left[ \exp\left(-\sigma m_\mu(y,0 ) \right) \right]
\leq \exp\left(C\sigma \left(1+\frac{R}{\mu}\log\left(2+\frac{R}{\mu}\right)\right)^{\frac12}\right) G_{\mu, \sigma}(t) G_{\mu,\sigma}(s).
\end{equation*}
Summing over all $y\in \hat H_{t+s} \cap B_{R}$  and using Lemma~\ref{chop} yields
\begin{align*}
G_{\mu, \sigma}(s+t) \leq C R^{\d-1}\sigma^{1-\d} \exp\left(C\sigma \left(1+\frac{R}{\mu}\log\left(2+\frac{R}{\mu}\right)\right)^{\frac12}\right) G_{\mu, \sigma}(t) G_{\mu, \sigma}(s).
\end{align*}
Taking logarithms and dividing by $\sigma$ concludes the proof.
\end{proof}

The rest of this section follows~\cite{ACS}, the main differences being the values of the constants. We next use Hammersley-Fekete lemma to obtain a rate of convergence for the means $t^{-1} \E \left[ m_\mu(H_t,0 ) \right]$ to their limit $\overline m_\mu(H_t)$.

\begin{lem}\label{EHtrate}
There exists a constant $C>0$ such that, for every $t\geq 0$,
\begin{equation} \label{EHtrateq}
\E \left[ m_\mu( H_t,0 ) \right] \leq \overline m_\mu(H_t) +  C \left(\frac{t}{\mu^5}\right)^{\frac12}\log \left(2+\frac t\mu\right).
\end{equation}
\end{lem}
\begin{proof}
According to Lemma~\ref{superadd}, the quantity $g_{\sigma,\mu}$ is almost superadditive. More precisely, for all $s,t> 0$, we have
\begin{equation}\label{ineq:hammer}
g_{\sigma,\mu}(s+t)\geq g_{\sigma,\mu}(s)+g_{\sigma,\mu}(t) - \Delta_{\sigma,\mu}(s+t),
\end{equation}
where
\begin{equation*}
\Delta_{\sigma,\mu}(t) := C\left(\frac{1}{\sigma}+\frac{\sqrt{t}}{\mu}\right)\log\left(2+\frac{t}{\sigma\mu}\right).
\end{equation*}
Since $\Delta_{\sigma,\mu}$ is increasing on $[1,\infty)$ and
\begin{equation*}\label{}
\int_1^\infty \frac{\Delta_{\sigma,\mu}(t)}{t^2}\, dt < \infty,
\end{equation*}
we may apply Hammersley-Fekete lemma to deduce that $\overline g_{\sigma,\mu} : = \lim_{t\to \infty} g_{\sigma,\mu}(t)/t$ exists and,  for every $t> 1$,
\begin{equation*}
\frac{1}{t}g_{\sigma,\mu}(t) - 4 \int_{2t}^\infty \frac{\Delta_{\sigma,\mu}(s)}{s^2}\, ds \leq \overline g_{\sigma,\mu}.
\end{equation*}
An easy integration by parts yields
\begin{equation*}
4\int_{2t}^\infty \frac{\Delta_{\sigma,\mu}(s)}{s^2}\, ds \leq C\left( \frac{1}{\sigma t}+ \frac{1}{\mu \sqrt{t}}\right) \log\left(2+\frac{t}{\sigma\mu}\right).
\end{equation*}
Assume now that $t\geq C\mu^{-2}$ so that, in view of the second inequality in \eqref{goodappr}, we have 
$$
\overline g_{\sigma,\mu} \leq \lim_{t\to +\infty}\frac{1}{t} \left( \E\left[ m_\mu(H_t,0 ) \right]  + C\right),
$$
where 
 $$
 \lim_{t\to +\infty}\frac{1}{t} \E\left[ m_\mu(H_t,0 ) \right] = \overline m_\mu(H_1)
$$
from the growth estimate \eqref{control2} and the locally uniform convergence of $z\to m_\mu(t z,0 )/t$ to $\overline m_\mu(z)$ as $t\to+\infty$.
Combining the previous three estimates, we obtain
\begin{equation*}
\frac{1}{t} g_{\sigma,\mu}(t)\leq 
\overline m_\mu(H_1) +C\left( \frac{1}{\sigma t}+ \frac{1}{\mu \sqrt{t}}\right) \log\left(2+\frac{t}{\sigma\mu}\right).
\end{equation*}
Multiplying by $t$, applying the first inequality in \eqref{goodappr} and using the positive homogeneity of $\overline m_\mu$ yields, for $t\geq C\mu^{-2}$,
\begin{equation*}
\E \left[ m_\mu(H_t,0 ) \right] \leq \overline m_\mu(H_t) + C\left( \frac{\sigma t}{\mu^5} +\left( \frac1\sigma+ \frac{ \sqrt{t}}{\mu}\right)
  \log\left( 2+\frac t{\sigma\mu} \right) \right),
\end{equation*}
and taking $\sigma:= \mu^{\frac52} t^{-\frac12}(\log(2+t/\mu))^{\frac12}$ completes the proof of \eqref{EHtrateq} when $t\geq C\mu^{-2}$.  If $t< C\mu^{-2}$, then 
$$
\E \left[ m_\mu(H_t,0 ) \right] -\overline m_\mu(H_t)\leq  Ct \leq C\left( \frac{t}{\mu^2} \right)^{1/2}
$$
so that \eqref{EHtrateq} holds as well. 
\end{proof}

\subsection{Error estimates for $M_\mu(y)-\overline m_\mu(y)$ and the proof of~\eqref{hard}}
It is the rate of convergence of $t^{-1}M_\mu(ty)$ to $\overline m_\mu(y)$ that we wish to estimate, not that of $t^{-1} \E \left[ m_\mu(H_t,0 ) \right]$ to $\overline m_\mu(H_1)$. In order to reach our desired goal, we must compare the quantities $M_\mu(y)$ and $\E \left[ m_\mu(H_t,0 ) \right]$. This is accomplished in two steps. The first is to show that the quantities $\E\left[ m_\mu(H_t,0 ) \right]$ and
\begin{equation*}\label{}
M_\mu(H_t):= \min_{y\in H_t} M_\mu(y)
\end{equation*}
are close, which then gives an estimate for the difference between $M_\mu(H_t)$ and $\overline m_\mu(H_t)$. The second step is to use elementary convex geometry to relate $M_\mu(y)$ to the values of $M_\mu(H)$ for all the possible planes $H$ passing through $y$.

\begin{lem} \label{cmte-minE}
There exists $C> 0$ such that, for each $t\geq 1$,
\begin{equation}\label{wrestle}
 M_\mu(H_t) \leq \E \left[ m_\mu(H_t,0 ) \right] + C\left( \frac{t}{\mu^5}\log \left(2+\frac t\mu \right)\right)^\frac12.
\end{equation}
\end{lem}
\begin{proof}
Let $R:= (L_\mu/l_\mu)t$. We may choose $\hat z \in \widehat H_t \cap B_R$ such that
\begin{equation*}\label{}
m_\mu(\hat z,0 ) \leq m_\mu(H_t,0 ) + L_\mu (\d-1)^{\frac12}. 
\end{equation*}
For every $z\in H_t$ we have $\E \left[ m_\mu(z,0 ) \right] =M_\mu(z) \geq M_\mu(H_t)$ and thus, for every $\lambda >0$,
\begin{multline}\label{}
\left\{ \omega\in \Omega\,:\, M_\mu(H_t) - m_\mu(H_t,0 ,\omega) \geq \lambda + L_\mu (\d-1)^{\frac12} \right\} \\
\subseteq \bigcup_{z\in \hat H_t\cap B_R} \left\{ \omega\in \Omega\,:\, m_\mu(z,0,\omega ) \leq M_\mu(z) - \lambda \right\}.
\end{multline}
Assuming $t\geq C\mu^{-2}$ and applying~\eqref{e.fluc}, we find
\begin{align}\label{probbnd}
\lefteqn{\Prob\left[ M_\mu(H_t) - m_\mu(H_t,0 )  \geq \lambda + L_\mu(\d-1)^{\frac12} \right]} \qquad \qquad  & \\
& \leq C R^{\d-1} \max_{z\in H_t} \Prob\left[ m_\mu(z,0 ) - M_\mu(z) \leq -\lambda \right] \nonumber \\
& \leq CR^{\d-1} \exp\left( -\frac{\mu^4\lambda^2}{CR} \right) \leq C \mu^{1-\d} t^{\d-1}\exp\left( -\frac{\mu^5\lambda^2}{Ct} \right).\nonumber
\end{align}
We next estimate the right side of the inequality
\begin{equation}\label{spread}
M_\mu(H_t) - \E\left[ m_\mu(H_t,0 ) \right]  \leq \int_0^\infty \Prob\left[ M_\mu(H_t) - m_\mu(H_t,0 ) \geq \lambda  \right]\, d\lambda.
\end{equation}
Fix $\beta\geq 1$, to be selected below, define 
\begin{equation*}\label{}
\lambda_1:= \left( \frac{\beta t}{\mu^5} \log\left( 2  +\frac t\mu\right) \right)^{\frac12}
\end{equation*}
and then estimate the right side of~\eqref{spread} by 
\begin{align*}
\lefteqn{ \int_0^\infty \Prob\left[ M_\mu(H_t) - m_\mu(H_t,0 ) \geq \lambda  \right]\, d\lambda } \qquad  \\ 
& \leq  \lambda_1 + L_\mu(\d-1)^{\frac12} + \int_{\lambda_1}^\infty \Prob\left[ M_\mu(H_t) - m_\mu(H_t,0 ) \geq \lambda + L_\mu (\d-1)^{\frac12} \right]\, d\lambda \\
& \leq \lambda_1 + L_\mu(\d-1)^{\frac12} + C\mu^{1-\d}t^{\d-1}\int_{\lambda_1}^\infty \exp\left(-\frac{\mu^5\lambda^2}{Ct}\right) \, d\lambda.
\end{align*}
Observe that
\begin{multline*}
\mu^{1-\d}t^{\d-1}\int_{\lambda_1}^\infty \exp\left(-\frac{\mu^5\lambda^2}{Ct}\right) \, d\lambda \leq \mu^{1-\d}t^{\d-1} \int_{\lambda_1}^\infty \exp\left( - \frac{\mu^5 \lambda_1\lambda}{Ct} \right) \, d\lambda \\
= C \mu^{1-\d}t^{\d-1} \frac{t}{\mu^5\lambda_1} \exp\left( -\frac{\mu^5\lambda_1^2}{Ct} \right) \leq  C\left( \frac t\mu\right)^{\d+4} \left(1+\frac t\mu\right)^{-\beta/C}.
\end{multline*}
Taking $\beta\geq C$, the last expression on the right is bounded by $C$. The previous two sets of inequalities and~\eqref{spread} yield 
\begin{equation}\label{}
M_\mu(H_t) - \E\left[ m_\mu(H_t,0 ) \right]  \leq \lambda_1 + L_\mu(\d-1)^\frac12 + C \leq \lambda_1+ C,
\end{equation}
which gives, for every $t\geq C\mu^{-2}$.
$$
 M_\mu(H_t) \leq \E \left[ m_\mu(H_t,0 ) \right] +C\left( \frac{t}{\mu^5}\log \left(2+\frac t\mu \right)\right)^\frac12.
$$
If $t< C\mu^{-2}$, then we obtain that 
$$
M_\mu(H_t) - \E \left[ m_\mu(H_t,0 ) \right] \leq Ct \leq C\left(  \frac{t} {\mu^2}\right)^{1/2}. 
$$
Hence \eqref{wrestle} also holds in this case. 
\end{proof}

Lemmas~\ref{EHtrate} and~\ref{cmte-minE} give an estimate on the difference of $M_\mu(H_t)$ and $\overline m_\mu(H_t)$.
\begin{cor}
There exists $C> 0$ such that, for every $t\geq 1$,
\begin{equation}\label{MHtrate}
M_\mu(H_t) \leq \overline m_\mu(H_t) +C\left(\frac t {\mu^5}\right)^{\frac12} \log \left(2+\frac t\mu\right).
\end{equation}
\end{cor}

The previous corollary  yields a rate of convergence for $M_\mu(y)$ to $\overline m_\mu(y)$. 

\begin{proof}[{Proof of Proposition~\ref{p.bias}}] 
The first step is to show that, for every $z\in \Rd$ with $|z|>1$,
\begin{equation}\label{zinCH} 
z \in \conv\left\{ y\in \Rd \,:\, M_\mu(y) \leq \overline m_\mu(z) + C'\left(\frac{|z|}{\mu^5}\right)^{\frac12}\log \left(2+\frac{|z|}{\mu}\right)\right\},
\end{equation}
where $C'>0$ is the constant $C$ in~\eqref{MHtrate}. Suppose on the contrary that ~\eqref{zinCH} fails for some $z\in \Rd$ with $t:=|z|>1$. By elementary convex separation, there exists a plane $H$  with $z\in H$ such that 
\begin{equation*}\label{}
M_\mu(H) > \overline m_\mu(z) + A \quad \mbox{where} \quad A:=  C' \left(\frac t{\mu^5}\right)^{\frac12} \log \left(2+\frac t\mu\right).
\end{equation*}
Since $H$ is at most a distance of $|z|=t$ from the origin, we may assume with no loss of generality that $H=H_s$ for some $s\leq t$. We deduce that
\begin{equation*}\label{}
M_\mu(H_s) > \overline m_\mu(z) + A  \geq \overline m_\mu(H_s) + C'\left(\frac s{\mu^5}\right)^{\frac12} \log \left(2+\frac s\mu\right),
\end{equation*}
a contradiction to~\eqref{MHtrate}. Thus~\eqref{zinCH} holds. 

Next we recall that, according to~\cite[Lemma 4.9]{ACS}, there exists $C''> 0$ such that, for every $N\in \N$ and $\alpha > 0$,
\begin{equation}\label{convhull}
 \conv \left\{ y\in \Rd \, : \, M_\mu(y) \leq \alpha \right\} \subseteq \left\{ y \in \Rd \,: \, M_\mu(Ny) \leq (N+C''/\mu)\alpha \right\}.
\end{equation}
Fix $y\in \Rd$ with $|y|> 1 $ and apply~\eqref{zinCH} to $z:=y/N$, with $N \in\N^*$ to be chosen below, to conclude that 
\begin{equation*}\label{}
y/N \in \conv\left\{ x\in \Rd \,:\, M_\mu(x) \leq \overline m_\mu(y/N) + C'\left( \frac{|y|}{N\mu^5} \right)^{\frac12}\log\left(2+ \frac{|y|}{N\mu}\right)\right\}
\end{equation*}
and, after an application of~\eqref{convhull},
\begin{align*}\label{}
M_\mu(y) & \leq  \left( N + \frac{C''}{\mu} \right) \left( \overline m_\mu(y/N) + C' \left(\frac{|y|}{N\mu^5}\right)^{\frac12} \log\left(2+\frac{|y|}{N\mu}\right)\right) \\
& \leq \overline m_\mu(y) + C''L_\mu \frac{|y|}{N\mu} + 
C'   \left( \frac {N|y|}{\mu^5}\right)^\frac12\log\left(2+\frac{|y|}{N\mu}\right) \\
& \qquad + C'C''  \left( \frac{|y|}{N\mu^7}\right)^\frac12 \log\left(2+\frac{|y|}{N\mu}\right).
\end{align*}
We take $N$ be the smallest integer larger than $|y|^{\frac13}\mu$ to get 
\begin{align*}
M_\mu(y) & \leq \overline m_\mu(y) + C \frac{|y|^{\frac23}}{\mu^2} + 
C \frac{|y|^{\frac23}}{\mu^2}  \log\left(2+\frac{|y|^{\frac23}}{\mu^2}\right)+ C\frac{|y|^{\frac13}}{\mu^{4}}   \log\left(2+\frac{|y|^{\frac23}}{\mu^2}\right).
\end{align*}
This is~\eqref{meanseq}. 
\end{proof}

\section{Error estimates for the approximate cell problem}\label{s.cellpb}

In this section, we prove Theorem \ref{p.deltaGlobal}. As in~\cite{ACS}, the argument consists of controlling $\delta v^\delta(\cdot ;p)+\overline H(p)$ by the difference $m_\mu(\cdot, 0 )-\overline m_\mu(\cdot)$ and then applying Theorem~\ref{mpEE}. The idea of proving the convergence of $\delta v^\delta$ to $-\overline H(p)$ by comparing to the maximal subsolutions was introduced in~\cite{ASo3} and extended to the second-order case in~\cite{AT2}, although the importance of the metric problem goes back to Lions~\cite{Li}.

\smallskip

This comparison is somewhat simpler in the first-order case, because the function $-m_\mu(0, \cdot )$ can be used as a global subsolution. This is not true if $A\not\equiv 0$ and, as explained in~\cite{AT2}, one overcomes this difficulty by looking at directions $p$  which are \emph{exposed points} of sublevel-sets of the form $\{\overline H=\mu\}$ where $\mu>0$: indeed, for these directions, the convergence of $\delta v^\delta(\cdot ;p)$to $-\overline H(p)$ to $0$ can be proved by using the convergence of $m_\mu(\cdot, 0 )$ to $\overline m_\mu(\cdot)$, because the function $\overline m_\mu$ is differentiable (and hence relatively flat) along the ray that touches the plane $y\mapsto p\cdot y$. For the other directions $p$, there are two cases: either $\overline H(p)>0$, and $p$ can be written as a convex combination of exposed points of some level-set $\{\overline H=\mu\}$ where $\mu>0$ (this is  Straszewiscz's theorem): then one can translate the convergence for the exposed points to the convergence for $p$. In the case $\overline H(p)=0$, that is, $p$ belongs to the ``flat spot," there is no relation between the approximate corrector and the metric problem, and different technique must be applied. This is not a limitation of the method, but intrinsic to the problem: as shown in~\cite{ACS}, the limit~\eqref{e.cellhomog} cannot be quantified without further information about the law of $H$.

\smallskip

Our task in this section is to quantify the argument of~\cite{AT2}. For this purpose we must consider a stronger version of the notion of exposed points, the so-called points of \emph{$r$-strict convexity}. For these points there is a sharp control of the difference $\delta v^\delta(\cdot ;p) - \left(-\overline H(p)\right)$ by the difference $m_\mu(\cdot, 0 )-\overline m_\mu(\cdot)$: this yields a quantitative result, which can be translated into a rate for the directions $p$ such that $\overline H(p)>0$ by a {\it quantitative version} of  Straszewiscz's theorem.

To handle the flat spot $\{\overline H=0\}$, we use the additional assumption~\eqref{e.ConstH} which essentially rules out its existence. The condition implies in particular that the flat spot is~$\{0\}$, and we obtain a convergence rate for points near zero by a simple interpolation. There are other, weaker hypotheses one could use to obtain quantitative results on the flat spot. As this is not our primary focus, we leave this issue to the reader. 

\smallskip

Throughout this section, we fix $\xi\geq 1$, which will serve as an upper bound for $|p|$. The symbols $C$ and $c$ denote positive constants which may change in each occurrence and depend only on $(d,q,\Lambda,\xi)$.

We first recall the following standard estimates on the approximate corrector, which will be needed throughout the section (c.f.~\cite{AT2,AT} for the proofs): for every $p,\tilde p\in B_\xi$ and $y,z\in \Rd$,
\begin{equation}\label{dvdsup}
-C \leq -\Lambda |p|^q -\esssup_{ \Omega} H(p,0 ) \leq \delta v^\delta(y ,p) \leq -\essinf_{\Omega} H(p,0 ) \leq -\frac{1}{\Lambda}|p|^q + \Lambda \leq C,
\end{equation}
\begin{equation}\label{vdlip}
\left| v^\delta(y,p)  - v^\delta(z,p) \right| \leq C|y-z|,
\end{equation}
and
\begin{equation}\label{vdpdepp}
\big| \delta v^\delta(y ,p) - \delta v^\delta(y ,\tilde p)\big| \leq C\left|p-\tilde p \right|.
\end{equation}

We break the proof of Theorem~\ref{p.deltaGlobal} into several steps, beginning with a bound on $-\delta v^\delta(\cdot,p) - \barH(p)$ from above, for $p\in \Rd$ satisfying $\overline H(p)>0$. This is the main focus of this section and where our arguments are different from those of~\cite{ACS}.

\subsection{Estimate of $-\delta v^\delta - \barH$ from above, off the flat spot}

The first step in the proof of Theorem \ref{p.deltaGlobal}  is an  estimate from below of the quantity $\delta v^\delta(\cdot ;p)+\overline H(p)$.

\begin{prop}
\label{p.met-to-wep1} 
There exists $C>0$ such that, for every $|p|\leq \xi$ with $\overline H(p)>0$ and $\delta,\lambda>0$ which satisfy
\begin{equation}\label{condcondA}
\lambda \geq   C \left( \frac{\delta^{\frac17}}{\overline H(p)^{\frac{6}{7}}} + \frac{\delta^{\frac25}}{\overline H(p)^{\frac{12}{5}}} \right) \log\left(2 + \frac{1}{\delta \overline H(p)} \right),
\end{equation}
we have
\begin{equation} \label{e.metwepconc}
\Prob\Big[\, -\delta v^\delta(0 ,p) \geq \overline H(p) + \lambda \Big] 
\leq  C\lambda^{-3\d}\exp\left(-\frac{(\overline H(p))^4 \lambda^4}{ C  \delta} \right)
\end{equation}
\end{prop}

We begin by recalling some relevant definitions. For a compact subset  $K$ of $\Rd$, we say that $x\in K$ is an \emph{exposed point} of $K$ if there is a plane $P$ such that $P\cap K = \{ x \}$. If $r>0$, we say that $x\in K$ is a \emph{point of $r$-strict convexity of $K$} if there is a closed ball $B$ of radius $r$ such that $K\subseteq B$ and $x\in K \cap \partial B$. Note that
\begin{equation*}\label{}
\mbox{points of $r$-strict convexity} \ \subseteq \ \mbox{exposed points}\  \subseteq \ \mbox{extreme points.}
\end{equation*}

\begin{lem}[Quantitative version of Straszewicz's theorem]
\label{l.quanstrasz}
Suppose that $K$ is compact, convex and let $r> \diam(K)$. Denote by $K^r$ the intersection of all closed balls of radius $r$ which contain $K$, and let $K_r$ be the closed convex hull of the points of $r$-strict convexity of $K$. Then 
\begin{equation*}\label{}
\dist_H(K_r,K^r) \leq \frac{(\diam(K))^2}{r}. 
\end{equation*}
\end{lem}
Lemma~\ref{l.quanstrasz} implies Straszewiscz's theorem since, for every $\diam(K) < r < \infty$,
\begin{equation*}\label{}
K_r\subseteq \conv \{ \mbox{exposed points of $K$} \} \subseteq K \subseteq K^r.
\end{equation*}
In order to put Lemma~\ref{l.quanstrasz} to good use, we need the following convex analytic lemma linking the points of $r$-strict convexity of the sublevel sets of $\barH$ to a quantitative estimate of linear approximation for the $\overline m_\mu$'s.

\begin{lem}
\label{l.linkrstrict}
Let $\mu, r > 0$ and $p\in \Rd$ be a point of $r$-strict convexity of $\{ q \in \Rd\,:\, \barH(q) \leq \mu \}$. Then there exists $e\in \Rd$ with $|e|=1$ such that $p=D\overline m_\mu(e)$ and, for every $x\in B_{1/2}$,
\begin{equation}\label{e.linkrstrict}
0\leq \overline m_\mu(e+x) - \overline m_\mu(e) - p\cdot x \leq r |x|^2 .
\end{equation}
\end{lem}

We first present the proof of Proposition~\ref{p.met-to-wep1}. The proofs of Lemmas~\ref{l.quanstrasz} and~\ref{l.linkrstrict} are postponed until the end of the subsection.

\begin{proof}[{Proof of Proposition~\ref{p.met-to-wep1}}]
We break the proof into five steps. In the first four steps, we prove the proposition under the assumption that $p$ is a point of $r$-strict convexity of a sublevel set of $\overline H$. In the final step, we remove the condition on $p$ by using the convexity of $\overline H$ and  simple convex geometry facts (Carath\'eodory's Theorem and Lemma~\ref{l.quanstrasz}).

\smallskip

In each of steps~1, 2, and 3, we fix $\lambda,\delta, \mu > 0$, a point $p\in\Rd$ of $r$-strict convexity of the convex set $\{ q\in \Rd\,:\, \overline H(q) \leq \mu\}$ with $|p| \leq \xi$ and coefficients~$\omega=(\Sigma,H)\in \Omega$ belonging to the event that~$-\delta v^\delta (0 ,p) \geq \overline H(p) + \lambda$. We note that $\mu = \overline H(p)$ and, in particular, $0< \mu \leq \Lambda |p|^q \leq \Lambda \xi^q \leq C$.

\smallskip

{\it Step 1.} We modify $v^\delta(\cdot,p)$ to prepare it for comparison. With $0<c<1$ selected below, define
\begin{equation*}\label{}
w(y):= v^\delta(y ,p) - v^\delta(0 ,p) + c\lambda \left( \left( 1+ |y|^2 \right)^{1/2} - 1\right).
\end{equation*}
Observe that $w$ satisfies
\begin{equation*}\label{}
-\tr\left(A(y ) D^2w \right) + H(p+Dw,y ) \geq - \delta v^\delta(y ,p) - \frac14\lambda \quad \mbox{in} \ \Rd,
\end{equation*}
provided that $c>0$ is chosen sufficiently small, depending on $(q,\Lambda,\xi)$. Select $y_0 \in \Rd$ such that
\begin{equation*}\label{}
w(y_0) = \inf_{\Rd} w \leq w(0) = 0
\end{equation*}
and observe that $|y_0| \leq C / \lambda \delta$. Next we define
\begin{equation*}\label{}
\hat w(y) := w(y) - w(y_0) + \hat c\lambda \left( \left( 1+ |y-y_0|^2 \right)^{1/2} - 1\right)
\end{equation*}
and observe, for $\hat c>0$ chosen sufficiently small, depending on $(q,\Lambda,\xi)$,~$\hat w$ satisfies
\begin{equation*}\label{}
-\tr\left(A(y ) D^2\hat w \right) + H(p+D\hat w,y ) \geq - \delta v^\delta(y ,p) - \frac12\lambda \quad \mbox{in} \ \Rd.
\end{equation*}
Consider the domain
\begin{equation*}\label{}
V:= \left\{ y\in \Rd \,:\, \hat w(y) < \lambda / 4\delta \right\}.
\end{equation*}
It is clear from the definition of $y_0$ and $\hat w$ that 
\begin{equation*}\label{}
V \subseteq B(y_0,R/\delta),
\end{equation*}
for some $R> 0$, depending only on $(q,\Lambda,\xi)$. Moreover, it is simple to check that
\begin{equation*}\label{}
\hat w(y) \leq \lambda/4\delta \qquad \mbox{implies that} \qquad -\delta v^\delta (y ,p) \geq -\delta v^\delta (0 ,p) - \frac\lambda4 \geq \mu + \frac{\lambda}{4}. 
\end{equation*}
Hence we obtain
\begin{equation}\label{e.hatwforcomp1}
\left\{ \begin{aligned}
& -\tr\left(A(y ) D^2\hat w \right) + H(p+D\hat w,y ) \geq \mu +  \frac14\lambda &  \mbox{in} & \ V, \\
& \hat w(y_0) = 0\quad \mbox{and}  \quad \hat w \equiv \lambda/4\delta \quad \mbox{on} \ \partial V.
\end{aligned} \right.
\end{equation}

\smallskip

{\it Step 2.} We modify $m_\mu$ to prepare it for comparison. According to Lemma~\ref{l.linkrstrict}, there exists a unit vector~$e$ such that $p = D\overline m_\mu(e)$ and, for every $x\in \Rd$ with $|x|\leq 1/2$,
\begin{equation}\label{e.donotbend}
\left| \overline m_\mu(e+x) - p\cdot (e+x) \right| \leq r |x|^2.
\end{equation}
With $s>R\delta^{-1}+1$ to be selected below, define
\begin{equation*}\label{}
\hat m(y):= m_\mu(y,y_0-se ) - m_\mu(y_0,y_0-se ) - p\cdot (y-y_0)
\end{equation*}
and observe that $\hat m$ satisfies $\hat m(y_0) = 0$ and
\begin{equation}\label{e.hatmforcomp1}
-\tr\left(A(y ) D^2\hat m \right) + H(p+D\hat m,y ) = \mu \quad \mbox{in} \ \Rd \setminus B(y_0-se,1) \supseteq V.
\end{equation}

{\it Step 3.} We apply the comparison principle (see~\cite[Theorem 2.1]{AT}). Using~\eqref{e.hatwforcomp1} and~\eqref{e.hatmforcomp1}, we obtain
\begin{equation*}\label{}
0 = \hat m(y_0) - \hat w(y_0) \leq  \sup_{V} \left( \hat m - \hat w\right) = \max_{y\in \partial V} \left(  \hat m - \hat w\right) = \max_{y\in \partial V} \hat m - \lambda/4\delta. 
\end{equation*}
A rearrangement yields 
\begin{equation*}\label{}
 \frac\lambda{4\delta} \leq \max_{y\in \partial V} \left(m_\mu(y,y_0-se ) - m_\mu(y_0,y_0-se ) - p\cdot (y-y_0) \right).
\end{equation*}
Using that $\partial V \subseteq B(y_0,R/\delta)$, we deduce that, for some $|y_0| \leq C/\lambda \delta$, either 
\begin{equation}\label{e.firstalt1}
\max_{y\in B(y_0,R/\delta)} \left(  m_\mu(y,y_0-se ) - p\cdot(y-(y_0-se)) \right) \geq \frac\lambda{8\delta}
\end{equation}
or else
\begin{equation}\label{e.secondalt1}
m_\mu(y_0,y_0-se ) - \overline m_\mu(se) = m_\mu(y_0,y_0-se ) - p\cdot (se) \leq - \frac\lambda{8\delta}.
\end{equation}
From~\eqref{e.donotbend} and positive homogeneity of $\bar m_\mu$, we have, for every $y\in  B(y_0,R/\delta)$ and assuming $R/(s\delta)\leq 1/2$, that
\begin{equation*}\label{}
\left| \overline m_\mu(se - y_0+y) - p\cdot (se - y_0+y) \right| \leq r s^{-1} |y_0-y|^2 \leq r s^{-1} R^2\delta^{-2}.
\end{equation*}
Hence the first alternative~\eqref{e.firstalt1} implies that, for $s\geq  16rR^2/(\delta\lambda)$,
\begin{equation}\label{e.firstalt1.b}
\max_{y\in B(y_0,R/\delta)} \left(  m_\mu(y,y_0-se ) - \overline m_\mu(se - y_0+ y)) \right) \geq \frac\lambda{8\delta} -  r s^{-1} R^2\delta^{-2} \geq \frac{\lambda}{16\delta}.
\end{equation}
Let us choose  $s:= 16rR^2/(\delta\lambda)$ (so that $R/(s\delta)= \lambda/ (16rR)\leq 1/2$ for $r$ large enough).

\smallskip

{\it Step 4.} We apply the results of~Theorem~\ref{mpEE}. To summarize the previous steps, we have shown that, if $-\delta v^\delta (0 ,p) \geq \overline H(p) + \lambda$ holds, then there exists $|y_0| \leq C/\lambda \delta$ such that at least one of~\eqref{e.secondalt1} and~\eqref{e.firstalt1.b} must hold. This is a purely deterministic statement relating the convergence of~\eqref{e.methomog} and~\eqref{e.cellhomog}. Using this and the Lipschitz continuity of $m_\mu$ and $\bar m_\mu$, we deduce that
\begin{multline*} 
\Prob\Big[\, -\delta v^\delta(0 ,p) \geq \overline H(p) + \lambda\Big]  \\
\leq \sum_{y_0\in B(C/\lambda \delta)\cap (c\lambda/\delta) \Z^{\d}} \Bigg( \Prob\bigg[\, m_\mu(y_0,y_0-se ) - \overline m_\mu(se)  \leq - \frac\lambda{16\delta}\bigg]  \\
+ \Prob\bigg[\, \max_{y\in B(y_0,R/\delta)} \left(  m_\mu(y,y_0-se ) - \overline m_\mu(se  - y_0+ y) \right) \geq \frac\lambda{32\delta}\bigg]\Bigg)
\end{multline*}
To estimate the first term in the sum, we apply~\eqref{easy} to find that, if $\lambda\geq C\delta\mu^{-2}$, then
\begin{multline*} 
\Prob\bigg[\, m_\mu(y_0,y_0-se ) - \overline m_\mu(se)  \leq - \frac\lambda{16\delta}\bigg] = 
\Prob\bigg[\, m_\mu(se,0 ) - \overline m_\mu(se) \leq - \frac\lambda{16\delta} \bigg]\\
 \leq \exp\left(-\frac{\mu^4 \lambda^2}{C s\delta^2} \right)= \exp\left(-\frac{\mu^4 \lambda^3}{C r\delta} \right). 
\end{multline*}
To bound the second term in the sum, we first use stationarity, the Lipschitz estimate for $m_\mu$ and $\overline m_\mu$ and the fact that the set $B(0,R/\delta)\cap(c\lambda/\delta)\Z^{\d}$ has at most $C\lambda^{-d}$ elements to obtain
\begin{align*}
\lefteqn{ \Prob\bigg[\, \max_{y\in B(y_0,R/\delta)} \left(  m_\mu(y,y_0-se ) - \overline m_\mu(se  - y_0+ y) \right) \geq \frac\lambda{32\delta}\bigg] } \qquad \qquad & \\
& = \Prob\bigg[\, \max_{z\in B(0,R/\delta)} \left(  m_\mu(z+se,0 ) - \overline m_\mu(se + z) \right) \geq \frac\lambda{32\delta}\bigg] \\
& \leq \sum_{z\in B(0,R/\delta)\cap(c\lambda/\delta)\Z^{\d}} \Prob\bigg[  m_\mu(se+z,0 ) - \overline m_\mu(se + z)  \geq \frac\lambda{64\delta}\bigg] \\ 
& \leq C\lambda^{-\d}\max_{z\in B(0,R/\delta)\cap(c\lambda/\delta)\Z^{\d}}  \Prob\bigg[  m_\mu(se+z,0 ) - \overline m_\mu(se+z )  \geq \frac\lambda{64\delta}\bigg].
\end{align*}
As $s= 16rR^2/(\delta\lambda)$, we apply the second statement of~Theorem~\ref{mpEE} to deduce that, if~$\lambda$ satisfies
$$
\frac{\lambda}{\delta} \geq  C\left( \frac{s^{\frac23}}{\mu^{2}} + \frac{s^\frac13}{\mu^4} \right) \log\left(2+\frac{s}{\mu}\right)
= C \left(\frac{r^{\frac23}}{(\delta\lambda)^{\frac23}\mu^{2}} + \frac{r^\frac13}{(\delta\lambda)^{\frac13}\mu^4} \right)  \log\left(2+\frac{r}{\delta\lambda\mu}\right),
$$
then we have, by \eqref{hard},
\begin{align*}
\lefteqn{\Prob\bigg[\, \max_{y\in B(y_0,R/\delta)} \left(  m_\mu(y,y_0-se ) - \overline m_\mu(se  - y_0+ y) \right) \geq \frac\lambda{32\delta}\bigg] } \qquad \qquad & \\
& \leq C\lambda^{-\d}\max_{z\in B(R/\delta)\cap(c\lambda/\delta)\Z^{\d}}\Prob\bigg[  m_\mu(se+z,0 ) - \overline m_\mu(se +z)  \geq \frac\lambda{64\delta}\bigg] \\
& \leq C\lambda^{-\d}\exp\left(-\frac{\mu^4 \lambda^3}{C r\delta} \right).
\end{align*}
We have proved that, if $p$ is a point of $r$-strict convexity of~$\{ q\in \Rd\,:\, \overline H(q) \leq \mu\}$ with $r\ge C$ and $\lambda$ satisfies
\begin{equation}\label{condlambda}
\frac{\lambda}{\delta} \geq   C \left(\frac{r^{\frac23}}{(\delta\lambda)^{\frac23}\mu^{2}} + \frac{r^\frac13}{(\delta\lambda)^{\frac13}\mu^4} \right) \log\left(2+\frac{r}{\delta\lambda\mu}\right), 
\end{equation}
 then 
\begin{equation*}
\Prob\left[ -\delta v^\delta(0 ,p) \geq \overline H(p) + \lambda\right] \leq C\lambda^{-3\d}\exp\left(-\frac{\mu^4 \lambda^3}{C r \delta} \right).
\end{equation*}
Note that $\lambda \geq C\delta \mu^{-2}$ is redundant in view of~\eqref{condlambda}, and the latter condition holds provided that 
$$
\lambda \geq   C \left(\frac{r^{\frac25}\delta^{\frac15}}{\mu^{\frac65}} + \frac{r^{\frac14}\delta^{\frac12}}{\mu^{3}} \right) \left( 1+ |\log r|+ |\log \delta |+|\log \mu|\right). 
$$

\smallskip

{\it Step 5.} We use the convexity of $H$ to obtain an estimate for general $|p|\leq K$. Fix $p\in \Rd$ with $|p| \leq K$ and set $\mu:= \overline H(p)$. According to Carath\'eodory's Theorem (c.f.~\cite{Gb}) and Lemma~\ref{l.quanstrasz}, for each $r>0$, there exist points $p_1, \ldots,p_{\d+1}$ of $r$-strict convexity of $\{ q \in \Rd \,:\, \overline H(q) \leq \mu\}$ and $\theta_1,\ldots,\theta_{\d+1} \in [0,1]$ such that $\sum_{i=1}^{\d+1} \theta_i=1$ and
\begin{equation*}\label{}
\left| p - \sum_{i=1}^{\d+1} \theta_i p_i \right| \leq \frac{C}{r}.
\end{equation*}
We consider coefficients $\omega=(\Sigma,H)\in\Omega$ belonging to the event that, for every $i=1, \dots, \d+1$, 
$$
-\delta v^\delta(0 ,p_i) \leq \overline H(p_i) +\frac\lambda2,
$$
where $\overline H(p_i)\leq \mu=\overline H(p)$. 
Then, by the convexity of $\overline H$ and the Lipschitz continuity of the map $p\mapsto -\delta v^\delta(y ,p)$, we have
\begin{equation*}
-\delta v^\delta(0 ,p) \leq \sum_{i=1}^{\d+1} \theta_i \left(-\delta v^\delta(0 ,p_i)\right)+\frac{C}{r} 
\leq  \sum_{i=1}^{\d+1} \theta_i \overline H(p_i) +\frac\lambda2 +\frac{C}{r} \leq \overline H(p)+\frac\lambda2 +\frac{C}{r}.
\end{equation*}
Choosing $r$ such that $\displaystyle\frac{C}{r}=\frac{\lambda}{4}< \frac{\lambda}{2}$, we obtain that 
\begin{align*}\label{}
\left\{ \omega\in \Omega\,:\, -\delta v^\delta(0,p,\omega) \geq \overline H(p) + \lambda \right\} & \subseteq \bigcup_{i=1}^{\d+1} \left\{ \omega\in \Omega\,:\, -\delta v^\delta(0 ,p_i,\omega) \geq \overline H(p_i) + \frac\lambda2 \right\} .
\end{align*}
Then from step 4 we obtain 
\begin{align*}\label{}
\Prob\Big[ -\delta v^\delta(0 ,p) \geq \overline H(p) + \lambda \Big] & \leq \sum_{i=1}^{\d+1}\Prob\bigg[ -\delta v^\delta(0 ,p_i) \geq \overline H(p_i) + \frac\lambda2 \bigg] \\
& \leq C\lambda^{-3\d}\exp\left(-\frac{\mu^4 \lambda^3}{C r \delta} \right)= C\lambda^{-3\d}\exp\left(-\frac{\mu^4 \lambda^4}{C  \delta} \right),
\end{align*}
(which is~\eqref{e.metwepconc}) provided that
\begin{equation*}
\lambda \geq   C \left( \frac{\delta^{\frac15}}{\mu^{\frac65}\lambda^{\frac25}} + \frac{\delta^{\frac12}}{\mu^{3}\lambda^{\frac14}} \right) \left( 1+ |\log(\lambda)|+ |\log(\delta)|+|\log(\mu)|\right).
\end{equation*}
This condition holds if $\lambda$ and $\delta$ satisfy
\begin{equation*}
\lambda \geq   C \left( \frac{\delta^{\frac17}}{\mu^{\frac{6}{7}}} + \frac{\delta^{\frac25}}{\mu^{\frac{12}{5}}} \right) \log\left(2 + \frac{1}{\delta\mu} \right),
\end{equation*}
which is~\eqref{condcondA}.
\end{proof}

\begin{proof}[{Proof of Lemma~\ref{l.quanstrasz}}]
It is easy to see that $(K_r)^r = K^r$, and therefore we may assume that $K = K_r$. Let $x\in K_r$ and $y\in K$ such that $|x-y| = \dist_H(K,K^r)$. 

The hyperplane $P$ passing through $y$ which is perpendicular to $x-y$ separates $K$ from the point $x$. That is, $K$ is contained in the half-space $S$ whose boundary is $P$ and which does not contain $x$. In particular, if $s=\diam K$, then 
\begin{equation*}\label{}
K \subseteq S \cap\overline B(y,s).
\end{equation*}
Set
\begin{equation*}\label{}
z := y + (r^2-s^2)^{1/2} \frac{y-x}{|y-x|}
\end{equation*}
and observe that $K \subseteq S \cap\overline (y,s) \subseteq \overline B(z,r)$. It follows from the definition of $K^r$ that $K^r\subseteq \overline B(z,r)$. Hence $x\in B(z,r)$ and
\begin{equation*}\label{}
\dist_H(K,K^r) = |x-y| = |x-z| - |y-z| \leq r - \left( r^2-s^2 \right)^{1/2}.
\end{equation*}
The proof is complete, as $r - \left( r^2-s^2 \right)^{1/2} \leq \frac12s^2 r^{-1} + o(r^{-1})$ as $r\to \infty$.
\end{proof}

\begin{proof}[{Proof of Lemma~\ref{l.linkrstrict}}]
Write $K:=\{ q \in \Rd\,:\, \barH(q) \leq \mu \}$ and select a closed ball $B$ of radius $r$ such that $K \subseteq B$ and $p\in \partial B$. Let $e$ be the outward-pointing unit normal vector to $\partial B$ at $p$. Let $\phi(x):= \max_{y\in B} (x\cdot y)$ denote the support function of $B$, and observe that, for every $x\in \Rd$,
\begin{equation*}\label{}
p\cdot x \leq \overline m_\mu(x)  \leq \phi(x)
\end{equation*}
with equality holding at $x=e$, that is, $\phi(e) = p\cdot e$. Hence, for every $x\in \Rd$,
\begin{equation*}\label{}
p\cdot x = p\cdot(x+e) - p\cdot e \leq \overline m_\mu(x+e) - \overline m_\mu(e) \leq \phi(x+e) - \phi(e). 
\end{equation*}
Rearranging and using the fact that $|D^2\phi(e)| = r$, we obtain
\begin{equation*}\label{}
0 \leq  \overline m_\mu(x+e) - \overline m_\mu(e) - p\cdot x \leq \phi(x+e) - \phi(e) - p\cdot x \leq \frac12 r |x|^2.
\qedhere
\end{equation*}
\end{proof}

\subsection{Global estimate of $-\delta v^\delta - \barH$}

The estimate from below for $-\delta v^\delta - \barH$ away from the flat spot is nearly verbatim from the argument of~\cite[Theorem 2(ii)]{ACS}. We state the result without proof, refereeing to \cite{ACS} for details.

\begin{prop}
\label{p.deltabelowOFS} 
There is a constant $C>0$, depending only on $(d,q,\Lambda,\xi)$, such that, for every $p\in \Rd$ with $|p|\leq \xi$ and $\overline H(p)>0$, $\delta,\lambda>0$ with
\begin{equation}\label{condcondB}
\lambda \geq  C\left( \frac{\delta^{\frac13}}{(\overline H(p))^{2}} + \frac{\delta^{\frac23}}{(\overline H(p))^{4}} \right) \log \left(2 + \frac{1}{\delta \overline H(p)} \right),
\end{equation}
we have
\begin{equation*}
\Prob\Big[\, -\delta v^\delta(0 ,p) \leq \overline H(p) + \lambda \Big] 
\leq  C \delta^{-2\d}\lambda^{-\d}\exp\left(-\frac{(\overline H(p))^4 \lambda^2}{C\delta} \right)
\end{equation*}
\end{prop}

We next deduce Theorem~\ref{p.deltaGlobal} from Propositions~\ref{p.met-to-wep1} and~\ref{p.deltabelowOFS}.

\begin{proof}[Proof of Theorem \ref{p.deltaGlobal}.]
The condition \eqref{e.ConstH} implies (c.f.~\cite{AT2,AT}) that, for every $p\in\Rd$,
\begin{equation}\label{ClPourBarH}
-\delta v^\delta(x ,p)\geq \frac{1}{\Lambda} |p|^q\quad \mbox{and} \quad \overline H(p) \geq \frac{1}{\Lambda}|p|^q.
\end{equation}
Since $\delta v^\delta(y,0)=\overline H(0)=0$,~\eqref{vdpdepp} yields that, for every $|p| \leq \xi$,
$$
\left|\delta v^\delta(0 ,p) + \overline H(p) \right|\leq C|p| \qquad \forall |p|\leq K.
$$
Hence to obtain~\eqref{e.thmdeltaGlobal}, we may suppose without loss of generality that $\lambda\leq C|p|\leq C$. If conditions~\eqref{condcondA} and~\eqref{condcondB} hold, then Propositions~\ref{p.met-to-wep1} and~\ref{p.deltabelowOFS}, in view of $\lambda\leq C|p| \leq C$ and~ \eqref{ClPourBarH}, yield
\begin{align*}
\Prob\Big[\, \left|\delta v^\delta(0 ,p) + \overline H(p) \right|\geq  \lambda \Big] 
& \leq  C\lambda^{-\d} \left(\lambda^{-2\d}+\delta^{-2\d}\right)
\exp\left(- \frac{(\overline H(p))^4 \lambda^4}{ C\delta} \right)\\
& \leq C\lambda^{-\d} \left(\lambda^{-2\d}+\delta^{-2\d}\right) 
\exp\left(- \frac{ \lambda^{4(1+q)}}{ C\delta}\right).
\end{align*}
This is~\eqref{e.thmdeltaGlobal}. Using again that $\lambda \leq C|p|$ and~\eqref{ClPourBarH}, we find that~\eqref{condcondA} and~\eqref{condcondB} are satisfied provided that we have both
\begin{equation*}\label{condcondAhb}
\lambda \geq   C \left( \frac{\delta^{\frac17}}{\lambda^{\frac{6q}{7}}} + \frac{\delta^{\frac25}}{\lambda^{\frac{12q}{5}}} \right) \log \left( 2+  \frac1{\delta\lambda} \right)
\end{equation*}
and 
\begin{equation*}\label{condcondBkjqbvsd}
\lambda \geq   C  \left( \frac{\delta^{\frac13}}{\lambda^{2q}} + \frac{\delta^{\frac23}}{\lambda^{4q}} \right)  \log \left( 2+  \frac1{\delta\lambda} \right).
\end{equation*}
One easily checks that the second condition is redundant with the first, and the first term in parentheses of the first condition is the limiting one. We obtain finally that both hold if
\begin{equation*}\lambda \geq   C  \delta^{\frac{1}{7+6q}} \left( 1+  \left|\log \delta\right|\right).
\end{equation*}
This completes the proof of the theorem.
\end{proof}

\section{Error estimates for the time dependent problem}\label{s.timedep}

In this section we prove Theorem \ref{p.cvuep} on the convergence rate of the solution $u^\ep$ of \eqref{HJq} to the solution $ u$ of the homogenized problem \eqref{HJqhom}. This part is very close to the corresponding one for first order problems (see section 7 in \cite{ACS}). However the second order term induces additional technicalities, so the result is not just routine adaptation of \cite{ACS}: this is the reason why we provide the details. 

\smallskip

Before beginning the proof, we introduce some notation and recall some properties of~\eqref{HJq}. According to~\cite{AT}, there exists~$L>0$, depending on $(d,q,\Lambda,\| g \|_{C^{1,1}(\Rd)})$ such that, for all $\ep > 0$, $x,y\in \Rd$ and $s,t\geq 0$,
\begin{equation}\label{ueplip}
 |u^\ep(x,t ) - u^\ep(y,s )| \leq L\left( |x-y| + |s-t| \right)
\end{equation}
and 
\begin{equation}\label{ulip}
|u(x,t) - u(y,s)| \leq L\left( |x-y| + |s-t| \right).
\end{equation}
It also follows easily from this that for each $\ep > 0$, $x\in \Rd$ and $0 \leq t\leq T$,
\begin{equation}\label{uesup}
|u(x,t)| + |u^\ep(x,t )| \leq K + LT \leq C(1+T).
\end{equation}

The key point in the proof in the proof of Theorem \ref{p.cvuep} amounts to estimate the difference $u^\ep- u$ by the difference between $\delta v^\delta$ and $\overline H$. This is a purely deterministic PDE fact, which we summarize in the following lemma. Throughout, $C$ and $c$ denote positive constants which may vary in each occurrence and depend only on~$(d,q,\Lambda,\| g \|_{C^{1,1}(\Rd)})$.

\begin{lem}\label{l.jzfcojjh} There exists a constant $C>0$ depending only on $(d,q,\Lambda,\| g \|_{C^{1,1}(\Rd)})$ such that, with $L> 0$ as in~\eqref{ueplip} and~\eqref{ulip}, we have,  for any $\lambda, \delta, \ep>0$, with $0<\ep<\delta<1$ and $R := C\left(\frac{1}{\lambda\ep}+\frac{1}{\delta}\right)$,
\begin{multline}\label{jzfcojjh1}
\Bigg\{ \omega\in \Omega\ : \ \sup_{x\in B_T, \ t\in [0,T]} u^\ep(x,t,\omega)-u(x,t) >  C\left(\lambda+ \frac{\ep^{\frac23}}{\delta^{\frac13}}+ \frac{\ep^{\frac13}}{\delta^{\frac23}}\right) \Bigg \}  \\  
\subseteq \Big\{ \omega\in \Omega\ : \ \inf_{(y,p) \in B_{R} \times B_{L}} \left( -\delta v^\delta (y ,p,\omega) - \overline H(p)\right) <  - \lambda \Big\}
\end{multline}
and
\begin{multline}\label{jzfcojjh2}
\Bigg\{ \omega\in \Omega\ : \ \inf_{x\in B_T, \ t\in [0,T]} u^\ep(x,t)-u(x,t) < -C\left(\lambda+ \frac{\ep^{\frac23}}{\delta^{\frac13}}+ \frac{\ep^{\frac13}}{\delta^{\frac23}}\right) \Bigg \}  \\  
\subseteq \Big\{ \omega\in \Omega\ : \ \sup_{(y,p) \in B_{R} \times B_{L}} \left( -\delta v^\delta (y ,p) - \overline H(p)\right) >   \lambda \Big\}.
\end{multline}
\end{lem}

\begin{proof} Since the arguments for \eqref{jzfcojjh1} and \eqref{jzfcojjh2} are nearly identical we only prove \eqref{jzfcojjh1}. Fix $\lambda,\delta,\ep>0$, with $\ep<\delta<1$, and a realization of the coefficients $\omega=(\Sigma,H) \in \Omega$ such that 
\be\label{condsuromega}
\inf_{(y,p) \in B_{R} \times B_{L}} \left( -\delta v^\delta (y ,p) - \overline H(p)\right)\geq - \lambda. 
\ee
We are going to show that  
\begin{equation}\label{ineqfinale}
\sup_{x\in B_T, \ t\in [0,T]} u^\ep(x,t)-u(x,t) \leq C\left(\lambda+ \frac{\ep^{\frac23}}{\delta^{\frac13}}+ \frac{\ep^{\frac13}}{\delta^{\frac23}}\right).
\end{equation}
Define $\zeta:\Rd\to\Rd$ by
\begin{equation*}\label{}
\zeta(x):=\max\left\{ \frac L{|x|}, 1 \right\} x,
\end{equation*}
and notice that 
\begin{equation} \label{ss}
|\zeta(x)| = L\wedge |x| \quad \mbox{and} \quad |\zeta(x) - \zeta(y)| \leq |x-y|.
\end{equation}
Fix parameters $\alpha,q > 0$ and $d\in (0,1)$ to be chosen below and consider the auxiliary function $\Phi:\Rd \times \Rd \times [0,T] \times [0,T] \times \Omega \to \R$ given by
\begin{multline}\label{auxfun}
\Phi(x,y,t,s ):= u^\ep ( x,t ) - u(y,s) - \ep v^\delta \left( \frac x\ep  \, ; \zeta\left( \frac{x-y}{\alpha}  \right) \right) -  \frac1{2\alpha}|x-y|^2 \\ -\frac1{2\ep} (t-s)^2 -ds-q  \left(1+|x|^2 \right)^{\frac12} + q.
\end{multline}
Using~\eqref{dvdsup},~\eqref{uesup} and~\eqref{ss}, we have, for each $(x,y,t,s ) \in \Rd \times \Rd \times [0,T] \times [0,T] \times \Omega$,
\begin{equation}\label{crudeP}
|\Phi(x,y,t,s )| \leq C(1+T) + C \ep\delta^{-1}  - \frac{1}{2\alpha}|x-y|^2 - \frac{1}{2\ep} (t-s)^2 -ds- q \left(1+|x|^2 \right)^{\frac12} + q.
\end{equation}
It follows that, for each $\omega\in \Omega$, the function $\Phi$ attains its global maximum at some point of $\Rd \times \Rd\times [0,T] \times [0,T]$. Set
\begin{equation*}\label{}
M(\omega):=\max_{\Rd \times \Rd\times [0,T] \times [0,T]} \Phi(\cdot ).
\end{equation*}
We now record two elementary estimates that necessarily hold for any $\omega\in \Omega$ and at any global maximum point $(x_0, y_0, t_0, s_0) \in \Rd \times \Rd\times [0,T] \times [0,T]$ of $\Phi(\cdot )$, i.e.,  such that
\begin{equation}\label{bgpt}
\Phi( x_0, y_0, t_0, s_0 ) = M(\omega).
\end{equation}
The inequality $\Phi( x_0, y_0, t_0, s_0 ) \geq \Phi(0,0,0,0 )$ yields, in light of~\eqref{crudeP}, that
\begin{equation}\label{easybnds}
q |x_0| + \frac{1}{2\ep} (t_0-s_0)^2 \leq C(1+T) + \frac{C\ep}{\delta}\leq CT.
\end{equation}
If $s_0\neq 0$, then by~\eqref{ulip} and since $s\mapsto u(y_0,s) + (s-t_0)^2/2\ep$ has a minimum at $s=s_0$, we deduce that
\begin{equation}\label{s0t0}
|s_0-t_0| \leq L \ep.
\end{equation}
Inequality~\eqref{s0t0} is also satisfied for a similar reason if $t_0\neq 0$, and trivially if $s_0=t_0=0$. We also claim that 
\begin{equation}\label{liplam}
|x_0-y_0| \leq L\alpha. 
\end{equation}
If not, then $y\mapsto \zeta( (x_0-y) / \alpha)$ is constant in a neighborhood of $y_0$ and we obtain from~\eqref{bgpt} that
\begin{equation*}
y \mapsto u(y,s_0) +  \frac1{2\alpha}|x_0-y|^2 \quad \mbox{has a local minimum at} \ y = y_0.
\end{equation*}
Due to~\eqref{ulip}, we conclude that $|x_0-y_0| \leq L \alpha$, which is a contradiction. In particular,
\begin{equation}\label{zetafix}
\zeta\left( \frac{x_0-y_0}{\alpha} \right) = \frac{x_0-y_0}{\alpha}.
\end{equation}

We now commence with the proof of~\eqref{ineqfinale}, following the classical argument for the proof of the comparison principle for viscosity solutions and ideas from~\cite{ICD}. Fix $\omega$ for which \eqref{condsuromega} holds and select a point $(x_0, y_0, t_0, s_0) \in \Rd \times \Rd\times [0,T] \times [0,T]$ as in~\eqref{bgpt}. We assume for the moment that $s_0>0$ and $t_0$ is sufficiently large, in a sense to be explained below.  

The first step is to fix $(x,t)=(x_0,t_0)$, allow $(y,s)$ to vary, and to use the equation for $u$. The goal is to derive~\eqref{goal1}, below. From~\eqref{bgpt}, we see that 
\begin{multline}\label{frtp}
(y,s) \mapsto u(y,s) + \ep v^\delta \left( \frac {x_0}\ep  \, ; \zeta\left( \frac{x_0-y}{\alpha}  \right)  \right) +  \frac1{2\alpha}|x_0-y|^2 +\frac1{2\ep} (t_0-s)^2+ds  \\ \mbox{has a local minimum at} \ (y,s) = (y_0,s_0).
\end{multline}
According to~\eqref{vdpdepp} and~\eqref{ss},
\begin{multline}\label{lstln}
\left|\ep v^\delta \left( \frac {x_0}\ep  \, ; \zeta\left( \frac{x_0-y}{\alpha}  \right)  \right) - \ep v^\delta \left( \frac {x_0}\ep  \, ; \zeta\left( \frac{x_0-y_0}{\alpha}  \right)  \right)\right| \\ \leq \frac{C\ep}{\delta} \left| \zeta\left( \frac{x_0-y}{\alpha}  \right) - \zeta\left( \frac{x_0-y_0}{\alpha}  \right) \right| \leq \frac{C\ep|y-y_0|}{\delta\alpha}.
\end{multline}
Using~\eqref{frtp}, \eqref{lstln}, the fact that equality holds in~\eqref{lstln} at $y=y_0$ and by enlarging $C> 0$ slightly, we obtain that, for $\theta>0$,
\begin{multline}\label{frto}
(y,s) \mapsto u(y,s) +  \frac1{2\alpha}|x_0-y|^2 +\frac1{2\ep} (t_0-s)^2+ds  + C\frac{\ep}{\delta\alpha} |y-y_0|+\frac\theta2(s-s_0)^2 \\ \mbox{has a strict local minimum at} \ (y,s) = (y_0,s_0).
\end{multline}
It follows that, for all sufficiently small $\beta > 0$, there exist  $(y_\beta,s_\beta)\in \Rd\times[0,T]$ such that $(y_\beta,s_\beta) \rightarrow (y_0,s_0)$ as $\beta \to 0$ and
\begin{multline}\label{frto2}
(y,s) \mapsto u(y,s) +  \frac1{2\alpha}|x_0-y|^2 +\frac1{2\ep} (t_0-s)^2+ds  + C\frac{\ep}{\delta\alpha}  \left(\beta+|y-y_0|^2 \right)^{\frac12}+\frac\theta2(s-s_0)^2 \\ \mbox{has a local minimum at} \ (y,s) = (y_\beta,s_\beta).
\end{multline}
Using equation \eqref{HJqhom} satisfied by $u$, we obtain
\begin{equation*}\label{}
-d+ \frac1\ep (t_0 - s_\beta)-\theta(s_\beta-s_0) + \overline H\left( \frac{x_0 - y_\beta}{\alpha} - Q_\beta \right) \geq 0,
\end{equation*}
where $Q_\beta := C\frac{\ep}{\delta\alpha}  \left( \beta+|y_\beta-y_0|^2 \right)^{-\frac12} (y_\beta-y_0)$. Since $|Q_\beta| \leq C\ep/(\delta\alpha)$, the Lipschitz continuity of $\overline H$ yields
\begin{equation*}\label{}
-d+\frac1\ep (t_0 - s_\beta) -\theta(s_\beta-s_0)+ \overline H\left( \frac{x_0 - y_\beta}{\alpha}  \right) \geq - C\frac{\ep}{\delta\alpha},
\end{equation*}
and, after letting $\beta \to 0$, we find
\begin{equation}\label{goal1}
-d+ \frac1\ep (t_0-s_0) + \overline H\left( \frac{x_0-y_0}{\alpha} \right) \geq - C\frac{\ep}{\delta\alpha}. 
\end{equation}

We next fix $(y,s) = (y_0,s_0)$ and let $(x,t)$ vary, in order to use the equations for $u^\ep$ and $v^\delta$. The intermediate goal is to prove~\eqref{goal22}, below, to complement~\eqref{goal1}. 

From~\eqref{bgpt}, we see that 
\begin{multline*}
(t,x)\to  u^\ep ( x,t )  - \ep v^\delta \left( \frac x\ep  \, ; \zeta\left( \frac{x-y_0}{\alpha}  \right) \right) -  \frac1{2\alpha}|x-y_0|^2 \\ -\frac1{2\ep} (t-s_0)^2 -q  \left(1+|x|^2 \right)^{\frac12}  \; \mbox{\rm has a maximum at $(x_0,t_0)$.}
\end{multline*}
Using \eqref{zetafix} and \eqref{lstln}, the above inequality implies, as in the previous step, that
\begin{multline}\label{x0t0optimal}
(t,x)\to  u^\ep ( x,t )  - \ep v^\delta \left( \frac x\ep  \, ;  \frac{x_0-y_0}{\alpha}  \right) -  \frac1{2\alpha}|x-y_0|^2  -\frac1{2\ep} (t-s_0)^2 \\-q  \left(1+|x|^2 \right)^{\frac12}  - C\frac{\ep}{\delta\alpha} |x-x_0|   \; \mbox{\rm has a maximum at $(x_0,t_0)$.}
\end{multline}
Let us now fix a small parameter $\theta\in (0,1)$, to be chosen later, and consider a maximum point $(x_\theta, t_\theta)$ of the perturbed problem
\begin{multline}\label{txtheta}
(t,x)\to  u^\ep ( x,t )  - \ep v^\delta \left( \frac x\ep  \, ;  \frac{x_0-y_0}{\alpha}  \right) -  \frac1{2\alpha}|x-y_0|^2  -\frac1{2\ep} (t-s_0)^2 \\-q  \left(1+|x|^2 \right)^{\frac12}  - \tilde C\frac{\ep}{\delta\alpha}\left(\left(\theta^2+ |x-x_0|^2\right)^{\frac12}  -\theta\right)-\frac{1}{2\ep} (t-t_0)^2,
\end{multline}
with $\tilde C$  to be chosen below.

For later use, we  estimate the distance from $(x_0,t_0)$ to $(x_\theta,t_\theta)$: 
by \eqref{x0t0optimal}, we have
\begin{multline*}
 u^\ep ( x_\theta,t_\theta )  - \ep v^\delta \left( \frac{x_\theta}{\ep}  \, ;  \frac{x_0-y_0}{\alpha}  \right) -  \frac1{2\alpha}|x_\theta-y_0|^2  -\frac1{2\ep} (t_\theta-s_0)^2 -q  \left(1+|x_\theta|^2 \right)^{\frac12}  - \frac{C\ep}{\delta\alpha} |x_\theta-x_0|   \\
\leq
 u^\ep ( x_0,t_0 )  - \ep v^\delta \left( \frac{x_0}{\ep } \, ;  \frac{x_0-y_0}{\alpha}  \right) -  \frac1{2\alpha}|x_0-y_0|^2  -\frac1{2\ep} (t_0-s_0)^2 -q  \left(1+|x_0|^2 \right)^{\frac12} . 
\end{multline*}
On another hand, optimality of $(x_\theta,t_\theta)$ in \eqref{txtheta} implies
\begin{multline*}
u^\ep ( x_0,t_0 )  - \ep v^\delta \left( \frac{x_0}{\ep}  \, ;  \frac{x_0-y_0}{\alpha}  \right) -  \frac1{2\alpha}|x_0-y_0|^2  -\frac1{2\ep} (t_0-s_0)^2 -q  \left(1+|x_0|^2 \right)^{\frac12} \\
\leq  u^\ep ( x_\theta,t_\theta )  - \ep v^\delta \left( \frac{x_\theta}{\ep}  \, ;  \frac{x_0-y_0}{\alpha}  \right) -  \frac1{2\alpha}|x_\theta-y_0|^2  -\frac1{2\ep} (t_\theta-s_0)^2 \\-q  \left(1+|x_\theta|^2 \right)^{\frac12}  - \tilde C\frac{\ep}{\delta\alpha}\left( \left(\theta^2+ |x_\theta-x_0|^2\right)^{\frac12}  -\theta\right)-\frac{1}{2\ep}(t_\theta-t_0)^2.
\end{multline*}
Putting together the above inequalities yields 
$$
\tilde C\frac{\ep}{\delta\alpha}\left(\left(\theta^2+ |x_\theta-x_0|^2\right)^{\frac12}  -\theta\right)+\frac{1}{\ep}(t_\theta-t_0)^2  \leq C\frac{\ep}{\delta\alpha} |x_\theta-x_0| ,
$$
so that, if $\tilde C$ is sufficiently large, 
\begin{equation}\label{estixt0-xttheta}
\left| x_0-x_\theta\right|\leq C\theta\qquad {\rm and }\qquad |t_\theta-t_0|\leq C \left(\frac{\ep^2\theta}{\delta\alpha}\right)^{\frac12}. 
\end{equation}
In particular, if $t_0> C \left(\frac{\ep^2\theta}{\delta\alpha}\right)^{\frac12}$, then $t_\theta>0$. 
We finally fix two other parameters $\sigma, \rho > 0$ (which will be sent to zero shortly) and introduce a last auxiliary function $\Psi:\Rd\times \Rd \times [0,T] \to \R$ defined by
\begin{align}\label{auxfun2}
\Psi(x,z,t) & := u^\ep ( x,t ) - \ep v^\delta \left( \frac z\ep  \, ;  \frac{x_0-y_0}{\alpha} \right) -  \frac1{2\alpha}|x-y_0|^2  -\frac1{2\ep} (t-s_0)^2 \\ 
& \qquad  - q\left(1+|x|^2 \right)^{\frac12}  - \frac{1}{2\sigma} |z-x|^2 - \tilde C\frac{\ep}{\delta\alpha}\left( \left(\theta^2+ |x-x_0|^2\right)^{\frac12}  -\theta\right) \nonumber
\\ 
& \qquad-\frac{1}{2\ep}(t-t_0)^2 -\frac\rho2|(x,t)-(x_\theta,t_\theta)|^2. \nonumber
\end{align}
The last term in~\eqref{auxfun2} provides some strictness and therefore, by~\eqref{bgpt}, there exist points $(x_\sigma,z_\sigma,t_\sigma) \in \Rd \times\Rd \times [0,T]$ such that $(x_\sigma,z_\sigma,t_\sigma)\to (x_\theta,x_\theta,t_\theta)$ as $\sigma \to 0$ and
\begin{equation*}\label{}
\Psi(x_\sigma,z_\sigma,t_\sigma) = \sup_{\Rd\times\Rd\times[0,T]} \Psi. 
\end{equation*}
From the Lipschitz regularity of $v^\delta \left(\cdot  \, ;  \frac{x_0-y_0}{\alpha} \right)$ recalled in \eqref{vdlip}, we also have 
\be\label{xthetax0}
|x_\sigma-z_\sigma|\leq K_p\sigma. 
\ee
According to the maximum principle for semicontinuous functions \cite{CIL}, it follows that, for any $\eta>0$ there exist $X_{\sigma,\eta}, Y_{\sigma,\eta}\in\SL^d$ such that 
$$
\left(\frac{1}{\ep}(t_\sigma-s_0)+\frac{1}{\ep}(t_\sigma-t_0)+ \rho(t_\sigma-t_\theta), \frac{x_\sigma-y_0}{\alpha}+\frac{x_\sigma-z_\sigma}{\sigma}+p_\sigma, X_{\sigma,\eta} \right) \in \overline {\mathcal P}^{2,+}u^\ep ( x_\theta,t_\theta )
$$
$$
\left(\frac{x_\sigma-z_\sigma}{\sigma}, Y_{\sigma,\eta} \right)\in \overline {\mathcal P}^{2,-}v^\delta\left(\frac{z_\sigma}{\ep} ;  \frac{x_0-y_0}{\alpha} \right), 
$$
\be\label{matineq}
\left(\begin{array}{cc}
X_\sigma & 0\\
0 & \frac{1}{\ep} Y_{\sigma,\eta} \end{array}
\right)
\leq M_{\sigma}
+ \eta M_{\sigma}^2, 
\ee
where 
$$
p_{\sigma}= \frac{q x_\sigma}{\left(1+|x_\sigma|^2 \right)^{\frac12}}+  \frac{\tilde C\ep(x_\sigma-x_0)}{\delta\alpha\left(\theta^2+ |x_\sigma-x_0|^2\right)^{\frac12} }+\rho(x_\sigma- x_\theta), 
$$
\be\label{defZsigma}
M_{\sigma}= \left(\begin{array}{cc} \frac{1}{\sigma} I_d & -\frac{1}{\sigma} I_d \\ -\frac{1}{\sigma} I_d & \frac{1}{\sigma} I_d\end{array}\right)
+ 
\left(\begin{array}{cc} P_{\sigma} & 0 \\ 0 & 0\end{array}\right)
\ee
$\; {\rm and}$
\begin{multline}
\; P_{\sigma}= \frac{1}{\alpha} I_d+ q \left(\frac{I_d}{\left(1+|x_\sigma|^2 \right)^{\frac12}}-\frac{x_\sigma\otimes x_\sigma}{\left(1+|x_\sigma|^2 \right)^{\frac32}}\right)
 \\ + \tilde C\frac{\ep}{\delta\alpha}\left(\frac{I_d}{\left(\theta^2+ |x_\sigma-x_0|^2\right)^{\frac12} }
- \frac{(x_\sigma-x_0)\otimes (x_\sigma-x_0)}{\left(\theta^2+ |x_\sigma-x_0|^2\right)^{\frac32} }\right)+\rho I_d.
\end{multline}
Note that
\begin{equation}\label{estipPSigma}
|p_\sigma|\leq q +  \frac{\tilde C\ep}{\delta\alpha }+\rho|x_\sigma- x_\theta|\; {\rm and}\; 
\|P_\sigma\| \leq \frac{1}{\alpha}+ 2q+ C\frac{\ep}{\delta\alpha\theta}+\rho. 
\end{equation}
From equation \eqref{HJq} satisfied by $u^\ep$ and since $t_\sigma>0$ for $\sigma$ small enough, we have 
\begin{multline}\label{eqpouruep}
\frac{1}{\ep}(t_\sigma-s_0)+\frac{1}{\ep}(t_\sigma-t_0)+\rho(t_\sigma-t_\theta) -\ep{\rm tr}\left(A\left(\frac{x_\sigma}{\ep} \right)X_{\sigma,\eta}\right)\\+ H\left( \frac{x_\sigma-y_0}{\alpha}+\frac{x_\sigma-z_\sigma}{\sigma}+p_\sigma,\frac{x_\sigma}{\ep} \right) \leq 0,
\end{multline}
while, in view of equation \eqref{e.cellp} for $v^\delta$, we have
\be\label{eqpourvdelta}
\delta v^\delta\left(\frac{z_{\sigma}}{\ep} ;  \frac{x_0-y_0}{\alpha} \right)-{\rm tr}\left(A\left(\frac{z_\sigma}{\ep} \right)Y_{\sigma,\eta}\right)+ H\left( \frac{x_\sigma-z_\sigma}{\sigma}+ \frac{x_0-y_0}{\alpha},\frac{z_\sigma}{\ep} \right) \geq 0.
\ee
We multiply inequality \eqref{matineq} by the positive matrix
$$
\left(\begin{array}{c}
\Sigma\left(\frac{x_\sigma}{\ep} \right) \\
\Sigma\left(\frac{z_\sigma}{\ep} \right)\end{array}\right)
\left(\begin{array}{c}
\Sigma\left(\frac{x_\sigma}{\ep} \right) \\
\Sigma\left(\frac{z_\sigma}{\ep} \right)\end{array}\right)^T
$$
and take the trace of the resulting expression to obtain, by \eqref{estipPSigma},
\begin{align*}
\lefteqn{ \tr\left(A\left(\frac{x_\sigma}{\ep} \right)X_{\sigma,\eta}\right)
 - \frac{1}{\ep} {\rm tr}\left(A\left(\frac{z_\sigma}{\ep} \right)Y_{\sigma,\eta}\right) } & \qquad \qquad \\
& \leq \frac{1}{\sigma}\left| \Sigma\left(\frac{x_\sigma}{\ep} \right)-\Sigma\left(\frac{z_\sigma}{\ep} \right)\right|^2
 +\left( \frac{1}{\alpha}+ 2q+ C\frac{\ep}{\delta\alpha\theta}+\rho\right) \left| \Sigma\left(\frac{x_\sigma}{\ep} \right)\right|^2\\
& \qquad  +\eta {\rm tr}\left(M_{\sigma,\eta}^2\left(\begin{array}{c}
\Sigma\left(\frac{x_\sigma}{\ep} \right) \\
\Sigma\left(\frac{z_\sigma}{\ep} \right)\end{array}\right)
\left(\begin{array}{c}
\Sigma\left(\frac{x_\sigma}{\ep} \right) \\
\Sigma\left(\frac{z_\sigma}{\ep} \right)\end{array}\right)^t
\right)
 \end{align*}
Using~\eqref{e.sigbnd} and~\eqref{e.siglip}, we obtain
 \begin{multline}\label{ineqpourtraces}
 {\rm tr}\left(A\left(\frac{x_\sigma}{\ep} \right)X_{\sigma,\eta}\right)
 - \frac{1}{\ep} {\rm tr}\left(A\left(\frac{z_\sigma}{\ep} \right)Y_{\sigma,\eta}\right) \\
 \leq \frac{C|x_\sigma-z_\sigma|^2}{\sigma\ep^2}+C\left( \frac{1}{\alpha}+ q+ \frac{\ep}{\delta\alpha\theta}+\rho\right)
 + \eta C_\sigma
 \end{multline}
where the constant  $C_\sigma$ actually depends also on all the other parameters of the problem, but is independent of $\eta$. 
On another hand we have by \eqref{xthetax0} and \eqref{estipPSigma}: 
\begin{multline}\label{ineqpourH}
\left| H\left( \frac{x_\sigma-y_0}{\alpha}+\frac{x_\sigma-z_\sigma}{\sigma}+p_\sigma,\frac{x_\sigma}{\ep} \right) 
-H\left( \frac{x_\sigma-z_\sigma}{\sigma}+ \frac{x_0-y_0}{\alpha},\frac{z_\sigma}{\ep} \right) \right|  \\
\leq C\left(\frac{|x_\sigma-x_0|}{\alpha}+\frac{|x_\sigma-z_\sigma|}{\ep}+ |p_\sigma|\right) \leq C\left(\frac{|x_\sigma-x_0|}{\alpha}+\frac{\sigma}{\ep}+q +  \frac{\ep}{\delta\alpha }+\rho|x_\sigma- x_\theta|\right).
 \end{multline}
Subtracting \eqref{eqpourvdelta} to \eqref{eqpouruep} and using \eqref{ineqpourtraces} and \eqref{ineqpourH}  we get
\begin{multline*}
\frac{1}{\ep}(t_\sigma-s_0)+\frac{1}{\ep}(t_\sigma-t_0)+\rho(t_\sigma-t_\theta)-\delta v^\delta\left(\frac{z_{\sigma}}{\ep} ;  \frac{x_0-y_0}{\alpha} \right) \\ 
\leq  C\Bigg(  \frac{|x_\sigma-z_\sigma|^2}{\sigma\ep}+\ep\left( \frac{1}{\alpha}+ q+ \frac{\ep}{\delta\alpha\theta}+\rho\right)
+ \frac{|x_\sigma-x_0|}{\alpha}\\ +\frac{\sigma}{\ep}+q+  \frac{\ep}{\delta\alpha }+\rho|x_\sigma- x_\theta|\Bigg)+ \ep\eta C(\sigma).
 \end{multline*}
We let $\eta$, $\sigma$  and then  $\rho$ tend to $0$ to obtain, thanks to \eqref{xthetax0}, 
\begin{multline*}
\frac{1}{\ep}(t_\theta-s_0)+\frac{1}{\ep}(t_\theta-t_0)-\delta v^\delta\left(\frac{x_\theta}{\ep} ;  \frac{x_0-y_0}{\alpha} \right) \\ \leq
 C\bigg(\ep\left( \frac{1}{\alpha}+ q+ \frac{\ep}{\delta\alpha\theta}\right) 
+ \frac{|x_\theta-x_0|}{\alpha}+q+  \frac{\ep}{\delta\alpha }\bigg).
 \end{multline*}
 Recalling \eqref{estixt0-xttheta} and using $0<\delta, \ep\leq 1$, 
\be\label{goal22}
\frac{1}{\ep}(t_0-s_0)-\delta v^\delta\left(\frac{x_\theta}{\ep} ;  \frac{x_0-y_0}{\alpha} \right) \leq
C\left(q+ \frac{\ep^2}{\delta\alpha\theta}+ \frac{\theta}{\alpha}+  \frac{\ep}{\delta\alpha }+ \left(\frac{\theta}{\delta\alpha}\right)^{\frac12}\right).
\ee
We  now put \eqref{goal1} and \eqref{goal22} together: 
$$
-\overline H\left( \frac{x_0-y_0}{\alpha} \right)-\delta v^\delta\left(\frac{x_\theta}{\ep} ;  \frac{x_0-y_0}{\alpha} \right)
 \leq -d+  C\left(q+ \frac{\ep^2}{\delta\alpha\theta}+ \frac{\theta}{\alpha}+  \frac{\ep}{\delta\alpha }+\left(\frac{\theta}{\delta\alpha}\right)^{\frac12}
 \right). 
$$
So if we choose 
\begin{equation}\label{e.choixded}
d=   \lambda+ 2C\left(q+ \frac{\ep^2}{\delta\alpha\theta}+ \frac{\theta}{\alpha}+  \frac{\ep}{\delta\alpha }+\left(\frac{\theta}{\delta\alpha}\right)^{\frac12}\right),
\end{equation}
we have a contradiction with \eqref{condsuromega} provided $\frac{|x_\theta|}{\ep}\leq R$. Thanks to \eqref{easybnds} and \eqref{estixt0-xttheta},  this latter inequality holds provided 
\begin{equation}\label{CondSurR}
R\geq  C\left(\frac{1}{q\ep}+\frac{\theta}{\ep}\right).
\end{equation}

Therefore, if \eqref{e.choixded} and \eqref{CondSurR} hold,  we have either $s_0=0$, or $t_0\leq C\left(\frac{\ep^2\theta}{\delta\alpha}\right)^{\frac12}$. Let us first assume that $s_0=0$. In this case, by the Lipschitz bound \eqref{ueplip}, we have
\begin{align*}
M(\omega) & \leq  u^\ep ( x_0,t_0 ) - u(y_0,0) - \ep v^\delta \left( \frac {x_0}\ep  \, ; \zeta\left( \frac{x_0-y_0}{\alpha}  \right) \right) -  \frac1{2\alpha}|x_0-y_0|^2  -\frac1{2\ep} t_0^2 \\ 
& \leq u^\ep ( y_0,0 ) +L|(x_0,t_0)-(y_0,0)|- u(y_0,0) +C \frac\ep\delta -  \frac1{2\alpha}|x_0-y_0|^2  -\frac1{2\ep} t_0^2\\
& \leq C\left(\alpha+\ep+\frac\ep\delta\right),
\end{align*}
using that $u^\ep ( y_0,0 )= u_0(y_0)=u(y_0,0)$. 
On the other hand, if $t_0\leq C\left(\frac{\ep^2\theta}{\delta\alpha}\right)^{\frac12}$, then~\eqref{s0t0} yields  
\begin{align*}
M(\omega) & \leq u^\ep ( x_0,t_0 ) - u(y_0,s_0) - \ep v^\delta \left( \frac {x_0}\ep  \, ; \zeta\left( \frac{x_0-y_0}{\alpha}  \right) \right) -  \frac1{2\alpha}|x_0-y_0|^2  -\frac1{2\ep} (t_0-s_0)^2 \\ 
& \leq u^\ep ( x_0,0 )- u(x_0,0) +L|(x_0,t_0)-(y_0,s_0)| +C\left(\frac{\ep^2\theta}{\delta\alpha}\right)^{\frac12}+C \frac\ep\delta\\ & \qquad  -  \frac1{2\alpha}|x_0-y_0|^2  -\frac1{2\ep} (t_0-s_0)^2\\
& \leq C\left(\alpha+\ep+\frac\ep\delta+\left(\frac{\ep^2\theta}{\delta\alpha}\right)^{\frac12}\right) .
\end{align*}
As a consequence, we have, for any $(x,t)\in B_T\times [0,T]$,  
\begin{multline*}
\Phi(x,x,t,t ):= u^\ep ( x,t ) - u(x,t) - \ep v^\delta \left( \frac x\ep  \, ; 0  \right) -dt-q  \left(1+|x|^2 \right)^{\frac12} + q\\
\leq M(\omega) \leq C\left(\alpha+\ep+\frac\ep\delta+\left(\frac{\ep^2\theta}{\delta\alpha}\right)^{\frac12}\right), 
\end{multline*}
so that, recalling  the choice of $d$ in \eqref{e.choixded} and using $0\leq \ep\leq 1$: 
$$
u^\ep ( x,t ) - u(x,t) 
\leq C\left(\lambda+q+\frac{\ep^2}{\delta\alpha\theta}+ \frac{\theta}{\alpha}+  \frac{\ep}{\delta\alpha }+\left(\frac{\theta}{\delta\alpha}\right)^{\frac12}+ 
\alpha\right). 
$$
If we choose $q=\lambda$, $\theta= \ep\delta^{-1}$ and $\alpha= \ep^{\frac13}\delta^{-\frac23}$, we get
$$
u^\ep ( x,t ) - u(x,t) 
\leq C\left(\lambda+ \frac{\ep^{\frac23}}{\delta^{\frac13}}+ \frac{\ep^{\frac13}}{\delta^{\frac23}}\right) .
$$
Note also that, by definition of $R$, \eqref{CondSurR} holds. To summarize, we have proved that, if \eqref{ineqfinale} holds, then 
$$
\sup_{(x,t)\in B_T\times [0,T]} \left(u^\ep ( x,t ) - u(x,t) \right) \leq C\left(\lambda+ \frac{\ep^{\frac23}}{\delta^{\frac13}}+ \frac{\ep^{\frac13}}{\delta^{\frac23}}\right) .
$$
This is \eqref{jzfcojjh1}. 
\end{proof}

\begin{proof}[Proof of Theorem \ref{p.cvuep}] From Lemma \ref{l.jzfcojjh}, if 
$\displaystyle 
\ep^{\frac23}\delta^{-\frac13}+ \ep^{\frac13}\delta^{-\frac23} \leq c\lambda, $
then 
\begin{multline*}
\Big\{ \omega\in \Omega\ : \ \sup_{x\in B_T, \ t\in [0,T]} \left|u^\ep(x,t)-u(x,t)\right| \geq  \lambda \Big \} \subseteq \\  
\Big\{ \omega\in \Omega\ : \ \sup_{(y,p) \in B_{R} \times B_{L}} \left| \delta v^\delta (y,p) + \overline H(p)\right| \geq c\lambda \Big\},
\end{multline*}
where 
$R := C\left(\frac{1}{\lambda\ep}+\frac{1}{\delta}\right)$. Theorem \ref{p.deltaGlobal} implies that, if inequality
\begin{equation}\label{hboshdf}
\lambda \geq  C\delta^{\frac{1}{7+6q}} \left( 1+  |\log(\delta)|\right)
\end{equation}
holds, then 
$$
\Prob\bigg[\,\sup_{(y,p) \in B_{R} \times B_{L}} \left| \delta v^\delta (y ,p) + \overline H(p)\right| \geq c\lambda \bigg] 
\leq C \lambda^{-2\d}R^{\d} \left(\lambda^{-2\d}+\delta^{-2\d}\right)\exp\left(-\frac{\lambda^{4(1+q)}}{C \delta} \right).
$$
Note that inequalities $\displaystyle 
\ep^{\frac23}\delta^{-\frac13}+ \ep^{\frac13}\delta^{-\frac23} \leq c\lambda$ and~\eqref{hboshdf} are ensured by the conditions  
$$
\delta:=C \ep^{\frac12}\lambda^{-\frac32}
\qquad{\rm 
and}\qquad
\lambda \geq C \ep^{\frac{1}{17+12q}}\left(1+\log(\ep)\right)^{\frac12}.
$$
We deduce that
$$
\Prob\bigg[  \ \sup_{x\in B_T, \ t\in [0,T]} \left|u^\ep(x,t )-u(x,t)\right| \geq  \lambda \bigg] \leq  
C\lambda^{-3\d}\ep^{-\d}\left(\lambda^{-2\d}+\ep^{-\d}\right)\exp\left(-\frac{\lambda^{\frac{11}{2}+4q}}{C \ep^{\frac12}} \right).
$$
As $\ep\leq \lambda$, this implies that \eqref{e.cvuep} holds. 
\end{proof}

\subsection*{Acknowledgements}
 This work was partially supported by the ANR (Agence Nationale de la Recherche) through HJnet project ANR-12-BS01-0008-01. We thank Ilya Bogdanov for indicating a nice proof of Lemma~\ref{l.quanstrasz} to us on Math Overflow. 

\bibliographystyle{plain}
\bibliography{HJrates}

\end{document}